%% file: arxiv-v2.tex
\documentclass[reqno]{amsart}

\input{preamble}

\title[Universality of span 2-categories]{Universality of span 2-categories and\\ the construction of 6-functor formalisms}
\author{Bastiaan Cnossen}
\address{B.C.: Mathematisches Institut, Universität Regensburg, Universitätsstraße 31, 93053 Regensburg, Germany}
\author{Tobias Lenz}
\address{T.L.: Mathematisches Institut, Rheinische Friedrich-Wilhelms-Universität Bonn, Endenicher Allee 60, 53115 Bonn, Germany}
\author{Sil Linskens}
\address{S.L.: Mathematisches Institut, Universität Regensburg, Universitätsstraße 31, 93053 Regensburg, Germany}

\begin{document}
\begin{abstract}
	Given an $\infty$-category $\myC$ equipped with suitable wide subcategories $I, P \subset E\subset C$, we show that the $(\infty,2)$-category $\SpanTwoP{\myC}{E}{I}{P}$ of \emph{higher} (or \emph{iterated}) \emph{spans} defined by Haugseng has the universal property that 2-functors $\SpanTwoP{\myC}{E}{I}{P} \to \twoD$ correspond precisely to \emph{$(I, P)$-biadjointable} functors $\myC\catop \to \twoD$, i.e.\ functors $F$ where $F(i)$ for $i \in I$ admits a left adjoint and $F(p)$ for $p \in P$ admits a right adjoint satisfying various Beck--Chevalley conditions. We further extend this universality to the lax symmetric monoidal setting, providing a conceptual explanation for---and an independent proof of---the Mann--Liu--Zheng construction of 6-functor formalisms from suitable functors $C^\op\to\CAlg(\Cat)$.
\end{abstract}

\DeclareRobustCommand{\SkipTocEntry}[5]{}
\newcommand{\starsubsec}[1]{\addtocontents{toc}{\SkipTocEntry}\subsection*{#1}}

\begingroup
\parskip=0pt
\maketitle
\setcounter{tocdepth}{2}
\makeatletter
\newif\ifhe@d
\def\@tocline#1#2#3#4#5#6#7{\begingroup\relax%
	\he@dfalse\ifcase#1\relax\or\he@dtrue\fi%
	\ifnum#1<3\ifhe@d\else\hspace{2em}\hspace{-2pt}\fi#6\hfill#7\par\fi\endgroup}
\makeatother
\tableofcontents
\endgroup

\section{Introduction}
The formalism of \emph{six operations} was originally developed by Grothendieck \cite{SGA4} to capture the fundamental operations ($f^*, f_*, f_!, f^!, \otimes, \iHom$) on étale cohomology and their intricate relationships (adjunctions, base change, projection formul\ae). Grothendieck's ideas have proven remarkably versatile, and analogous axiomatic structures appear in many different settings across pure mathematics, for example for sheaves on locally compact topological spaces \cite{Verdier,Iversen,Kashiwara_Schapira,volpe-six-ops}, $D$-modules \cite{Kashiwara_Schapira, GR2017studyDAG}, perverse sheaves \cite{BBD_Perverse,Kashiwara_Schapira}, ind-coherent sheaves \cite{GR2017studyDAG}, as well as in condensed mathematics \cite{ScholzeSixFunctors}, rigid-analytic geometry \cite{Mann2022SixFunctor}, or motivic homotopy theory \cite{Ayoub2007SixOperations, CisinskiDeglise2019Triangulated, equivariant-motivic-six-ops}.

In the modern language of $\infty$-categories, the structure encoded by a 6-functor formalism can be elegantly captured using $\infty$-categories of \textit{spans}  (also known as \textit{correspondences}). This perspective goes back to foundational work by Liu--Zheng \cite{Liu_Zheng_Gluing, Liu_Zheng_Enhanced} and Gaitsgory--Rozenblyum \cite{GR2017studyDAG}, and has been effectively formulated by Mann \cite{Mann2022SixFunctor} and Scholze \cite{ScholzeSixFunctors}. Specifically, given a base $\infty$-category $\myC$ of geometric objects and a collection of morphisms $E$ of $\myC$ closed under composition and base change, one considers the $\infty$-category $\mathrm{Span}(\myC, E)$ which has the same objects as $\myC$, while its morphisms from $x$ to $y$ are spans $x \leftarrow u \rightarrow y$ where the right-pointing map $u \to y$ belongs to $E$; composition is defined via pullback. When $\myC$ admits finite products, $\mathrm{Span}(\myC, E)$ naturally inherits a symmetric monoidal structure. With this setup, the functors $(f^*, f_!, \otimes)$ and their coherences can be encoded as a lax symmetric monoidal functor $\Dd \colon \mathrm{Span}(\myC, E) \to \Cat_{\infty}$, termed a \textit{3-functor formalism}; the remaining operations $(f_*, f^!, \iHom)$ are then obtained as right adjoints, i.e.~they are \emph{properties} as opposed to further structure.

Such 3-functor formalisms encode a wealth of structure in a highly coherent fashion, making them a powerful tool---at the same time, however, this means that constructing such a formalism $\Dd$ is often quite difficult. In practice, it is usually much easier to construct a functor $\Dd_0\colon \myC\catop \to \CMon(\Cat_{\infty})$ that encodes the pullback functors $f^*$ and the tensor products $\otimes$, and it is therefore natural to look for conditions ensuring $\Dd_0$ extends in some preferred way to $\mathrm{Span}(\myC, E)$. One such extension result was proven by Mann \cite[Proposition~A.5.10]{Mann2022SixFunctor}, building heavily on intricate $\infty$-categorical machinery developed by Liu and Zheng \cite{Liu_Zheng_Enhanced, Liu_Zheng_Gluing}; see also \cite{chowd1, chowd2, chowd3} for a recent reworking. The Mann--Liu--Zheng approach requires the specification of a so-called \emph{suitable decomposition} of $E$, i.e.~ two classes of morphisms, $I$ (thought of as `local isomorphisms') and $P$ (thought of as `proper maps'), satisfying some axioms (see Theorem~\ref{thm:IntroMainTheorem} below for details), including the requirement that every morphism in $E$ factors as a composite $p \circ i$ with $i \in I$ and $p \in P$. Assuming then that the functors $i^*$ for $i \in I$ admit left adjoints $i_\sharp$, the functors $p^*$ for $p \in P$ admit right adjoints $p_*$, and these adjoints satisfy appropriate base change equivalences (including a `mixed' condition $i_{\sharp} p'_* \simeq p_* i'_{\sharp}$) and projection formul\ae, Mann shows that one can extend $\Dd_0$ to a 3-functor formalism $\Dd\colon\Span(C,E)\to\Cat_\infty$, such that the exceptional pushforward $f_!$ is given by the \emph{left} adjoint to $f^*$ if $f\in I$, and by the \emph{right} adjoint if $f\in P$.

An alternative approach, initiated by Gaitsgory and Rozenblyum \cite{GR2017studyDAG}, focuses on the universal properties of certain $(\infty, 2)$-categories built from spans. Here we consider 2-functors $F\colon \myC\catop \to \twoD$, where $\twoD$ is an arbitrary symmetric monoidal $(\infty, 2)$-category taking the place of $\Cat_{\infty}$. The conditions on $\Dd_0$ considered previously may be formulated in this generality; in particular, we say that $F$ is \textit{left $I$-adjointable} if the pullback functor $i^* \coloneqq F(i)$ for $i \in I$ admits a left adjoint $i_{\sharp}$ satisfying the Beck--Chevalley condition, and dually we say that $F$ is \textit{right $P$-adjointable} if the functor $p^*$ for $p \in P$ admits a right adjoint $p_*$ satisfying the Beck--Chevalley condition. Gaitsgory and Rozenblyum conjectured that the universal target for left $I$-adjointable functors is a certain $(\infty, 2)$-categorical enhancement $\smash{\SpanThreeHalves(\myC,I)}$ of $\Span(\myC,I)$, whose 2-morphisms are maps
\[
	\begin{tikzcd}[row sep=scriptsize]
		& u\arrow[dl]\arrow[dd]\arrow[dr]\\
		x && y\\
		& \arrow[ul]v\arrow[ur]
	\end{tikzcd}
\]
between spans; a dual statement then follows for right $P$-adjointability. Rigorous proofs of this universal property were later provided independently by Macpherson \cite{MacPherson2022Bivariant} and Stefanich \cite{Stefanich2020Correspondences}; moreover, the former also provided a (lax) symmetric monoidal version of this universal property. Throwing away non-invertible 2-cells, this then gives an alternative proof of Mann's result in the special case where either $I$ or $P$ only consists of the equivalences in $\myC$. Gaitsgory and Rozenblyum also outlined the construction of suitable $(\infty,2)$-functors more generally in the presence of a decomposition $E=P\circ I$ satisfying similar, but more restrictive assumptions than in Mann's result.

\starsubsec{Main result}
In this article, we establish an $(\infty,2)$-categorical universal property in the full generality of Mann's result, providing a unified and conceptually streamlined perspective. More concretely, given an $\infty$-category $\myC$ with a wide subcategory $E\subset \myC$ closed under base change and a suitable decomposition $I,P\subset E$, we introduce a different $(\infty, 2)$-categorical enhancement $\SpanTwoP{\myC}{E}{I}{P}$ of $\Span(\myC,E)$ as a minor generalization of a construction due to Haugseng \cite{HaugsengSpans}; this time, the 2-morphisms are given by \emph{spans of spans}, i.e.~commutative diagrams
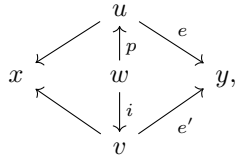
\[
\begin{tikzcd}[row sep=scriptsize]
	& u  \ar{dl}\ar{dr}{e} \\
	x & w \arrow[u,"p"'{yshift=-2pt}] \arrow[d,"i"{yshift=1pt}] & y\rlap, \\
	& v \ar{ul}\ar[swap]{ur}{e'}
\end{tikzcd}
\]
where the vertical maps $p$ and $i$ are required to lie in $P$ and $I$, respectively. We introduce the notion of an \textit{$(I, P)$-biadjointable functor} $F\colon \myC\catop \to \twoD$ as one that is both left $I$-adjointable and right $P$-adjointable such that the left and right adjoints interact via a \emph{double Beck--Chevalley equivalence} $i_{\sharp} p'_* \simeq p_* i'_{\sharp}$. Our main result establishes the universality of the inclusion $\myC\catop \hookrightarrow \SpanTwoP{\myC}{E}{I}{P}$ among $(I,P)$-biadjointable functors:

\begin{thmx}[See Theorem~\ref{thm:main-theorem}]
	\label{thm:IntroMainTheorem}
	Let $\myC$ be an $\infty$-category, and let $I, P \subset E \subset \myC$ be wide subcategories satisfying the following conditions:
	\begin{enumerate}
		\item $I$ and $P$ are closed under base change and are left cancellable (i.e.\ if $gf,g \in I$, then $f \in I$, and similarly for $P$).
		\item Every morphism in $E$ factors as $p \circ i$ for some $i \in I$ and $p \in P$.
		\item Every map in $I \cap P$ is truncated (i.e.\ $n$-truncated for some $n \geq -2$ which may depend on the morphism).
	\end{enumerate}
	Then the inclusion $h\colon \myC\catop \hookrightarrow \SpanTwoP{\myC}{E}{I}{P}$ is the initial $(I,P)$-biadjointable functor, i.e.~for every $(\infty,2)$-category $\twoD$ restriction along $h$ induces an equivalence
	\[
	h^*\colon \FUN(\SpanTwoP{\myC}{E}{I}{P}, \twoD) \iso \FUN_{(I,P)\dbadj}(\myC\catop, \twoD)
	\]
	of $(\infty, 2)$-categories, where the right-hand side denotes the locally full subcategory of $\FUN(\myC^\op,\twoD)$ spanned by the $(I,P)$-biadjointable functors and those natural transformations that satisfy the evident Beck--Chevalley conditions with respect to maps in $I$ and $P$.
\end{thmx}

Informally, the unique extension $\overline{F}\colon \SpanTwoP{\myC}{E}{I}{P} \to \twoD$ of an $(I,P)$-biadjointable functor $F\colon \myC^{\op} \to \twoD$ may be described as follows: On objects, we have $\overline{F}(x) = F(x)$. A span $x \xleftarrow{\;\smash{f}\;} u \xrightarrow{\;e\;} y$ is sent to the composite $e_!f^*\colon F(x) \to F(y)$, where $f^* \coloneqq F(f)$; the `exceptional pushforward' $e_!$ is obtained by choosing a decomposition $e = p \circ i$ into a map in $I$ followed by a map in $P$ and setting $e_! = p_*i_{\sharp}$, where $i_{\sharp}$ is a left adjoint to $i^*$ and $p_*$ is a right adjoint to $p^*$. The effect of $\overline{F}$ on 2-morphisms is constructed using the units and counits for the adjunctions $i_{\sharp} \dashv i^*$ and $p^* \dashv p_*$.

As a special case of the theorem, we may take $P$ to consist of just the equivalences of $\myC$, in which case we recover the universal property of $\smash\SpanThreeHalves$ recalled above. As the other extreme we can consider the fully biadjointable case $I=P=E=\myC$, for which the universal property of $\SpanTwo(C)$ has been conjectured by Ben-Moshe \cite[Conjecture~2.25]{Ben-Moshe_Transchromatic} based on \cite[Remark~4.2.5]{hopkinsLurie2013ambidexterity}. More generally, if $I=P=E$, then a functor $C\catop \to \twoD$ is biadjointable if and only if it is both left and right $E$-adjointable, and a certain inductively defined \emph{norm map} $\Nm_f\colon f_{\sharp}\to f_\ast$ \cite{hopkinsLurie2013ambidexterity} is an equivalence for every $f\in E$; if $\twoD$ is the $(\infty,2)$-category of $\infty$-categories, this is typically referred to as \emph{ambidexterity}. In this setting, Balmer and Dell'Ambrogio \cite{BalmerDellAmbrogio} proved the $(2,2)$-categorical analogue of Theorem~\ref{thm:IntroMainTheorem} under an additional faithfulness hypothesis on the maps in $E$. The general case proven here has been conjectured by Hoyois \cite[Expected~Theorem~2.7]{Hoyois_YTM}.

\starsubsec{Construction of 6-functor formalisms}
If $\myC$ admits finite products, or more generally if it comes with a symmetric monoidal structure suitably interacting with pullbacks and the classes $I$ and $P$, then $\SpanTwoP{\myC}{E}{I}{P}$ inherits a symmetric monoidal structure, and the functor $h$ becomes universal among \textit{symmetric monoidal} functors whose underlying functor is $(I, P)$-biadjointable, see Proposition~\ref{prop:sym-mon-univ-prop}. In fact, we show that it is also universal among a suitable notion of \textit{lax symmetric monoidal $(I, P)$-biadjointable functors} (Definition~\ref{def:lax-badj}). When $\myC$ is equipped with the cartesian symmetric monoidal structure, this result takes the following form:

\begin{thmx}[See Theorem~\ref{thm:lax-univ-prop}]\label{thm:introlax}
	Let $I,P\subset E\subset \myC$ as before, and assume that $\myC$ has finite products. Then the restriction
	\[
		h^*\colon \FUN^\textup{lax-$\otimes$}\big(\SpanTwoP{\myC}{E}{I}{P}^\otimes,\twoD^\otimes\big)\to\FUN^\textup{lax-$\otimes$}\big((\myC^\times)^\op,\twoD^\otimes\big)
	\]
	is a locally full inclusion for any symmetric monoidal $(\infty,2)$-category $\twoD^\otimes$. Moreover, the objects in its image are precisely the lax symmetric monoidal $(I,P)$-biadjointable functors, and its 1-morphisms are precisely those natural transformations of lax symmetric monoidal 2-functors whose underlying natural transformations of functors $\myC^\op\rightrightarrows\twoD$ are $(I,P)$-biadjointable.
\end{thmx}
The datum of an $(I,P)$-biadjointable lax symmetric monoidal functor $(\myC^\times)^\op\to\twoD^\otimes$ can also be encoded as a functor $\myC^\op\to\CAlg(\twoD^\otimes)$ satisfying suitable base change and projection formul\ae. Thus, applying Theorem~B to $\twoD = \CAT_{\infty}$ with the cartesian symmetric monoidal structure, we in particular obtain a 2-categorical enhancement of Mann's result on the construction of 6-functor formalisms; we record this as Theorem~\ref{thm:6FF}. Discarding the non-invertible 2-cells, this also gives an independent proof of Mann's original result, whose proof built on work of Liu--Zheng. Compared to the work of Liu--Zheng, which crucially uses the quasicategorical model of $\infty$-categories as simplicial sets satisfying certain horn-filling conditions, our arguments are model-independent: the desired extensions are constructed by techniques of a purely categorical nature (most notably Kan extensions and by computing $(\infty,1)$-categorical localizations), avoiding any explicit recourse to a pointset model and in particular to the combinatorics of the simplex category.

The universal properties also show that the $(\infty,2)$-categorical extensions we construct are unique. A similar uniqueness result for Mann's construction has recently been obtained by Kuijper \cite{Kuijper_Axiomatixation_6FF} and Dauser--Kuijper \cite{Dauser_Kuijper_Uniqueness} using the machinery of Liu--Zheng: under the stronger assumption that all morphisms in $I$ and $P$ are truncated, they show that the lax symmetric monoidal \textit{1-functors} $\Span(\myC,E) \to \Cat_{\infty}$ arising from Mann's construction can be uniquely characterized using Scholze's notions of \emph{cohomologically proper} and \emph{cohomologically étale} maps. Combined with our result, this means that in this truncated setting the additional adjunction data encoded in the 2-functor is essentially unique whenever it exists. It also shows that in this case our construction forgets to the Mann--Liu--Zheng construction.

\starsubsec{Lax transformations of 6-functor formalisms}
By Theorems \ref{thm:IntroMainTheorem} and \ref{thm:introlax}, every (lax symmetric monoidal) $(I,P)$-biadjointable natural transformation $\phi\colon F \to G$ between $(I,P)$-biadjointable functors $\myC\catop \to \twoD$ uniquely extends to a (lax symmetric monoidal) natural transformation between their extensions to the span 2-category. In particular, $(I,P)$-biadjointable natural transformations between lax symmetric monoidal functors $\myC\catop\to \CAT_\infty$ extend to natural transformations of six functor formalisms, as was proven previously by Mann. In practice, this turns out to be a strong requirement on $\phi$: there exist various interesting examples of transformations $\phi$ which are right $P$-adjointable, but not left $I$-adjointable, or vice versa. In the former case, for example, the Beck--Chevalley maps $i_{\sharp}\phi \to \phi i_{\sharp}$ for maps $i$ in $I$ are not necessarily invertible, so no extension to a natural transformation of six functor formalisms can be expected. It is now natural to wonder whether $\phi$ can still be uniquely extended to a \emph{lax} transformation, with the Beck-Chevalley maps serving as the lax naturality cells.

The $(\infty,2)$-categorical generality of Theorems \ref{thm:IntroMainTheorem} and \ref{thm:introlax} provides a simple way to address questions like these. A natural transformation $\phi$ may equivalently be described as a functor from $C\catop$ into the arrow 2-category $\Ar(\twoD)$ of $\twoD$. After including $\Ar(\twoD)$ into its \emph{lax} version $\Ar^{\lax}(\twoD)$, one can show that a right $P$-adjointable transformation $\phi$ automatically becomes $(I,P)$-biadjointable. By Theorem \ref{thm:IntroMainTheorem}, it then uniquely extends to a 2-functor from $\SpanTwoP{C}{E}{I}{P}$ to $\Ar^{\lax}(\twoD)$, which amounts to a lax transformation between the extensions of $F$ and $G$ to $\SpanTwoP{C}{E}{I}{P}$. Similarly, if the transformation $\phi$ is assumed to be left $I$-adjointable, it uniquely extends to an \emph{oplax} natural transformation. While it seems plausible that results like these can also be proved in a more ad hoc way using the combinatorial methods of Liu--Zheng, they become easy corollaries of the general $(\infty,2)$-categorical result.

In fact, the same idea allows us to prove a stronger result: we can completely characterize \emph{all} (op)lax natural transformations between 2-functors out of $\SpanTwoP{\myC}{E}{I}{P}$ in terms of (op)lax natural transformations out of $\myC\catop$. To formulate this, we consider for a pair of $(\infty,2)$-categories $\twoC$ and $\twoD$ the $(\infty,2)$-category $\FUN(\twoC,\twoD)^{\mathrm{(op)}\lax}$ whose objects, morphisms, and 2-morphisms are 2-functors from $\twoC$ to $\twoD$, (op)lax natural transformations, and modifications; see Definition \ref{def:FunLax}. We then have:

\begin{thmx}[See~Theorem~\ref{thm:lax_nat_trans_of_4FF}]\label{thm:IntroLaxTrans}
	Let $I,P\subset E\subset C$ as before. For every $(\infty,2)$-category $\twoD$, the restriction
	\[
	\FUN(\SpanTwoP{C}{E}{I}{P},\twoD)^{\lax}\to\FUN(C^\op,\twoD)^{\lax}
	\]
	is a locally full inclusion, whose image consists of the $(I,P)$-biadjointable functors and those lax natural transformations $\phi\colon F \to G$ such that the maps
	\[
		i^* \phi \to \phi i^* \qquadtext{ and } \phi p_* \to p_* \phi
	\]
	are equivalences for all $i \in I$ and $p \in P$. Dually, the restriction functor
	\[
	\FUN(\SpanTwoP{C}{E}{I}{P},\twoD)^{\oplax}\to\FUN(C^\op,\twoD)^{\oplax}
	\]
	is a locally full inclusion, whose image consists of the $(I,P)$-biadjointable functors and those oplax natural transformations $\psi\colon F \to G$ for which the maps
	\[
		i_{\sharp} \psi \to \psi i_{\sharp} \qquadtext{ and } \psi p^* \to p^* \psi
	\]
	are equivalences for all $i \in I$ and $p \in P$.
\end{thmx}

We refer to Theorem~\ref{thm:lax_nat_trans_of_4FF} for more precise formulations of the conditions on $\phi$ and $\psi$. We also prove a lax symmetric monoidal version of this result, for which we refer to Theorem \ref{thm:lax_nat_trans_of_6FF}. By applying the latter to the cartesian monoidal structure on $\myC$, this in particular provides the construction of (op)lax natural transformations between 6-functor formalisms discussed before.

\starsubsec{Strategy and outline}
Our proof for Theorem \ref{thm:IntroMainTheorem} relies on a recognition principle (Theorem \ref{thm:Recognizing_The_Universal_Biadjointable_Functor}) which reduces the universal property of $\SpanTwoP{\myC}{E}{I}{P}$ to universal properties of its $\HOM$-functors. More specifically, we show that an $(I,P)$-biadjointable functor $h\colon \myC\catop \to \twoD$ is universal if and only if it is essentially surjective and for every $a \in \myC$, the functor
\[
	\HOM_{\twoD}\big(h(a), h(-)\big)\colon \myC\catop \to \Cat_{\infty}
\]
is the \textit{free} $(I, P)$-biadjointable functor generated by the identity morphism $\id_{h(a)}$, i.e.~it corepresents evaluation at $a$. Identifying these free functors then constitutes the technical heart of our argument.

For this, we more generally explain how to construct for any functor $\Xx\colon \myC\catop \to \Cat_{\infty}$ an $(I,P)$-biadjointable functor $\CIP{\Xx}{E}{I}{P}\colon \myC\catop \to \Cat_{\infty}$ which is freely generated by $\Xx$ (Theorem \ref{thm:biadd}). The proof of this universal property crucially uses the framework of \emph{parametrized higher category theory} in the sense of \cite{BDGNS2016ExposeI,martiniwolf2021limits}, and proceeds in stages: The case where $\Xx$ is terminal and where $P$ is contained in $I$ was addressed in our previous work \cite{CLL_Spans} via filtering $P$ by the subcategories $P_{\le n}$ of $n$-truncated morphisms and inducting over $n$. We will explain how this generalizes to arbitrary $\Xx$ using a parametrized Kan extension trick, adapting an argument due to Bachmann--Hoyois, who described free non-parametrized semiadditive $\infty$-categories as certain span categories. The further generalization to the case where $P$ is not necessarily contained in $I$ requires a new idea in the form of an intricate localization argument (see Theorem~\ref{thm:P-completion-SpanI}) that allows us to build free $(I,P)$-biadjointable functors from free $(I,I\cap P)$- and $(\core \myC,P)$-biadjointable ones. This argument in turn leverages the weak contractibility of the category of factorizations of a morphism $e\in E$ into a map in $I$ followed by a map in $P$, which was previously established by Liu--Zheng, and for which we give a simple independent proof (Proposition~\ref{prop:Decomp_Weakly_Contractible}).

We now provide a linear overview of our paper:

In Section~\ref{sec:badj} we introduce $(I,P)$-biadjointable functors and prove the existence of universal biadjointable functors (Theorem~\ref{thm:existence}). While this existence proof is constructive to a certain extent, the resulting description is completely intractable. We therefore prove the aforementioned recognition principle (Theorem~\ref{thm:Recognizing_The_Universal_Biadjointable_Functor}), which reduces understanding the universal biadjointable functor to understanding free $(I,P)$-biadjointable functors into the $(\infty,2)$-category of $(\infty,1)$-categories.

These free functors are then the subject of Section~\ref{sec:Free_Semiadditive}, where we define the $(I,P)$-biadjointable functor $\CIP{\Xx}{E}{I}{P}$ for any $\Xx\colon \myC^\op\to\Cat_\infty$ and prove its universal property following the outline above.

In Section~\ref{sec:main-results} we then introduce the $(\infty,2)$-categories $\SpanTwoP{\myC}{E}{I}{P}$ (incarnated as 2-fold Segal spaces), and explicitly describe their $\HOM$ functors. Combining this with the results from the previous two sections, we then obtain the proof of Theorem~\ref{thm:IntroMainTheorem}. 

The other two main results now follow as formal consequences. In Section~\ref{sec:Lax_Monoidal_Universality}, we show that the Grothendieck construction of a diagram of span 2-categories is itself a span 2-category, and use this fact to reduce Theorem~\ref{thm:introlax} to Theorem~\ref{thm:IntroMainTheorem}. In Section~\ref{sec:lax_transformations_6FF}, we explain how basic facts about adjunctions in (op)lax functor categories allow us to deduce Theorem~\ref{thm:IntroLaxTrans} from Theorem~\ref{thm:IntroMainTheorem}.

\starsubsec{Acknowledgements}
We are grateful to Maxime Ramzi, Shay Ben-Moshe, and Lior Yanovski for several helpful discussions related to this project. We especially thank Shay Ben-Moshe and later Marc Hoyois for drawing our attention to the problem of establishing the universal property for span 2-categories. It is a pleasure to thank Shay Ben-Moshe, Rune Haugseng, and Marc Hoyois for feedback on an earlier version of this article.

All three authors would like to thank the Isaac Newton Institute for Mathematical Sciences, Cambridge, for support and hospitality during the programme \emph{Equivariant homotopy theory in context}, where work on this paper was undertaken. This work was supported by EPSRC grant EP/Z000580/1.

B.C.\ and S.L.\ are associate members of the SFB 1085 Higher Invariants. T.L.\ is an associate member of the Hausdorff Center for Mathematics at the
University of Bonn (DFG GZ 2047/1, project ID 390685813).

\starsubsec{Conventions}  We will refer to $(\infty,1)$-categories simply as \emph{categories} (or \emph{1-categories} for emphasis) and to $(\infty,n)$-categories for $n>1$ simply as \emph{n-categories}; similarly, we will write $\Cat$ and $\Cat_n$ for the resulting ($\infty$-)categories. Throughout, we will distinguish $n$-categorical constructions from 1-categorical ones by using small caps ($\SpanTwo$, $\FUN$, $\HOM,\dots$) or blackboard bold letters; in particular generic $n$-categories (or double categories) will be denoted $\twoC,\twoD,\dots$ and $\CAT$ will denote the 2-category of 1-categories. We write $\core C$ for the groupoid core of a category $C$, and write $\iota_1 \twoC$ for the underlying category of an $n$-category $\twoC$. We will never need to commit to a specific model of $n$-categories apart from the construction of $\SpanTwoP{C}{E}{I}{P}$ in §\ref{subsec:span2}; the reader is free to choose their favourite model amongst \cite{lurieGoodwillie,GepnerHaugseng2015Enriched,haugseng-comparison}.

\section{Universal biadjointable functors}\label{sec:badj}
In this section, we introduce the notion of a (contravariant) \textit{biadjointable functor} out of a category $\myC$ with respect to suitably chosen subcategories $I$ and $P$, and show that there is always a universal such functor, denoted $h\colon \myC\catop \to \Univ{\myC}{I}{P}$.

\subsection{Biadjointability}
We start by recalling the notions of adjointability and biadjointability.

\begin{definition}[Adjointable squares]\label{def:adj_squares}
	Let $\twoD$ be a 2-category, and consider a commutative square
	\begin{equation}\label{eq:gen_sq}
	\begin{tikzcd}
		X' \rar{g^*} \dar[swap]{f^*} & X \dar{k^*} \\
		Y' \rar[swap]{h^*} & Y
	\end{tikzcd}
	\end{equation}
	in $\twoD$.
	\begin{enumerate}
		\item The square is called \textit{vertically left adjointable} if the morphisms $f^*$ and $k^*$ admit left adjoints $f_{\sharp}$ and $k_{\sharp}$ in $\twoD$ and the Beck--Chevalley map $\BC_{\sharp}\colon k_{\sharp}h^* \to g^*f_{\sharp}$, i.e.\ the 2-morphism
		\[
		k_{\sharp}h^* \xrightarrow{\eta_f} k_{\sharp}h^*f^*f_{\sharp} = k_{\sharp}k^*g^*f_{\sharp} \xrightarrow{\epsilon_k} g^*f_{\sharp},
		\]
		is an equivalence in the hom-category $\HOM_{\twoD}(Y',X)$.
		\item Dually we say that the square is \textit{vertically right adjointable} if it is left adjointable in $\twoD^{\co}$, i.e.\ $f^*$ and $k^*$ admit \textit{right} adjoints $f_*$ and $k_*$ and the Beck--Chevalley map $\BC_* \colon g^*f_* \to k_*h^*$ is an equivalence in $\HOM_{\twoD}(Y',X)$.
		\item The notions of \emph{horizontally left/right adjointable squares} are defined analogously.
		\item \label{it:Biadjointable} The square is called \textit{biadjointable} if it is horizontally left adjointable, and the resulting commutative square
		\[
		\begin{tikzcd}
			X' \dar[swap]{f^*}\arrow[dr, Rightarrow, "\BC_{{\sharp}}^{-1}"{description}] &[4pt]\arrow[l,"g_{\sharp}"'] X \dar{k^*} \\[4pt]
			Y' &\arrow[l,"h_{\sharp}"] Y
		\end{tikzcd}
		\]
		is vertically \textit{right} adjointable. Given a general horizontally left adjointable square $(\ref{eq:gen_sq})$, we will refer to the Beck--Chevalley map $\BC_{{\sharp},*}\colon g_{\sharp}k_*\to f_*h_{\sharp}$ of the induced square above as the \textit{double Beck--Chevalley map} of the original square. So the original square is biadjointable if it is horizontally left adjointable and the double Beck--Chevalley map is an equivalence.
	\end{enumerate}
\end{definition}

\begin{remark}
	There is an apparent asymmetry in the definition of biadjointable squares: we could have alternatively asked the original square to be vertically right adjointable and the resulting commutative square
	to be horizontally left adjointable. By \cite[Lemmas~F.10 and F.11]{Cnossen2023PhD} this condition is equivalent to \eqref{it:Biadjointable} and the resulting double Beck--Chevalley map $ g_{\sharp}k_*\to f_*h_{\sharp}$ agrees with $\BC_{{\sharp},*}$. In Proposition~\ref{prop:corep-badj} below we will also give an alternative characterization of biadjointable squares that is inherently symmetric.
\end{remark}

\begin{convention}\label{conv:indexing}
	We will always depict a functor $X\colon [1]\times[1]\to\twoD$ as
	\[
		\begin{tikzcd}
			X_{00}\arrow[r]\arrow[d] & X_{01}\arrow[d]\\
			X_{10}\arrow[r] & X_{11}
		\end{tikzcd}
	\]
	(i.e.~using the usual indexing convention for matrices). This way, we can unambiguously ask for $X\colon [1]\times[1]\to\twoD$ to be vertically or horizontally left or right adjointable, or ask for $X$ to be biadjointable. For example, horizontal left adjointability of $X$ in particular demands the existence of left adjoints to the maps $X_{00}\to X_{01}$ and $X_{10}\to X_{11}$.
\end{convention}

\begin{definition}[{cf.\ \cite[Definition~2.2.5]{ElmantoHaugseng2023Distributivity}}]
\label{def:Adjointable_Functor}
	Let $\myC$ be a category and let $I, P \subset \myC$ be two wide subcategories closed under base change. Let $\twoD$ be a 2-category, and consider a functor $F\colon \myC\catop \to \twoD$. For a morphism $f\colon X \to Y$ in $\myC$ we write $f^* \coloneqq F(f)\colon F(Y) \to F(X)$.
	\begin{enumerate}[(1)]
		\item We say that $F$ is \textit{left $I$-adjointable} if for every pullback square
		\[
		\begin{tikzcd}
			X' \dar[swap]{g}  \rar{j} \drar[pullback] & X \dar{f} \\
			Y' \rar[swap]{i} & Y
		\end{tikzcd}
		\]
		in $\myC$ with $i \in I$, the associated square
		\[
		\begin{tikzcd}
			F(Y) \dar[swap]{f^*} \rar{i^*} & F(Y') \dar{g^*} \\
			F(X) \rar[swap]{j^*} & F(X')
		\end{tikzcd}
		\]
		in $\twoD$ is horizontally left adjointable. In particular, $i^*$ and $j^*$ admit left adjoints $i_{\sharp}$ and $j_{\sharp}$.
		\item Dually, we say that $F$ is \textit{right $P$-adjointable} if for every pullback square
		\[
		\begin{tikzcd}
			X' \dar[swap]{q}  \rar{g} \drar[pullback] & X \dar{p} \\
			Y' \rar[swap]{f} & Y
		\end{tikzcd}
		\]
		in $\myC$ with $p \in P$, the associated square
		\[
		\begin{tikzcd}
			F(Y) \dar[swap]{p^*} \rar{f^*} & F(Y') \dar{q^*} \\
			F(X) \rar[swap]{g^*} & F(X')
		\end{tikzcd}
		\]
		in $\twoD$ is vertically right adjointable. In particular, $p^*$ and $q^*$ admit right adjoints $p_*$ and $q_*$.
		\item We say that $F$ is \textit{$(I,P)$-biadjointable} if it is both left $I$-adjointable and right $P$-adjointable, and for every pullback square
		\[
		\begin{tikzcd}
			X' \dar[swap]{q}  \rar{j} \drar[pullback] & X \dar{p} \\
			Y' \rar[swap]{i} & Y
		\end{tikzcd}
		\]
		in $\myC$ with $i \in I$ and $p \in P$, the commutative square
		\[
		\begin{tikzcd}
			F(Y) \dar[swap]{p^*} \rar{i^*} & F(Y') \dar{q^*} \\
			F(X) \rar[swap]{j^*} & F(X')
		\end{tikzcd}
		\]
		is biadjointable.
	\end{enumerate}

    If $F,G\colon \myC\catop \to \twoD$ are left $I$-adjointable, we say that a natural transformation $\alpha\colon F \Rightarrow G$ is \textit{left $I$-adjointable} if for every morphism $i\colon X \to Y$ in $I$ the naturality square
    \[
    \begin{tikzcd}
        F(Y)\arrow[r,"i^*"]\arrow[d,"\alpha_Y"'] & F(X)\arrow[d,"\alpha_X"]\\
		G(Y)\arrow[r,"i^*"'] & G(X)
    \end{tikzcd}
    \]
    is horizontally left adjointable, i.e.~the Beck--Chevalley map $i_{\sharp}\alpha\to\alpha i_{\sharp}$ is invertible. We may similarly define what it means for $\alpha$ to be \textit{right $P$-adjointable}, and we say it is \textit{$(I,P)$-biadjointable} if it is both left $I$-adjointable and right $P$-adjointable. We denote by
    \[
    	\FUN_{(I,P)\dbadj}(\myC\catop,\twoD)
    \]
    the locally full sub-2-category\footnote{Local fullness means that the induced maps on $\HOM$-categories are fully faithful inclusions.} of $\FUN(\myC\catop,\twoD)$ spanned by the biadjointable functors and natural transformations, respectively. If $I = P$, we will also speak of \textit{$I$-biadjointability.}
\end{definition}

\begin{remark}
	Our terminology is inspired by \cite[Definition~2.2.5]{ElmantoHaugseng2023Distributivity}, except that we say `right $P$-adjointable' instead of `right $P$-coadjointable.' Alternative names for adjointable functors used in the literature are \textit{bivariant theory with base change} \cite{MacPherson2022Bivariant} or functors that \textit{satisfy the Beck--Chevalley condition} \cite{GR2017studyDAG}. If $\twoD=\CAT$ is the (very large) 2-category of 1-categories, the underlying 1-category of $\FUN_{(I,P)\dbadj}(\myC^\op,\CAT)$ also appears under the name $\mathrm{BCFun}(\myC,E,I,P)$ in \cite{Dauser_Kuijper_Uniqueness}.
\end{remark}

\begin{example}\label{ex:param}
	Let $\twoD = \CAT$ again. Functors $\myC^\op\to\CAT$ are also known as \emph{$\myC$-parametrized categories} \cite{BDGNS2016ExposeI}. In this setting, left and right adjointability of functors fits into the general framework of parametrized colimits and limits developed by Martini--Wolf \cite{martiniwolf2021limits}: a $\myC$-parametrized category  $\myC\catop \to \CAT$ is left  $I$-adjointable if and only if it is \textit{$I$-cocomplete}, and right $P$-adjointable if and only if it is \textit{$P$-complete}, see e.g.~\cite[Example~3.30]{LLP}.

	Similarly, if $Q=I=P$ is truncated and left cancellable (a so-called \emph{inductible subcategory}) then $Q$-biadjointable functors $\myC^\op\to\CAT$ have been studied under the name \emph{$Q$-semiadditive} $\myC$-parametrized categories in \cite{CLL_Spans}. We will revisit this connection in Section~\ref{sec:Free_Semiadditive}, where we will exploit parametrized techniques to understand hom categories in the target $\Univ{\myC}{I}{P}$ of the universal biadjointable functor.
\end{example}

\begin{example}
	If $P = \core\myC$ consists only of the equivalences in $\myC$, then the notion of $(I,P)$-biadjointability agrees with that of left $I$-adjointability; similarly, $(\core\myC,P)$-biadjointability recovers right $P$-adjointability. Accordingly, all of our results about biadjointability proven below will specialize to results about left or right adjointability.
\end{example}

\begin{remark}
	\label{rmk:Postcompose_With_2Functor}
	Any 2-functor $G\colon \twoD \to \twoE$ preserves (bi)adjointable squares. In particular, if $F\colon \myC\catop \to \twoD$ is $(I,P)$-biadjointable, then so is the composite $G \circ F\colon \myC\catop \to \twoE$.
\end{remark}

\subsection{The universal biadjointable functor}
We shall now move to a discussion of the \textit{universal} biadjointable functor out of $\myC\catop$, generalizing the discussion of \cite{ElmantoHaugseng2023Distributivity} in the left adjointable case.

\begin{definition}
	\label{def:Universal_Biadj_Functor}
	Let $\myC$ be a category and let $I,P\subset\myC$ be wide subcategories closed under base change. An $(I,P)$-biadjointable functor $h\colon \myC\catop \to \Univ{\myC}{I}{P}$ is called \textit{universal} if the restriction
	\[
		h^*\colon\FUN(\Univ{\myC}{I}{P},\twoD)\to\FUN_{(I,P)\dbadj}(\myC^\op,\twoD)
	\]
	(see Remark~\ref{rmk:Postcompose_With_2Functor}) is an equivalence of 2-categories for every 2-category $\twoD$.
\end{definition}

As the main result of this subsection we will show:

\begin{theorem}\label{thm:existence}
	Let $I,P\subset \myC$ be as above. Then there exists a universal $(I,P)$-biadjointable functor $h\colon\myC^\op\to\Univ{\myC}{I}{P}$.
\end{theorem}

The proof of Theorem~\ref{thm:existence} will rely on a representability result for (bi)adjointable squares. For this we first recall the \emph{walking adjunction} $\Adj$ from \cite{RiehlVerityAdj}. This is a 2-category coming with two morphisms $\ell\colon{\ominus}\to{\oplus}$ and $r\colon{\oplus}\to{\ominus}$ such that $\ell\dashv r$. The universal property of $\Adj$ from \cite{RiehlVerityAdj} as refined in \cite[Corollary~2.1.7]{ElmantoHaugseng2023Distributivity} can then be described as follows:

\begin{theorem}
	Let $\twoD$ be any 2-category. Then $\ell^*\colon\FUN(\Adj,\twoD)\to\FUN([1],\twoD)$ is a locally full inclusion. The objects in its image are precisely the left adjoint maps in $\twoD$, and the morphisms are precisely the right adjointable natural transformations, i.e.~those $\tau\colon F\Rightarrow G$ for which the naturality square
	\[
		\begin{tikzcd}
			F(0)\arrow[r, "F(0\to 1)"]\arrow[d,"\tau_0"'] &[1em] F(1)\arrow[d,"\tau_1"]\\
			G(0)\arrow[r,"G(0\to 1)"'] & G(1)
		\end{tikzcd}
	\]
	is horizontally right adjointable.

	Dually, $r^*\colon\FUN(\Adj,\twoD)\to\FUN([1],\twoD)$ identifies the source with the locally full 2-subcategory of right adjoint maps and left adjointable natural transformations.\qed
\end{theorem}

As a consequence of this, we can represent categories of (bi)adjointable squares:

\begin{proposition}\label{prop:corep-badj}
	Let $\twoD$ be any 2-category.
	\begin{enumerate}
		\item The 2-functor
		\begin{equation}\label{eq:idtimesr}
			(\id\times r)^*\colon\FUN([1]\times\Adj,\twoD)\to\FUN([1]\times[1],\twoD)
		\end{equation}
		is a locally full inclusion. The objects in its image are precisely the horizontally left adjointable squares
		\[
			\begin{tikzcd}
				X_{00}\arrow[r, "i"]\arrow[d] & X_{01}\arrow[d]\\
				X_{10}\arrow[r, "j"'] & X_{11}
			\end{tikzcd}
		\]
		and its morphisms are those natural transformations $\alpha$ that are horizontally left adjointable in the sense that the Beck--Chevalley maps $i_{\sharp}\alpha\to\alpha i_{\sharp}$ and $j_{\sharp}\alpha\to\alpha j_{\sharp}$ are invertible.
		\item Dually, the functor $(\ell\times\id)^*$ identifies $\FUN(\Adj\times[1],\twoD)$ with the locally full 2-subcategory of $\FUN([1]\times[1],\twoD)$ given by the vertically right adjointable squares and the vertically right adjointable natural transformations.
		\item Finally, the functor
		\[
			(\ell\times r)^*\colon\FUN(\Adj\times\Adj,\twoD)\to\FUN([1]\times[1],\twoD)
		\]
		identifies the source with the locally full 2-subcategory given by the biadjointable squares and those natural transformations that are both horizontally left and vertically right adjointable.
	\end{enumerate}
	\begin{proof}
		The first statement is immediate from the previous theorem by identifying $(\ref{eq:idtimesr})$ with
		\[
			\FUN\big([1],r^*\big)\colon\FUN\big([1],\FUN(\Adj,\twoD)\big)\to\FUN\big([1],\twoD\big).
		\]
		The proof of the second statement is dual.

		For the third statement, recall from \cite[Proposition~2.1.5]{ElmantoHaugseng2023Distributivity} and its proof that a map $\alpha\colon F\Rightarrow G$ in $\FUN([1],\twoD)$ admits a right adjoint if and only if
		\begin{equation}\label{diag:to-be-radjb}
			\begin{tikzcd}
				F(0)\arrow[r]\arrow[d,"\alpha_0"'] & F(1)\arrow[d, "\alpha_1"]\\
				G(0)\arrow[r] & G(1)
			\end{tikzcd}
		\end{equation}
		is vertically right adjointable; moreover, in this case, the right adjoint is given by the Beck--Chevalley transformation. Writing $\FUN^\text{ladj}([1],\twoD)$ for the image of $r^*\colon\FUN(\Adj,\twoD)\to\FUN([1],\twoD)$, it follows that a map $\alpha$ in $\FUN^\text{ladj}([1],\twoD)$ admits a right adjoint if and only if $(\ref{diag:to-be-radjb})$ is vertically right adjointable and the resulting square is again a map in $\FUN^\text{ladj}([1],\twoD)$, i.e.~horizontally left adjointable. This is precisely what it means for $(\ref{diag:to-be-radjb})$ to be biadjointable. In summary, we have shown that the objects in the image of the locally full inclusion
		\[
			\ell^*\colon\FUN\big(\Adj,\FUN^\text{ladj}([1],\twoD)\big)\to\FUN\big([1],\FUN^\text{ladj}([1],\twoD)\big)\subset\FUN([1]\times[1],\twoD)
		\]
		are precisely the biadjointable squares. It remains to characterize the maps in its image as the vertically right and horizontally left adjointable transformations, which by part (2) amounts to saying that a natural transformation $\alpha\colon F\to G$ of functors $[1]\rightrightarrows\FUN^\text{ladj}([1],\twoD)$ is right adjointable if and only if the induced transformations $\ev_0F\to\ev_0 G$ and $\ev_1F\to\ev_1 G$ are so. This is immediate from the pointwise description of right adjoints in a functor category recalled above.
	\end{proof}
\end{proposition}

\begin{proof}[Proof of Theorem~\ref{thm:existence}]
	We need to show that the 2-functor
	\[
		\FUN_{(I,P)\dbadj}(\myC\catop,-)\colon \CatTwo \to \CatTwo
	\]
	is corepresentable. We will construct the representing 2-category $\Univ{\myC}{I}{P}$ in three steps via pushouts in $\CatTwo$, enforcing first horizontal left adjointability, then vertical right adjointability, and finally biadjointability. 
	
	Denote by $\mathrm{PB}_I$, $\mathrm{PB}_P$ and $\mathrm{PB}_{I,P}$ the collections of equivalence classes of pullback squares in $\myC$ whose horizontal morphisms belong to $I$, whose vertical morphisms belong to $P$, or both, respectively.
	
	As the first step, we consider the pushout
	\[
		\begin{tikzcd}
			\arrow[dr,phantom,"\ulcorner"{very near end,xshift=7pt}]\arrow[d]\coprod_{\mathrm{PB}_I}\big([1]\times[1]\big) \arrow[r,"\coprod (\id\times r)"] &[1.5em] \coprod_{\mathrm{PB}_I}\big([1]\times\Adj\big)\arrow[d]\\
			\myC\catop\arrow[r] & \UnivGen_I(\myC)
		\end{tikzcd}
	\]
	in $\CatTwo$, where the left-hand vertical map sends each square to its image in $\myC\catop$. Applying $\FUN(-,\twoD)$ to this yields a pullback in $\CatTwo$; the previous proposition therefore shows that the restriction $\FUN(\UnivGen_I(\myC),\twoD)\to\FUN(\myC\catop,\twoD)$ is a locally full inclusion, and that its image consists precisely of the left $I$-adjointable functors and natural transformations.

	In the second step, we dually form the pushout along
	\[
		\coprod(\ell\times\id)\colon\coprod\nolimits_{\mathrm{PB}_P}\big([1]\times[1]\big)\to \coprod\nolimits_{\mathrm{PB}_P}\big(\Adj\times[1]\big)
	\]
	to obtain a map $\UnivGen_I(\myC)\to\UnivGen_{I/P}(\myC)$
	such that the restriction $\FUN(\UnivGen_{I/P}(\myC),\twoD)\to\FUN(\myC\catop,\twoD)$ is a locally full inclusion, with essential image those functors and natural transformations that are both left $I$-adjointable and right $P$-adjointable.

	Finally, we form the pushout
	\[
		\begin{tikzcd}
			\arrow[dr,phantom,"\ulcorner"{very near end,xshift=7pt}]\arrow[d]\coprod_{\mathrm{PB}_{I,P}}\big([1]\times[1]\big) \arrow[r,"\coprod (\ell\times r)"] &[1.5em] \coprod_{\mathrm{PB}_{I,P}}\big(\Adj\times\Adj\big)\arrow[d]\\
			\UnivGen_{I/P}(\myC)\arrow[r] & \Univ{\myC}{I}{P}\rlap.
		\end{tikzcd}
	\]
	The third part of the previous proposition then implies by the same arguments as before that the composite $h\colon\myC\catop\to\UnivGen_I(\myC)\to\UnivGen_{I/P}(\myC)\to\Univ{\myC}{I}{P}$ is the universal $(I,P)$-biadjointable functor.
\end{proof}

\begin{definition}\label{def:BAdjTrip}
	We write $\BAdjTrip$ for the 2-category whose objects are triples $(\myC,I,P)$ as in Definition~\ref{def:Universal_Biadj_Functor}, whose morphisms $(\myC,I,P)\to(\myC',I',P')$ are those functors $F\colon\myC\to\myC'$ with $F(I)\subset I'$, $F(P)\subset P'$ that preserve pullbacks along $I$ and along $P$. More formally, $\BAdjTrip$ is the (non-full) subcategory of $\FUN(\Lambda^2_2,\Cat)$ spanned by those diagrams $I\to\myC\leftarrow P$ where the two maps are wide subcategory inclusions with $I$ and $P$ closed under base change, and whose morphisms are the commutative diagrams that preserve pullbacks along $I$ and along $P$.
\end{definition}

\begin{corollary}\label{cor:Univ_Functoriality}
	The assignment $(\myC,I,P) \mapsto \twoB_{I,P}(C)$ defines a 2-functor
	\[
		\twoB\colon \BAdjTrip \to \CatTwo.
	\]
\end{corollary}
\begin{proof}
	Sending $(\myC,I,P)$ to the copresheaf $\iota\FUN_{(I,P)\dbadj}(\myC\catop,-)\colon \CatTwo \to \Spc$ defines a 2-functor $\BAdjTrip \to \Fun(\CatTwo,\CatTwo)\catop$. By Theorem~\ref{thm:existence}, this factors through the contravariant Yoneda embedding $\CatTwo \hookrightarrow \Fun(\CatTwo,\Spc)\catop$.
\end{proof}

\subsection{A recognition principle for universal biadjointable functors}
Our eventual goal is to show that the 2-categories $\Univ{\myC}{I}{P}$ from Theorem~\ref{thm:existence} are, in many cases, given by 2-categories of (iterated) spans. As the first step towards this goal, we will provide a convenient criterion for recognizing that a given biadjointable functor is universal. This relies on the notion of \textit{free} biadjointable functors:

\begin{definition}
	\label{def:Free_Adjointable_Functor}
	Let $a\in\myC$ be arbitrary. We call an $(I,P)$-biadjointable functor $F\colon \myC\catop \to \CAT$ \textit{free on a class $1_a\in F(a)$} if for every other $(I,P)$-biadjointable functor $G\colon \myC\catop \to \CAT$ evaluation at $1_a$ induces an equivalence of categories
    \[
        \HOM_{\FUN_{(I,P)\dbadj}(\myC\catop,\CAT)}(F,G) \iso G(a).
    \]
\end{definition}

\begin{proposition}
\label{prop:Hom_Cats_Free_BC_Functor}
    For every object $a \in \myC$, the identity morphism $\id_{h(a)}$ exhibits the functor
    \[
        \HOM_{\Univ{\myC}{I}{P}}(h(a),h(-)) \colon \myC\catop \to \CAT
    \]
    as a free $(I,P)$-biadjointable functor.
\end{proposition}
\begin{proof}
    The fact that it is an $(I,P)$-biadjointable functor is immediate from Remark~\ref{rmk:Postcompose_With_2Functor} since it may be written as a composite
    \[
        \myC\catop \xrightarrow{\;h\;} \Univ{\myC}{I}{P} \xrightarrow{\;\HOM(h(a),-)\;} \CAT,
    \]
    where the first functor is $(I,P)$-biadjointable and the second is a 2-functor. To show it is also free on $a$, consider another $(I,P)$-biadjointable functor $G\colon \myC\catop \to \CAT$. By the universal property of $\UnivGen\coloneqq\Univ{\myC}{I}{P}$, we may uniquely extend $G$ to a 2-functor $\overline{G}\colon \UnivGen\to \CAT$. Consider now the following commutative diagram:
    \[
    \begin{tikzcd}
        \HOM_{\FUN(\UnivGen,\CAT)}\big(\HOM(h(a),-),\overline{G}\big) \dar{h^*}[swap]{\sim} \rar{\ev_{\id_{h(a)}}} &[.5em] \overline{G}(h(a)) \dar[equal] \\
        \HOM_{\FUN_{(I,P)\dbadj}(\myC\catop,\CAT)}\big(\HOM(h(a),h(-)),G\big) \rar[swap]{\ev_{\id_{h(a)}}} & G(a).
    \end{tikzcd}
    \]
    The left vertical map is obtained by passing to $\HOM$-categories from the equivalence \[h^*\colon \FUN(\UnivGen,\CAT) \iso \FUN_{(I,P)\dbadj}(\myC^\op,\CAT),\] hence is itself an equivalence. The top map is an equivalence by the 2-categorical Yoneda lemma. It follows that also the bottom map is an equivalence, as desired.
\end{proof}

So far, we have described the mapping categories in $\UnivGen$ between objects in the essential image of $h\colon\myC^\op\to\UnivGen$. By the following simple observation, this is already a complete description of the hom-categories in $\UnivGen$:

\begin{proposition}\label{prop:Essential_Surjectivity_Universal_BC_Functor}
	Any universal $(I,P)$-biadjointable functor $\myC^\op\to\UnivGen$ is essentially surjective.
	\begin{proof}
		This is completely analogous to the left/right adjointable case treated in \cite[Lemma~2.3.7]{ElmantoHaugseng2023Distributivity}; for the reader's convenience, we repeat the argument.

		Write $\twoI\subset\UnivGen$ for the full 2-subcategory spanned by the objects in the image of $h\colon\myC^\op\to\UnivGen$, and observe that the induced functor $\hbar\colon\myC^\op\to\twoI$ is still $(I,P)$-biadjointable. Thus, $\hbar$ admits an extension to $\UnivGen$ making the left-hand triangle in the following diagram commute:
		\[
			\begin{tikzcd}
				& \myC^\op\arrow[d,"\hbar"{description}]\arrow[dl,"h"', bend right=10pt]\arrow[dr,"h",bend left=10pt]\\[.5em]
				\UnivGen\arrow[r, dashed] & \twoI\arrow[r,hook] & \UnivGen\rlap.
			\end{tikzcd}
		\]
		On the other hand, the right-hand triangle commutes by construction. The universal property of $\UnivGen$ therefore shows that the composite $\UnivGen\to\twoI\to\UnivGen$ is the identity; in particular, $\twoI\hookrightarrow\UnivGen$ is essentially surjective. By definition of $\twoI$, this precisely tells us that $h$ is essentially surjective, as claimed.
	\end{proof}
\end{proposition}

Combining all the above statements, we now obtain the following recognition principle for the universal $(I,P)$-biadjointable functor out of $\myC\catop$:

\begin{theorem}
	\label{thm:Recognizing_The_Universal_Biadjointable_Functor}
    An $(I,P)$-biadjointable functor $F\colon \myC\catop \to \UnivGen$ is universal if and only if it is essentially surjective and for every object $a$ of $\myC$ the composite
    \[
        \myC\catop \xrightarrow{\;F\;} \UnivGen \xrightarrow{\;\HOM_{\UnivGen}(F(a),-)\;} \CAT
    \]
    is a free $(I,P)$-biadjointable functor on $\id_{F(a)}$.
\end{theorem}
\begin{proof}
    The `only if'-direction is immediate from Propositions~\ref{prop:Hom_Cats_Free_BC_Functor} and \ref{prop:Essential_Surjectivity_Universal_BC_Functor}. Conversely, assume that $F$ satisfies the two conditions in the theorem. By the universal property of $h\colon \myC\catop \to \Univ{\myC}{I}{P}$, the functor $F$ uniquely extends to a 2-functor $\overline{F}\colon \Univ{\myC}{I}{P} \to \UnivGen$, and it will suffice to show that $\overline{F}$ is an equivalence. Since $F = \overline{F} \circ h$ is essentially surjective by assumption, also $\overline{F}$ is essentially surjective. For fully faithfulness, we must show that for all objects $X,Y \in \Univ{\myC}{I}{P}$ the induced functor
    \[
        \overline{F}\colon \HOM_{\Univ{\myC}{I}{P}}(X,Y) \to \HOM_{\UnivGen}(\overline{F}(X),\overline{F}(Y))
    \]
    is an equivalence of categories. By essential surjectivity of $h\colon \myC\catop \to \Univ{\myC}{I}{P}$, it suffices to show that for every object $a \in \myC$ the induced transformation
    \begin{equation}\label{eq:bar-F-on-hom-cats}
        \overline{F}\colon \HOM_{\Univ{\myC}{I}{P}}(h(a),h(-)) \to \HOM_{\UnivGen}(F(a),F(-))
	\end{equation}
    is an equivalence of $(I,P)$-biadjointable functors $\myC\catop \to \CAT$.

	Observe that the left-hand side is free on the identity class by Proposition~\ref{prop:Hom_Cats_Free_BC_Functor} and the right-hand side is so by assumption on $F$. By 2-functoriality of $\overline{F}$, $(\ref{eq:bar-F-on-hom-cats})$ sends the identity to the identity, so the claim follows by the universal property of free $(I,P)$-biadjointable functors.
\end{proof}

\section{Free semiadditive categories}
\label{sec:Free_Semiadditive}
In this section we will prove the key ingredient for our main theorem: we will identify the free $(I,P)$-biadjointable functors $\myT\catop \to \Cat$ on an object $a \in \myT$ under some standard assumptions on $I,P$:

\begin{convention}\label{conv:most-general}
	Throughout this whole section, we fix a category $\myT$ together with two wide subcategories $I,P\subset \myT$. We write $E$ for the collection of maps in $\myT$ that factor as $e=pi$ with $p\in P$ and $i\in I$. We have the following standing assumptions:
	\begin{enumerate}
		\item $I$ and $P$ are left cancellable and closed under base change.
		\item $E$ is closed under composition (hence a wide subcategory of $\myT$).
		\item Every  map in $I\cap P$ is truncated.
	\end{enumerate}
	In the terminology of \cite[Definition~A.5.9]{Mann2022SixFunctor}, this says that $(I,P)$ is a \emph{suitable decomposition} of $(\myT,E)$. Two different (more restrictive) variations of the above definition have also appeared under the name \emph{Nagata setup} in \cite{Kuijper_Axiomatixation_6FF, Dauser_Kuijper_Uniqueness}.
\end{convention}

\begin{notation}
	To easily distinguish the above classes, we will denote maps in $I$ by `\kern.5pt$\rightarrowmono$,' maps in $P$ by `$\rightarrowepic$,' and maps in $E$ by `$\rightarrowepmo$.' These notations are not meant to invoke any association about the nature of these morphisms (for example, maps in $I$ are not assumed to be monomorphisms or any sort of embeddings). All maps decorated as $\iso$ are equivalences.
\end{notation}

Our proof will proceed via techniques from parametrized category theory \cite{BDGNS2016ExposeI,martiniwolf2021limits}, so it will be convenient to employ some of the standard terminology of the field, cf.~Example~\ref{ex:param}:
\begin{itemize}
	\item A functor $\Xx\colon \myT\catop \to \Cat$ will be referred to as a \textit{$\myT$-category}.
	\item Natural transformations $\Xx \to \Yy$ will be referred to as \textit{$\myT$-functors}, or sometimes simply \textit{functors} for short.
	\item We will say that $\Xx$ is \textit{$I$-cocomplete} if it is left $I$-adjointable, and dually that it is \textit{$P$-complete} if it is right $P$-adjointable. We wish to think of the left adjoints $i_{\sharp}$ and right adjoints $p_*$ as certain \textit{parametrized (co)limits}; accordingly, we will refer to $I$-left adjointable natural transformations as \emph{$I$-cocontinuous $\myT$-functors}, while $P$-right adjointable transformations will be called \emph{$P$-continuous $\myT$-functors.}
	\item We say that $\Xx$ is \textit{$(I,P)$-semiadditive} if it is $(I,P)$-biadjointable; this may be interpreted as asking $I$-colimits to commute with $P$-limits. We will call $(I,P)$-biadjointable natural transformations \emph{$(I,P)$-bicontinuous.}
	\item We denote by
	\[
	\Fun^{I\text-\amalg}_\myT(\Xx,\Yy), \qquad \Fun^{P\text-\Pi}_\myT(\Xx,\Yy), \qquad \Fun^{I\text-\amalg,P\text-\Pi}_\myT(\Xx,\Yy)
	\]
	the categories of $I$-cocontinuous, $P$-continuous, and $(I,P)$-bicontinuous functors, respectively, i.e.~the $\HOM$'s in the 2-{\hskip0pt}categories $\FUN_{(I,\core C)\dbadj}(\myT\catop,\CAT)$, $\FUN_{(\core C,P)\dbadj}(\myT\catop,\CAT)$, and $\FUN_{(I,P)\dbadj}(\myT\catop,\CAT)$, respectively.
\end{itemize}

\begin{remark}
Let us explicitly note that we depart from the standard notational convention in parametrized category theory by denoting the left adjoint to a functor $f^*$ by $f_{\sharp}$: in the literature on parametrized category theory this left adjoint is commonly denoted by $f_!$, which unfortunately clashes with the common notation for the exceptional pushforwards in a 6-functor formalism. We point this out to avoid any confusion which might arise when consulting the references (particularly \cite{CLL_Spans}) used in this section, where $f_!$ is used to denote the left adjoint of $f^*$.
\end{remark}

\begin{remark}
	If $Q=I=P$ (meaning in particular that all maps in $Q$ are truncated), then $(I,P)$-semiadditivity recovers one of the equivalent definitions of \emph{$Q$-semiadditivity} from \cite[Proposition~3.44 and Remark~3.46]{CLL_Spans}. We conclude that, for general $I$ and $P$ as before, any $(I,P)$-semiadditive $C$-category in the above sense is $(I\cap P)$-semiadditive in the sense of \cite{CLL_Spans}.\footnote{\cite{CLL_Spans} works in the more general setting of categories parametrized by a topos $\mathcal B$; the case of $\myC$-categories is obtained by taking $\mathcal B$ to be the presheaf topos $\PSh(C)$ (possibly after passing to a larger universe), see in particular the `user interface' presented in §3.3 of \emph{op.\ cit.}}
\end{remark}

Translated into the language of parametrized category theory, the question addressed in this section is the following: given a $\myT$-category $\Xx$, what is the free $(I,P)$-semiadditive $\myT$-category generated by $\Xx$? Applied to the $\myT$-category $\Xx = \Hom_{\myT}(-,a)\colon \myT\catop \to \Spc \subset \Cat$, this will in particular identify the free $(I,P)$-biadjointable functor generated by an element in degree $a \in C$. To motivate what is about to come, let us briefly recall the situation in the non-parametrized case:

\begin{example}
	Let $\Xx$ be a category. We denote by $\Xx^{\amalg}$ the category obtained by freely adjoining finite coproducts to $\Xx$. The inclusion $\core{\Xx} \hookrightarrow \Xx$ of the groupoid core induces an inclusion $(\core{\Xx})^{\amalg} \hookrightarrow \Xx^{\amalg}$ whose image we refer to as the subcategory $(\Xx^{\amalg})_{\fold}$ of \textit{fold maps}. The free semiadditive category generated by $\Xx$ is then given by the \emph{span category} (to be recalled below)
	\[
		\Span_{\fold,\all}(\Xx^{\amalg})
	\]
	of $\Xx^{\amalg}$, where the backwards maps are fold maps; see \cite[Lemma~C.4]{BachmannHoyois2021Norms}.
\end{example}

The parametrized situation will look similar. Given a $\myT$-category $\Xx$, we may form its \textit{free $E$-cocompletion} $\Xx^{E\text-\amalg}$, recalled below in Definition~\ref{def:Free_N_Cocompletion}. This has certain wide subcategories $(\Xx^{E\text-\amalg})_{P\text-\fold}$ of \textit{$P$-fold maps} and $(\Xx^{E\text-\amalg})_{I}$ of maps \emph{over $I$}. The free $(I,P)$-semiadditive $\myT$-category generated by $\Xx$ is then given by a certain parametrized span category $\CIP{\Xx}{E}{I}{P}$. More precisely, the inclusion $i\colon \Xx \hookrightarrow \CIP{\Xx}{E}{I}{P}$ satisfies the following universal property:

\begin{theorem}\label{thm:biadd}
	Let $I,P\subset E\subset \myT$ be as in Convention~\ref{conv:most-general}. Let $\Xx$ be any $\myT$-category, and let $\Yy$ be an $(I,P)$-semiadditive $\myT$-category. Then restriction along ${i}$ induces an equivalence
	\[
	\Fun_\myT^{I\text-\amalg,P\text-\Pi}(\CIP{\Xx}{E}{I}{P},\Yy)\iso\Fun_\myT(\Xx,\Yy).
	\]
	In other words, ${i}$ exhibits $\CIP{\Xx}{E}{I}{P}$ as the free $(I,P)$-semiadditive $\myT$-category generated by $\Xx$.
\end{theorem}

As an application of this theorem (combined with Proposition~\ref{prop:Hom_Cats_Free_BC_Functor}) we will obtain an explicit identification of the $\HOM$-categories in the target of the universal biadjointable functor:

\begin{theorem}\label{thm:free_IP_biadj_on_1}
For an object $a\in C$, the functor
	\begin{equation}\label{eq:free_adj_funct_on_1}
		F = \Span_{P,I}(\myC_{/a}\times_{\myC} E_{/-})\colon \myC^{\op} \to \Cat
	\end{equation}
	is the free $(I,P)$-biadjointable functor $\myC^{\op} \to \Cat$ on the class $\smash{(a\xleftarrow{\;\id\;}a\xrightarrow{\;\id\;}a)}$ in $F(a) = \Span_{P,I}(C_{/a}\times_C E_{/a})$.
\end{theorem}

The goal of this section is to provide a proof of these results. We start in §\ref{subsec:Construction_Universal_Example} by giving a precise construction of $\CIP{\Xx}{E}{I}{P}$. The remaining subsections are then dedicated to proving increasingly more general instances of Theorem \ref{thm:biadd}.

\subsection{Construction of the universal example}
\label{subsec:Construction_Universal_Example}
We begin by recalling Macpherson's description of free parametrized (co)completions and our conventions on span categories.

\begin{definition}[Free $N$-cocompletion]\label{def:Free_N_Cocompletion}
	Let $N \subset \myT$ be a left cancellable wide subcategory closed under base change, and let $\Xx$ be a $\myT$-category. The \emph{free $N$-cocompletion of $\Xx$} is the initial $N$-cocomplete $\myT$-category $\Xx^{N\text-\amalg}$ equipped with a $\myT$-functor $j\colon \Xx \to \Xx^{N\text-\amalg}$. More precisely, $\Xx^{N\text-\amalg}$ is $N$-cocomplete, and for any other $N$-cocomplete $\myT$-category $\Yy$, restriction along $j$ induces an equivalence of categories
	\[
		j^*\colon \Fun_\myT^{N\text-\amalg}(\Xx^{N\text-\amalg},\Yy) \iso \Fun_\myT(\Xx,\Yy).
	\]
\end{definition}

\begin{remark}
	\label{rmk:Free_N_Cocompletion}
	The $\myT$-category $\Xx^{N\text-\amalg}$ exists and is unique by \cite[Theorem~7.1.13]{martiniwolf2021limits}, where it is also shown that the inverse to $j^*$ is given by parametrized left Kan extension along $j$ in the sense of \cite[Definition~6.3.1]{martiniwolf2021limits}.
\end{remark}

An explicit model for the free $N$-cocompletion has been described by Macpherson:

\begin{construction}\label{cons:MacPherson_Free_Cocompletion}
	For $N\subset \myT$ as in Definition \ref{def:Free_N_Cocompletion}, let $\Ar_{N}(\myT)$ denote the full subcategory of the arrow category $\mathrm{Ar}(\myT)$ spanned by the maps in $N$. Given a $\myT$-category $\Xx$ with cartesian unstraightening $\Un^\ct(\Xx) \to \myT$, we consider the following pullback diagram
	\[
	\begin{tikzcd}
		\Ar_N(\myT) \times_C \Un^\ct(\Xx) \arrow[r] \arrow[d] \arrow[dr,pullback] & \Un^\ct(\Xx) \arrow[d] \\
		\Ar_N(\myT) \arrow[d,"\ev_1"'] \arrow[r, "\ev_0"'] & \myT\\
		\myT\rlap.
	\end{tikzcd}
	\]
	The left vertical composite is a cartesian fibration, with cartesian edges given by pairs consisting of a pullback square in $\myT$ (i.e.~a {cartesian} edge of $\ev_1\colon\Ar_N(\myT)\to\myT$) and a cartesian edge of $\Unct(\Xx)$.

	The identity section $\text{const}\colon \myT\hookrightarrow \Ar_N(\myT), \, x \mapsto \id_x$ induces a cartesian functor $\bar{\jmath}\colon \Un^\ct(\Xx) = \Un^\ct(\Xx) \times_C C \hookrightarrow \Ar_N(\myT) \times_C \Un^\ct(\Xx)$ over $C$.
\end{construction}

\begin{proposition}[Macpherson]
	\label{prop:MacPherson_Model_Satisfies_UP}
	Given a $\myT$-category $\Xx$, there is an equivalence
	\[
		\Un^\ct(\Xx^{N\text-\amalg}) \simeq \Ar_N(\myT) \times_C \Un^\ct(\Xx)
	\]
	of cartesian fibrations over $\myT$. The functor $\Un^\ct(j)\colon \Un^\ct(\Xx) \to \Un^\ct(\Xx^{N\text-\amalg})$ corresponds to the inclusion $\bar{\jmath}$ under this equivalence.
\end{proposition}
\begin{proof}
	Translated to our language, Macpherson shows in \cite[Proposition 3.5.2]{MacPherson2022Bivariant} that the cartesian fibration $\Ar_N(\myT) \times_C \Un^\ct(\Xx) \to \myT$ corresponds to an $N$-cocomplete $\myT$-category, and that the inclusion $\bar{\jmath}$ exhibits it as the unstraightening of the free $N$-cocomplete $\myT$-category on $\Xx$.
\end{proof}

\begin{remark}\label{rmk:Description_Un(X^N-amalg)}
	As a consequence of the proposition, we may explicitly describe the free $N$-cocompletion of $\Xx$ as $\Xx^{N\text-\amalg}(a)=\Un^\ct(\Xx)\times_{\myT} N_{/a}$ for every $a\in\myT$. A generic object of $\Xx^{N\text-\amalg}(a)$ is thus a pair $(X,n\colon b \to a)$ with $X \in \Xx(b)$ and $n \in N$. We may think of this object as a `formal $N$-colimit' $n_{\sharp}(j(X))$ of the `generator' $j(X) \in \Xx^{N\text-\amalg}(b)$. We will generally display morphisms in $\Xx^{N\text-\amalg}(a)$ as diagrams of the form
	\[
	\begin{tikzcd}
		X \arrow[r] & Y\\[-2ex]
		b\arrow[d] \arrow[r] & c\arrow[d]\\
		a\arrow[r,equals] & a\rlap,
	\end{tikzcd}
	\]
	where the morphisms to $a$ are in $N$ (hence so is $b \to c$ by left cancellability) and the morphism $X \to Y$ in $\Un^{\ct}(\Xx)$ maps to $b \to c$.
\end{remark}

\begin{definition}
	We dually define the \textit{free $N$-completion} of $\Xx$ as the initial $N$-complete $\myT$-category $\Xx^{N\text-\Pi}$ equipped with a $\myT$-functor $\Xx \to \Xx^{N\text-\Pi}$, in the sense that restriction along this functor induces an equivalence
	\[
		\Fun_C^{N\text-\Pi}(\Xx^{N\text-\Pi},\Yy) \iso \Fun_C(\Xx,\Yy)
	\]
	for every $N$-complete $\myT$-category $\Yy$.
\end{definition}

\begin{remark}\label{rmk:(co)cart-edges-of-Pi}
	Observe that we have $\Xx^{N\text-\Pi}\simeq\big((\Xx^\op)^{N\text-\amalg}\big){}^\op$ by comparing universal properties. This results in a concrete description of $\Xx^{N\text-\Pi}$:

	Recall that if $\Xx\colon \myT^\op\to\Cat$ is a functor with cocartesian unstraightening $\Un^\cc(\Xx)\to \myT^\op$, then the \emph{cartesian} unstraightening of the functor $\Xx^\op\colon \myT^\op\to\Cat$ is given by $\Un^\cc(\Xx)^\op$, and vice-versa. Applying this once, we get from Proposition \ref{prop:MacPherson_Model_Satisfies_UP} that the cartesian unstraightening of $(\Xx^\op)^{N\text-\amalg}$ sits in a pullback square
	\[
	\begin{tikzcd}
		\Un^\ct\big({(\Xx^\op)^{N\text-\amalg}}\big)\arrow[dr,pullback] & {\Un^\cc(\Xx)}^\op \\
		\Ar_N(\myT)& {\myT}\rlap.
		\arrow[from=1-1, to=1-2]
		\arrow[from=1-1, to=2-1]
		\arrow[from=1-2, to=2-2]
		\arrow["s", from=2-1, to=2-2]
	\end{tikzcd}
	\]
	Applying this a second time, it follows that the \emph{co}cartesian unstraightening of $\Xx^{N\text-\Pi}$ sits in a pullback square
	\[
	\begin{tikzcd}
		\Un^\cc\big({\Xx^{N\text-\Pi}}\big)\arrow[dr,pullback] & {\Un^\cc(\Xx)} \\
		\Ar_N(\myT)^\op& {\myT^\op}\!\rlap.\,
		\arrow[from=1-1, to=1-2]
		\arrow[from=1-1, to=2-1]
		\arrow[from=1-2, to=2-2]
		\arrow["s", from=2-1, to=2-2]
	\end{tikzcd}
	\]
	Morphisms in $\Un^\cc\big({\Xx^{N\text-\Pi}}\big)$ may thus be displayed as diagrams of the form
	\[
	\begin{tikzcd}
		X \arrow[r] & X'\\[-2ex]
		b\arrow[d, "n"'] & b'\arrow[d, "n'"] \arrow[l]\\
		a & a' \arrow[l]\!\rlap,\,
	\end{tikzcd}
	\]
	where the bottom square is a morphism in $\Ar_N(\myT)\catop$ and $X \to X'$ is a morphism in $\Un^{\cc}(\Xx)$ which maps to $b \leftarrow b'$ in $\myT\catop$. Dual to the case of $\Xx^{N\text-\amalg}$, such a morphism is cocartesian if and only if the square in $\myT$ is a pullback square and the morphism $X \to X'$ is a cocartesian lift of $b \leftarrow b'$. Since $\Xx^{N\text-\Pi}$ is $N$-complete, it also admits cartesian morphisms over morphisms in $N$, which are those for which the maps $b \leftarrow b'$ and $X \to X'$ are equivalences.
\end{remark}

Next we will briefly recall our conventions on span categories \cite{Barwick2017SpectralMackey, HHLN2022TwoVariable}.

\begin{definition}
We define an \emph{adequate triple} to be a triple $(\Xx,\Xx_B,\Xx_F)$ of a category $\Xx$ together with two wide subcategories $\Xx_B,\Xx_F\subset \Xx$ of \emph{backwards} and \emph{forwards} maps respectively, such that the base change of a forward map along a backward map exists in $\Xx$ and is again in $\Xx_F$ and vice-versa. A map $(\Xx,\Xx_B,\Xx_F) \to (\Yy,\Yy_B,\Yy_F)$ of adequate triples is a functor $F\colon \Xx\to \Yy$ such that $F(\Xx_B)\subset \Yy_B$, $F(\Xx_F) \subset \Yy_F$, and $F$ preserves pullbacks of forward maps along backwards maps. This defines a category $\AdTrip$, see \cite[Definition 2.1]{HHLN2022TwoVariable} for a formal definition.
\end{definition}

\begin{definition}\label{def:Tw[n]}
	We define $(s,t)\colon \Tw[n]\to [n]\times [n]^{\op}$ to be the cartesian unstraightening\footnote{To avoid confusion we note that sometimes the twisted arrow category is defined to be the \emph{cocartesian} unstraightening of the hom functor. The two conventions differ by taking opposites.} of the hom functor of $n$, also known as the \emph{twisted arrow category} of the $n$-simplex category. We denote the objects of $\Tw[n]$ by $(i\leq j)$. We also define the wide subcategories
	\begin{align*}\Tw[n]^{\mathrm{f}}\kern1pt\null&\coloneqq \Tw[n]\times_{[n]\times[n]^{\op}} [n]\times \core{[n]^{\op}}\\ \Tw[n]^{\mathrm{b}} &\coloneqq \Tw[n]\times_{[n]\times[n]^{\op}} \core{[n]} \times [n]^{\op}.\end{align*}
\end{definition}

\begin{example}\label{ex:picture_of_Tw}
	The category $\Tw[2]$ is the poset
	\[
	\begin{tikzcd}
		[row sep=2.4ex, column sep=-1.5ex]
		& & (0 \leq 2) \ar[ld]\ar[rd]& & \\
		& (0 \leq 1) \ar[ld]\ar[rd]& & (1 \leq 2) \ar[ld]\ar[rd]& \\
		(0 \leq 0) && (1 \leq 1) && (2 \leq 2),
	\end{tikzcd}
	\]
	and $\Tw[n]$ is the obvious generalization. $\Tw[n]^{\mathrm{f}}$ and $\Tw[n]^{\mathrm{b}}$ correspond to the subcategories spanned by maps pointing to the right and left, respectively.
\end{example}

\begin{construction}\label{cons:Span}
By \cite[Theorem~2.13]{HHLN2022TwoVariable} there is a functor
\[
\Span(-)\colon \AdTrip\to \Cat, \quad (\Xx,\Xx_B,\Xx_F)\mapsto \Span(\Xx,\Xx_B,\Xx_F)
\]
which sends an adequate triple $(\Xx,\Xx_B,\Xx_F)$ to the category of \emph{spans} in $\Xx$ with backwards maps in $\Xx_B$ and forwards maps in $\Xx_F$. It is determined by the existence of a natural equivalence
\[
	\Hom_{\Cat}([n],\Span(\Xx,\Xx_B,\Xx_F)) \simeq \Hom_{\AdTrip}(\Tw[n],(\Xx,\Xx_B,\Xx_F)),
\]
i.e.\ the complete Segal space associated to $\Span(\Xx,\Xx_B,\Xx_F)$ is the simplicial space determined by the right-hand side. Informally, we may describe this category as follows: its space of objects agrees with that of $\Xx$, and for all $X,Y\in \Xx$, the space of morphisms from $X$ to $Y$ in $\Span(\Xx,\Xx_B,\Xx_F)$ is given by the space of spans
\[\begin{tikzcd}
	& Z \\
	X && Y
	\arrow["b"', from=1-2, to=2-1]
	\arrow["f", from=1-2, to=2-3]
\end{tikzcd}\]
in $\Xx$ where $b$ is in $\Xx_B$ and $f$ is in $\Xx_F$. Composition of spans is given by forming pullbacks in $\Xx$.
\end{construction}

\begin{notation}
	For brevity, we will write $\Span(\Xx,\Xx_F)\coloneqq\Span(\Xx,\Xx,\Xx_F)$ whenever $\Xx_F\subset\Xx$ is a wide subcategory closed under base change. We will also often write $\Span_{B,F}(\Xx)\coloneqq\Span(\Xx,\Xx_B,\Xx_F)$ and $\Span_F(\Xx)\coloneqq\Span_{\text{all},F}(\Xx)\coloneqq\Span(\Xx,\Xx_F)$.
\end{notation}

We note the following useful property of $\Span(-)$.

\begin{proposition}
The category $\AdTrip$ admits limits, which are computed in $\Cat$, and the functor $\Span\colon \AdTrip \to \Cat$ preserves limits.
\end{proposition}

\begin{proof}
This is a combination of \cite[Lemma 2.4 and Theorem 2.18]{HHLN2022TwoVariable}.
\end{proof}

Having completed our recollections, we now begin our construction of the $\myT$-category $\CIP{\Xx}{E}{I}{P}$. To this end, let $I,P\subset E\subset \myT$ be as in Convention~\ref{conv:most-general}. As previously mentioned, we want to apply the construction of parametrized free cocompletions to the case $N=E$; for this to make sense we first need:

\begin{lemma}
The wide subcategory $E\subset \myT$ is left cancellable and closed under base change.

	\begin{proof}
		The second statement is clear. For the first statement it will then suffice by \cite[Lemma 3.2]{CLL_Spans} that $E$ is closed under diagonals. For this we consider a generic map in $E$ with factorization
		\[
		\begin{tikzcd}
			a\arrow[r, mono, "i"] & b\arrow[r, epic, "p"] & c.
		\end{tikzcd}
		\]
		Its diagonal $\Delta_{pi}$ then factors as
		\[
		\begin{tikzcd}
			a\arrow[r, "\Delta_i"] & a\times_ba\arrow[r, "{\id\times_p\id}"] &[1.5em] a\times_ca\rlap.
		\end{tikzcd}
		\]
		The first map belongs to $I$ as $I$ is left cancellable (hence closed under diagonals). On the other hand, the second map is a base change of $(\id,i)\colon a\to a\times_cb$. The projection $a\times_cb\to a$ is in $P$ as it is a base change of $p\colon b\to c$. By left cancellability, we conclude that $(\id,i)$ also belongs to $P$. Thus, the above diagram expresses $\Delta_{pi}$ as a composite of a map in $I$ followed by a map in $P$, i.e.~$\Delta_{pi}$ is in $E$ as claimed.
	\end{proof}
\end{lemma}

\begin{notation}
	Recall from Remark \ref{rmk:Description_Un(X^N-amalg)} that morphisms in the category $\Xx^{E\text-\amalg}(a)$ for $a \in \myT$ take the form
	\[
	\begin{tikzcd}
		X \arrow[r] & Y\\[-2ex]
		b\arrow[d,epmo] \arrow[r,epmo] & c\arrow[d,epmo]\\
		a\arrow[r,equals] & a\rlap,
	\end{tikzcd}
	\]
	where the bottom square is in $E$ and the morphism $X \to Y$ in $\Un^{\ct}(\Xx)$ maps to $b \to c$. We say that this morphism \textit{lies over $I$} if the map $b \to c$ is in $I$. We call it a \textit{$P$-fold map} if the morphism $b \to c$ is in $P$ and the morphism $X \to Y$ is a cartesian lift of $b \to c$. We denote the resulting two wide subcategories of $\Xx^{E\text-\amalg}(a)$ as follows:
	\begin{align*}
		\Xx^{E\text-\amalg}_{P\textup{-fold}}(a)&\coloneqq \Un^\ct(\Xx)_\ct\times_{\myT} E_{/a}\times_EP,\\
		\Xx^{E\text-\amalg}_I(a)&\coloneqq \Un^\ct(\Xx)\times_{\myT} E_{/a}\times_E I;
	\end{align*}
	here $\Un^\ct(\Xx)_\ct \subset \Un^\ct(\Xx)$ denotes the wide subcategory spanned by the cartesian edges (i.e.~this is the right fibration over $\myT$ classified by the functor $\core\Xx\colon C\catop \to \Spc$).
	Using the description of the cartesian edges for $\Un^\ct(\Xx^{E\text-\amalg}) \to C$ given in Construction \ref{cons:MacPherson_Free_Cocompletion}, one immediately checks that these wide subcategories are closed under cartesian transport, resulting in wide $\myT$-subcategories $\Xx^{E\text-\amalg}_{P\textup{-fold}}, \Xx^{E\text-\amalg}_I \subset \Xx^{E\text-\amalg}$.
\end{notation}

\begin{lemma}
	The triple $(\Xx^{E\text-\amalg},\Xx^{E\text-\amalg}_{P\textup{-fold}},\Xx_I^{E\text-\amalg})$ defines an adequate triple in $\myT$-categories, i.e.~it promotes to a functor $\myT^\op\to\AdTrip$.
	\begin{proof}
		Let us first fix $a\in\myT$ and show that $(\Xx^{E\text-\amalg}(a),\Xx^{E\text-\amalg}_{P\textup{-fold}}(a),\Xx_I^{E\text-\amalg}(a))$ is adequate. It follows from the assumptions that $(E,P,I)$ is adequate, so \cite[Proposition~2.6]{HHLN2022TwoVariable} shows that also the triple
		\[
			\big(\Un^\ct(\Xx)\times_{\myT}E,\Un^\ct(\Xx)_\ct\times_{\myT}P,\Un^\ct(\Xx)\times_{\myT}I\big)
		\]
		is adequate, with the projection to $(E,P,I)$ being a map of adequate triples. On the other hand, also $(E_{/a},E_{/a}\times_EP,E_{/a}\times_EI)\to (E,P,I)$ is a map of adequate triples by our assumptions; as $\AdTrip$ has pullbacks (computed in $\Cat$), the claim follows by recognizing $(\Xx^{E\text-\amalg}(a),\Xx^{E\text-\amalg}_{P\textup{-fold}}(a),\Xx_I^{E\text-\amalg}(a))$ as the pullback
		\[
			\big(\Un^\ct(\Xx)\times_{\myT}E,\Un^\ct(\Xx)_\ct\times_{\myT}P,\Un^\ct(\Xx)\times_{\myT}I\big)\times_{(E,P,I)}\big(E_{/a},E_{/a}\times_EP,E_{/a}\times_EI\big).
		\]

		It remains to show that for any $f\colon a\to b$ the structure map $f^*\colon\Xx^{E\text-\amalg}(b)\to\Xx^{E\text-\amalg}(a)$ preserves the requisite pullback squares. For this observe that by the proof of the aforementioned \cite[Proposition~2.6]{HHLN2022TwoVariable}, a square
		\[
			\begin{tikzcd}
				\cdot\arrow[r,"p"]\arrow[d,"i"'] & \cdot\arrow[d,"i'"]\\
				\cdot\arrow[r,"p'"'] & \cdot
			\end{tikzcd}
		\]
		in $\Un^\ct(\Xx)\times_{\myT}E$, where $p,p'$ are cartesian over $P$ and $i,i'$ are maps over $I$, is a pullback square if and only if its image in $E$ is so. As the required pullbacks in $\Un^\ct(\Xx)\times_{\myT}E_{/a}$ are computed componentwise, and since the structure map $\Un^\ct(\Xx)\to \myT$ induces a natural map
		\[
			\big(\Xx^{E\text-\amalg}(-),\Xx^{E\text-\amalg}_{P\textup{-fold}}(-),\Xx_I^{E\text-\amalg}(-)\big)\to \big(E_{/-},E_{/-}\times_EP,E_{/-}\times_EI\big),
		\]
 		it will be enough that $f^*\colon E_{/b}\to E_{/a}$ preserves pullbacks (of maps in $P$ along maps in $I$), which is obvious.
	\end{proof}
\end{lemma}

\begin{definition}
	Let $I,P\subset E\subset\myT$ be as in Convention~\ref{conv:most-general}. For any $\myT$-category $\Xx$, we write $\CIP{\Xx}{E}{I}{P}$ for the composite
	\[
		\myT^\op\xrightarrow{(\Xx^{E\text-\amalg},\Xx^{E\text-\amalg}_{P\textup{-fold}},\Xx_I^{E\text-\amalg})}\AdTrip\xrightarrow{\Span}\Cat.
	\]
\end{definition}

More concretely, this means a morphism in $\CIP{\Xx}{E}{I}{P}(a)$ can be pictured as a pair
\[
	\begin{tikzcd}
		X & Y\arrow[l, "\ct"']\arrow[r] & Z\\[-2ex]
		b\arrow[d,epmo] & \arrow[d,epmo]\arrow[l, epic, "p"']c\arrow[r,mono, "i"] & d\arrow[d,epmo]\\
		a\arrow[r,equals] & a\arrow[r,equals] & a
	\end{tikzcd}
\]
with the vertical maps in $E$, the horizontal maps in the middle row in $P$ and $I$ as indicated, and such that $Y\to X$ is a cartesian edge of $\Un^\ct(\Xx)$ over $p$ while $Y\to Z$ is any edge over $i$.

Having introduced the definition of $\CIP{\Xx}{E}{I}{P}$ for a $\myC$-category $\Xx$, we can now explain how Theorem~\ref{thm:free_IP_biadj_on_1} follows from Theorem~\ref{thm:biadd}.

\begin{proof}[Proof of Theorem~\ref{thm:free_IP_biadj_on_1}, assuming Theorem~\ref{thm:biadd}]
Write $\ul{a}$ for the presheaf represented by $a\in \myC$, which we may view as a $\myC$-category. Combining Theorem~\ref{thm:biadd} with the usual Yoneda lemma, $\CIP{\ul{a}}{E}{I}{P}$ will then be the free $(I,P)$-biadjointable functor generated by the image of $\id_a\in\ul{a}(a)=\Hom(a,a)$ under the inclusion $\ul a\hookrightarrow\CIP{\ul{a}}{E}{I}{P}$.

Recall now that the source functor $s\colon C_{/a} \to C$ is the cartesian unstraightening of $\ul a$. The result now follows since by direct inspection $\CIP{\ul{a}}{E}{I}{P}$ is precisely the functor (\ref{eq:free_adj_funct_on_1}).
\end{proof}

\begin{remark}\label{rmk:unstraigthening-of-free}
	Using \cite[Theorem~3.9]{HHLN2022TwoVariable} we can describe the cocartesian unstraightening of $\CIP{\Xx}{E}{I}{P}$: it is given by the pullback
	\[
		\Span_{P\textup{-pb},I\textup{-fw}}(\Ar_E(\myT))\times_{\Span_I(\myT)}\Span_{\ct,\textup{all}}(\Un^\ct(\Xx)),
	\]
 	where $\Ar_E(\myT)_{P\textup{-pb}}\subset\Ar_E(\myT)$ denotes the wide subcategory given by those squares
	\[
		\begin{tikzcd}
			a\arrow[d,epmo]\arrow[r] & b\arrow[d,epmo]\\
			c\arrow[r] & d
		\end{tikzcd}
	\]
	such that the induced map $a\to b\times_dc$ is in $P$, while $\Ar_E(\myT)_{I\textup{-fw}}$ is the subcategory given by the squares of the form
	\[
		\begin{tikzcd}
			a\arrow[d,epmo]\arrow[r,mono] & b\arrow[d,epmo]\\
			c\arrow[r,"\sim"'] & d\rlap.
		\end{tikzcd}
	\]
	We may therefore picture morphisms in the unstraightening of $\CIP{\Xx}{E}{I}{P}$ as spans
	\[
		\begin{tikzcd}[column sep=small,row sep=small]
			X &[-.75em]&[-.5em]\arrow[ll,"\ct"'] Y\arrow[rr] && Z\\[-.5ex]
			a\arrow[dd,epmo] &&\arrow[ll] b\arrow[dl,epic]\arrow[dd,epmo]\arrow[rr,mono] && c\arrow[dd,epmo]\\[-1em]
			& \cdot\arrow[ul,dashed]\arrow[dr,dashed]\arrow[dl,phantom,"\llcorner"{very near start,xshift=3pt,yshift=4pt}]\\
			a'&&\arrow[ll] b'\arrow[rr,equals] && c'\!\rlap.\,
		\end{tikzcd}
	\]
	By the aforementioned \cite[Theorem~3.9]{HHLN2022TwoVariable} such an edge is cocartesian if and only if it is a backwards morphism corresponding to a cartesian edge, i.e.~if and only if it is of the form
	\[
		\begin{tikzcd}[column sep=small,row sep=small]
			X &[-.75em]&[-.5em]\arrow[ll,"\ct"'] Y\arrow[rr,equals] && Y\\[-.5ex]
			a\arrow[dd,epmo] &&\arrow[ll] b\arrow[ddll,phantom,"\llcorner"{very near start}]\arrow[dd,epmo]\arrow[rr,equals] && b\arrow[dd,epmo]\\[-1em]
			& \phantom\cdot\\
			a'&&\arrow[ll] b'\arrow[rr,equals] && b'\!\rlap.\,
		\end{tikzcd}
	\]
\end{remark}

\begin{example}\label{ex:unst_of_Span_of_slices}
Using the previous remark, we can make the cocartesian unstraightening of $\CIP{\ul a}{E}{I}{P}$ completely explicit: as noted above, the cartesian unstraightening $\Unct(\ul{a})$ of $\ul a$ is given by the functor $\myC_{/a}\to \myC$, and so the cocartesian unstraightening of $\CIP{\ul a}{E}{I}{P}$ is given by the target map
	\[
	\Span(\ev_1)\colon\Span_{\text{$P$-pb},\text{$I$-fw}}(\Ar_E(C)\times_CC_{/a})\to C^\op,
	\]
	where the pullback is taken over the source map $\ev_0\colon\Ar_E(C)\to C$.
\end{example}

\begin{construction}
	Note that $\smash{\CIP{\Xx}{E}{I}{P}}$ comes with an inclusion $\Xx^{I\text-\amalg}\hookrightarrow\CIP{\Xx}{E}{I}{P}$ induced by the inclusion $\Xx^{I\text-\amalg}\hookrightarrow\Xx^{E\text-\amalg}$.
	Using that the cocartesian unstraightening of $\Xx$ can be computed from its \emph{cartesian} unstraightening as the span category $\Span_{\ct,\fw}\big(\Un^\ct(\Xx)\big)$ \cite[Corollary~3.18]{HHLN2022TwoVariable},
	where $\Un^\ct(\Xx)_{\fw}$ denotes the subcategory of those maps inverted by $\Un^\ct(\Xx)\to \myT$, we obtain a similar inclusion $\Xx^{P\text-\Pi}\hookrightarrow\CIP{\Xx}{E}{I}{P}$. These inclusions fit into a commutative diagram
	\[
		\begin{tikzcd}
			\Xx\arrow[r,hook,"{j}"]\arrow[d,hook,"{j}"'] & \Xx^{I\text-\amalg}\arrow[d,hook]\\
			\Xx^{P\text-\Pi}\arrow[r,hook] & \CIP{\Xx}{E}{I}{P}\rlap,
		\end{tikzcd}
	\]
	and we denote the diagonal composite by ${i}\colon\Xx\hookrightarrow\CIP{\Xx}{E}{I}{P}$.
\end{construction}

\begin{lemma}\label{lemma:universal-example-is-biadd}
	The $\myT$-category $\CIP{\Xx}{E}{I}{P}$ is $(I,P)$-semiadditive. Moreover, the inclusion $\Xx^{I\text-\amalg}\hookrightarrow\CIP{\Xx}{E}{I}{P}$ is $I$-cocontinuous, while $\Xx^{P\text-\Pi}\hookrightarrow\CIP{\Xx}{E}{I}{P}$ is $P$-continuous.
	\begin{proof}
		It suffices to verify the assumptions of \cite[Proposition~4.4]{CLL_Spans}. This is completely analogous to the special case $\Xx=1$ proven in Corollary~4.12 of \emph{op.~cit.}, and we therefore omit the details.
	\end{proof}
\end{lemma}

\begin{corollary}\label{cor:cart-of-CIP}
	The cocartesian fibration $\Un^\cc\big(\CIP{\Xx}{E}{I}{P}\big)\to \myT^\op$ admits cartesian lifts over $P^\op$, given  by the spans of the form
	\[
		\begin{tikzcd}[column sep=small,row sep=small]
			X &[-.75em]&[-.5em]\arrow[ll,equals] X\arrow[rr,equals] && X\\[-.5ex]
			b\arrow[dd,epmo] &&\arrow[ll,equals] b\arrow[dd,epmo]\arrow[rr,equals] && b\arrow[dd,epmo]\\[-1em]
			& \phantom\cdot\\
			a'&&\arrow[ll,epic]b'\arrow[rr,equals] && b'\!\rlap.\,
		\end{tikzcd}
	\]
	\begin{proof}
		This is immediate from the previous lemma, using the  identification of the cartesian edges of $\Xx^{P\text-\Pi}$ from Remark~\ref{rmk:(co)cart-edges-of-Pi}.
	\end{proof}
\end{corollary}

\subsection{The special case \texorpdfstring{\for{toc}{$P\subset I$}\except{toc}{$\bm{P\subset I}$}}{P ⊂ I}}
As a first step towards the proof of Theorem~\ref{thm:biadd}, we will now use the results of \cite{CLL_Spans} to prove the special case of our theorem where we assume that $I = E$, so that $P\subset I$ consists solely of truncated maps.

\begin{proposition}\label{prop:free-semiadd}
	Assume that $P\subset I$. For any $\myT$-category $\Xx$, the inclusion ${i}\colon\Xx\hookrightarrow\CIP{\Xx}{I}{I}{P}$ is the initial example of a map from $\Xx$ to an $(I,P)$-semiadditive $\myT$-category: for every $(I,P)$-semiadditive $\myT$-category $\Yy$, restriction along $i$ induces an equivalence
	\[
	{i}^*\colon \Fun_\myT^{{I\text-\amalg,P\text-\Pi}}(\CIP{\Xx}{I}{I}{P},\Yy)\iso\Fun_\myT(\Xx,\Yy).
	\]
\end{proposition}

\begin{remark}
Note that because $P\subset I$, a functor $F\colon \Xx\to \Yy$ between $(I,P)$-semiadditive $\myT$-categories which preserves $I$-colimits automatically preserves $P$-limits by \cite[Corollary~3.17]{CLL_Spans}. We will therefore drop the decoration `$P\text-\Pi$' to the functor category above for the remainder of this subsection.
\end{remark}

We begin by recalling that the case of the terminal $\myC$-category $\Xx = 1$ is the content of \cite[Theorem 5.43]{CLL_Spans}:
\begin{theorem}\label{thm:semiadd_on_a_point_P_in_I}
Assume that $P\subset I$. Then
\[
 \CIP{1}{I}{I}{P} = \Span(I_{/-},I_{/-}\times_IP, I_{/-})
\]
is the free $I$-cocomplete $P$-semiadditive $\myT$-category generated by the identity section.\qed
\end{theorem}

We will generalize this result to any $\myT$-category $\Xx$ by adapting the argument from \cite[Lemma C.4]{BachmannHoyois2021Norms} to the parametrized setting. This will require one further lemma:

\begin{lemma}\label{lemma:adjoin-univ-prop-N-pi}
	Write $\alpha\colon 1^{I\text-\amalg}\times\Xx\to\Xx^{I\text-\amalg}$ for the unique functor which is $I$-cocontinuous in the first variable (meaning that the corresponding map from $1^{I\text-\amalg}$ into the internal hom $\ul\Fun(\Xx,\Xx^{I\text-\amalg})$ is $I$-cocontinuous) and restricts to the inclusion ${j}\colon \Xx\hookrightarrow\Xx^{I\text-\amalg}$. Then restriction and left Kan extension along $\alpha$ define mutually inverse equivalences
	\[
		\alpha^*\colon\Fun_\myT^\textup{$I$-$\amalg$}(\Xx^{I\text-\amalg},\Yy)\leftrightarrows\Fun_\myT^\textup{$I$-$\amalg$, any}(1^{I\text-\amalg}\times\Xx,\Yy)\noloc \alpha_{\sharp}
	\]
	for every $I$-cocomplete $\myT$-category $\Yy$, where the right-hand side denotes the full subcategory of those functors that preserve $I$-colimits in the first variable.
	\begin{proof}
		Note first that $\alpha^*$ indeed lands in $\Fun_\myT^\textup{$I$-$\amalg$, any}(1^{I\text-\amalg}\times\Xx,\Yy)$ as $\alpha$ preserves $I$-colimits in the first variable. The resulting map is then an equivalence as its postcomposition with the equivalence
		\begin{equation}\label{eq:res-N-Pi}
			\res\colon\Fun_\myT^\textup{$I$-$\amalg$, any}(1^{I\text-\amalg}\times\Xx,\Yy)\iso\Fun_\myT(\Xx,\Yy)
		\end{equation}
		recovers the equivalence
		\begin{equation}\label{eq:res-kappa}
			j^*\colon \Fun_\myT^{I\text-\amalg}(\Xx^{I\text-\amalg},\Yy)\iso \Fun_\myT(\Xx,\Yy)
		\end{equation}
		which defines $\Xx^{I\text-\amalg}$.

		It remains to show that the left Kan extension along $\alpha$ exists on the subcategory $\Fun_\myT^\textup{$I$-$\amalg$, any}(1^{I\text-\amalg}\times\Xx,\Yy)$ and that it indeed lands in $\Fun_\myT^\textup{$I$-$\amalg$}(\Xx^{I\text-\amalg},\Yy)$. Both claims follow at once by recalling from Remark \ref{rmk:Free_N_Cocompletion} that the inverse to $(\ref{eq:res-kappa})$ is given by left Kan extension, as is the left adjoint to $(\ref{eq:res-N-Pi})$ by the special case $\Xx=1$ of the same remark.
	\end{proof}
\end{lemma}

\begin{proof}[Proof of Proposition~\ref{prop:free-semiadd}]
	We contemplate the following commutative diagram
	\begin{equation}\label{diag:THE-diagram}
		\begin{tikzcd}
			\Xx\arrow[dr, "{m}"{description},hook]\\
			&[-1em]1^{I\text-\amalg}\times\Xx\arrow[r, "\ell", hook]\arrow[d, "\alpha"'] & \CIP{1}{I}{I}{P}\times\Xx\arrow[d, "\beta"]\\
			&\Xx^{I\text-\amalg}\arrow[r, hook, "{k}"'] & \CIP{\Xx}{I}{I}{P}
		\end{tikzcd}
	\end{equation}
	where ${m},\ell,{k}$ are the canonical inclusions and $\alpha,\beta$ are the unique functors preserving $I$-colimits in the first variable which restrict on $\Xx = 1 \times \Xx$ to the canonical inclusions $j\colon \Xx \hookrightarrow \Xx^{I\text-\amalg}$ and $i\colon \Xx \hookrightarrow \CIP{\Xx}{I}{I}{P}$, respectively; here $\beta$ is well-defined by Theorem \ref{thm:semiadd_on_a_point_P_in_I}. We claim that for every $I$-cocomplete $\Yy$ restriction along $\beta$ defines an equivalence
	\[
		\beta^*\colon \Fun_\myT^{I\text-\amalg}(\CIP{\Xx}{I}{I}{P},\Yy) \iso \Fun_\myT^{I\text-\amalg,\text{any}}(\CIP{1}{I}{I}{P}\times\Xx,\Yy).
	\]
	Momentarily assuming this, we may immediately deduce the result: if $\Yy$ is even $(I,P)$-semiadditive, then also the restriction functor
	\[
	(\ell{m})^*\colon\Fun_\myT^{I\text-\amalg,\text{any}}(\CIP{1}{I}{I}{P}\times\Xx,\Yy)\to\Fun_\myT(\Xx,\Yy)
	\]
	is an equivalence by Theorem~\ref{thm:semiadd_on_a_point_P_in_I}, hence so is the composite ${i}^*=(\ell{m})^*\beta^*$ as desired. It thus remains to prove the claim, which we do in several steps.

	\textit{Step 1.} Assume first that $\Yy$ is a cocomplete $\myT$-category (with respect to a larger universe); in particular, $\alpha^*$ and $\beta^*$ admit left adjoints given by left Kan extension \cite[Theorem~7.1.13]{martiniwolf2021limits}. We will show that the Beck--Chevalley transformation $\alpha_{\sharp}\ell^*\to{k}^*\beta_{\sharp}$ of functors $\Fun_\myT(\CIP{1}{I}{I}{P}\times\Xx,\Yy)\to\Fun_\myT(\Xx^{I\text-\amalg},\Yy)$ is invertible. For this observe that the categories on the right of \eqref{diag:THE-diagram} carry parametrized factorization systems (i.e.~factorization systems in each degree, such that the restriction maps preserve both classes of the factorization systems): for $\CIP{\Xx}{I}{I}{P}$ we take the standard factorization system on spans consisting of the backwards and forwards maps \cite[Proposition~4.9]{HHLN2022TwoVariable}, while for $\CIP{1}{I}{I}{P}\times\Xx$ we take the standard factorization system on spans and the trivial factorization system on $\Xx$ where only the invertible maps belong to the left class. Note that $\ell$ and ${k}$ are precisely the inclusions of the right halves of these factorization systems. By \cite[Lemma~B.6${}^{\op}$]{LLP} it will therefore suffice to show that the restriction of $\beta$ to the left halves of these factorization systems is a (pointwise) right fibration. But unwinding definitions we see that for each $a \in \myT$ the map $\beta(a)$ sits in a commutative diagram
	\[
	\begin{tikzcd}[column sep=small]
		I_{/a}^P\times\core\Xx(a)\arrow[rr, "{\beta(a)}"]\arrow[dr, "\pr"'] && I_{/a}^P\times_{\myT}\Un^{\ct}(\Xx)_\text{ct}\arrow[dl,"\pr"]\\
		& I_{/a}^P,
	\end{tikzcd}
	\]
	where $I_{/a}^P \subset I_{/a}$ is the wide subcategory on maps in $P$, so the claim follows from left-cancellability of right fibrations.

	\textit{Step 2.} Let us now show that $\beta_{\sharp}$ restricts to a functor
	\begin{equation}\label{eq:beta-lower-*-complete}
		\Fun_{\myT}^{I\text-\amalg,\text{any}}(\CIP{1}{I}{I}{P}\times\Xx,\Yy)\to\Fun_\myT^{I\text-\amalg}(\CIP{\Xx}{I}{I}{P},\Yy),
	\end{equation}
	where $\Yy$ is still assumed to be cocomplete. As ${k}\colon\Xx^{I\text-\amalg}\hookrightarrow\CIP{\Xx}{I}{I}{P}$ is essentially surjective and $I$-cocontinuous, it will be enough to check that $k^*\beta_{\sharp}$ lands in $\Fun_\myT^{I\text-\amalg}(\Xx^{I\text-\amalg},\Yy)$. Since ${k}^*\beta_{\sharp}\simeq\alpha_{\sharp}\ell^*$ by Step 1 and since $\ell^*$ obviously restricts to $\Fun_\myT^{I\text-\amalg,\text{any}}\!$, this follows from Lemma~\ref{lemma:adjoin-univ-prop-N-pi}.

	\textit{Step 3.} We will now show that the map $(\ref{eq:beta-lower-*-complete})$ is actually an equivalence, with inverse $\beta^*$. Note first that $\ell^*\beta^*\simeq{k}^*\alpha^*$ is conservative as ${k}$ is essentially surjective and $\alpha^*$ is an equivalence (after restricting to functors satisfying suitable colimit preservation conditions) by Lemma~\ref{lemma:adjoin-univ-prop-N-pi}. We conclude that $\beta^*$ is conservative, and so it will suffice that the unit $\eta\colon\id \to \beta^*\beta_{\sharp}$ is an equivalence on $\Fun_\myT^{I\text-\amalg,\text{any}}(\CIP{1}{I}{I}{P}\times\Xx,\Yy)$. As $\ell$ is essentially surjective, it will be enough to show that $\ell^*\eta$ is invertible. By \cite[Lemma~C.2]{CLL_Adams} this map fits into a commutative diagram
	\[
		\begin{tikzcd}
			\ell^*\arrow[d,"\eta"']\arrow[r,"\ell^*\eta"] &[2em] \ell^*\beta^*\beta_{\sharp}\arrow[d,"\sim"]\\
			\alpha^*\alpha_{\sharp}\ell^*\arrow[r, "\alpha^*\BC_{\sharp}"', "\sim"]& \alpha^*{k}^*\beta_{\sharp}\rlap.
		\end{tikzcd}
	\]
	Since the unit of $\alpha_{\sharp}\dashv\alpha^*$ is invertible when fed an $I$-cocontinuous functor by another application of the previous lemma, the claim follows by $2$-out-of-$3$.

	\textit{Step 4.} We will now relax the assumption that $\Yy$ be cocomplete, and instead assume merely that $\Yy$ is $I$-cocomplete: we shall show that the restriction functor
	\begin{equation}\label{eq:beta-upper-*-equiv}
		\beta^*\colon \Fun_\myT^{I\text-\amalg}(\CIP{\Xx}{I}{I}{P},\Yy) \iso \Fun_\myT^{I\text-\amalg,\text{any}}(\CIP{1}{I}{I}{P}\times\Xx,\Yy)
	\end{equation}
	is an equivalence, thereby finishing the proof of the proposition. For this we will fully faithfully $I$-cocontinuously embed $\Yy$ into a cocomplete $\myT$-category $\Zz$, e.g.~via the parametrized coYoneda embedding (see \cite[Corollary~4.7.16]{martini2021yoneda} and \cite[Proposition~5.2.9${}^\op$]{martiniwolf2021limits}). We then only have to show that the equivalence
	\[\beta^*\colon \Fun_\myT^{I\text-\amalg}(\CIP{\Xx}{I}{I}{P},\Zz) \iso \Fun_\myT^{I\text-\amalg,\text{any}}(\CIP{1}{I}{I}{P}\times\Xx,\Zz)\]
	produced in Step 3 restricts to an equivalence \eqref{eq:beta-upper-*-equiv}. Let $F\colon\CIP{\Xx}{I}{I}{P}\to\Zz$ be an $I$-cocontinuous functor such that $F\beta$ (hence also $F{i}$) factors through $\Yy$. We need to show that $F$ itself factors through $\Yy$. As ${k}\colon\Xx^{I\text-\amalg}\hookrightarrow\CIP{\Xx}{I}{I}{P}$ is essentially surjective, it will be enough that $F{k}$ factors accordingly. But $F{k}$ is an $I$-colimit preserving functor extending $F{i}$, so the claim follows by applying the universal property of $\Xx^{I\text-\amalg}$ for functors into both $\Yy$ and $\Zz$.
\end{proof}

\subsection{\texorpdfstring{$\except{toc}{\bm{\CIP{\Xx}{E}{I}{P}}}\for{toc}{\CIP{\Xx}{E}{I}{P}}$}{Span P-fold, I} as a localization}\label{subsec:biadd-full}
We now begin with the proof of the general version of Theorem~\ref{thm:biadd}. In particular, throughout this section $I,P\subset E\subset \myT$ will be arbitrary wide subcategories satisfying the assumptions of Convention~\ref{conv:most-general}.

Observe that the special case considered in the previous subsection shows that $\CIP{\Xx}{I}{I}{I\cap P}$ is the free $(I,I\cap P)$-semiadditive $\myT$-category on $\Xx$. We will use this as a stepping stone towards understanding the free $(I,P)$-semiadditive $\myT$-category on $\Xx$. Note first that $\CIP{\Xx}{I}{I}{I\cap P}$ does not have the required $P$-limits, so as a first approximation to $\CIP{\Xx}{E}{I}{P}$ we may freely add $P$-limits to obtain the $\myT$-category $\big(\CIP{\Xx}{I}{I}{I\cap P}\big){}^{P\text-\Pi}\!$. Since $\CIP{\Xx}{E}{I}{P}$ is $P$-complete, there is a unique $P$-continuous functor
\[
\gamma\colon \big(\CIP{\Xx}{I}{I}{I\cap P}\big){}^{P\text-\Pi}\to \CIP{\Xx}{E}{I}{P}
\]
which restricts to the inclusion $\CIP{\Xx}{I}{I}{I\cap P}\hookrightarrow\CIP{\Xx}{E}{I}{P}$. Our key result is Theorem~\ref{thm:P-completion-SpanI} below, which shows that our first approximation is not too far off---it will turn out that $\gamma$ is a localization at an explicit class of maps.

In order to describe this class of maps, let us give an explicit description of the cocartesian unstraightening of $\big(\CIP{\Xx}{I}{I}{I\cap P}\big){}^{P\text-\Pi}$.

\begin{definition}
	Consider the full subcategory $\PAr_{I,P}(\myT) \subset \Fun([2],\myT)$ on those pairs of maps $(c \rightarrowmono b \rightarrowepic a)$ where the first map is in $I$ and the second map is in $P$ as indicated. Restriction along $d_1\colon [1] \to [2]$ defines a `composition map'
	\[
	\comp\colon \PAr_{I,P}(\myT) \to \Ar_E(\myT).
	\]

	We will display morphisms in $\PAr_{I,P}(\myT)$ as follows:
	\[
	\begin{tikzcd}
		c\arrow[r]\arrow[mono,d] & c'\arrow[d, mono]\\
		b\arrow[r]\arrow[d, epic] & b'\arrow[d,epic]\\
		a\arrow[r] & a'\!\rlap.\,
	\end{tikzcd}
	\]
	Similar to before we write $\PAr_{I,P}(\myT)_{P \cap I\textup{-pb}}$ for the wide subcategory on those morphisms for which the map $c' \to c \times_b b'$ is in $P \cap I$, and we write $\PAr_{I,P}(\myT)_{I\textup{-fw}}$ for the wide subcategory on those morphisms for which $a \to a'$ and $b \to b'$ are equivalences in $\myT$ (so that $c\to c'$ is in $I$ by left cancellability).
\end{definition}

The functor $\comp$ may be regarded as a functor over $C$ via the evaluation maps $\ev_2\colon \PAr_{I,P}(\myT) \to \myT$ and $\ev_1\colon \Ar_E(\myT) \to \myT$. Taking the pullback along $\Un^{\ct}(\Xx) \to \myT$ results in a map
\begin{equation}\label{eq:gamma-unspanned}
	\comp\times_{\myT}\id\colon \PAr_{I,P}(\myT)\times_{\myT} \Un^{\ct}(\Xx) \to \Ar_E(\myT)\times_{\myT} \Un^{\ct}(\Xx).
\end{equation}

\begin{lemma}\label{lem:description_of_gamma}
	The unstraightening of $\big(\CIP{\Xx}{I}{I}{I\cap P}\big){}^{P\text-\Pi}$ is given by
	\[
		\Span_{P\cap I\textup{-pb},I\textup{-fw}}\big({\PAr_{I,P}(\myT)}\big)\times_{\Span_I(\myT)}\Span_{\ct,\textup{all}}\big(\Un^\ct(\Xx)\big).
	\]
	Furthermore, the functor
	\[
	\Un^\cc(\gamma)\colon \Un^\cc\big(\CIP{\Xx}{I}{I}{I\cap P}^{P\text-\Pi}\big) \to \Un^{\cc}\big(\CIP{\Xx}{E}{I}{P}\big)
	\]
	is naturally equivalent to $\Span$ applied to the functor \eqref{eq:gamma-unspanned}.
\end{lemma}
More concretely, this means that the objects of the cocartesian unstraightening of $\big(\CIP{\Xx}{I}{I}{I\cap P}\big){}^{P\text-\Pi}$ are given by pairs $(X,c\rightarrowmono b \rightarrowepic a)$, where $X\in \Xx(c) \subset \Unct(\Xx)$, while morphisms are given by spans
\[\begin{tikzcd}[column sep=small]
		X && {X'} &[4em] {X''} \\[-2ex]
		c & {c\times_bb'} & {c'} & {c''} \\
		b && {b'} & {b'} \\
		a && {a'} & {a'\!\rlap.\,}
		\arrow["{\mathrm{ct}}"', from=1-3, to=1-1]
		\arrow[from=1-3, to=1-4]
		\arrow[mono, from=2-1, to=3-1]
		\arrow[from=2-2, to=2-1]
		\arrow["\llcorner"{very near start,xshift=5pt}, phantom, from=2-2, to=3-1]
		\arrow[mono, from=2-2, to=3-3]
		\arrow[both, from=2-3, to=2-2]
		\arrow[mono, from=2-3, to=2-4]
		\arrow[mono, from=2-3, to=3-3]
		\arrow[mono, from=2-4, to=3-4]
		\arrow[epic, from=3-1, to=4-1]
		\arrow[from=3-3, to=3-1]
		\arrow[equals, from=3-3, to=3-4]
		\arrow[epic, from=3-3, to=4-3]
		\arrow[epic, from=3-4, to=4-4]
		\arrow[from=4-3, to=4-1]
		\arrow[equals, from=4-3, to=4-4]
\end{tikzcd}\]
The functor $\Un^{\cc}(\gamma)$ is given on objects by sending $(X,c\rightarrowmono b \rightarrowepic a)$ to $(X, c \rightarrowepmo a)$.

\begin{proof}
	Using Remark~\ref{rmk:unstraigthening-of-free}, we may explicitly describe the cocartesian unstraightening of $\big(\CIP{\Xx}{I}{I}{I\cap P}\big){}^{P\text-\Pi}$ as the iterated pullback
	\[
	\Ar_P(\myT)^{\op}\times_{\myT^{\op}} \Span_{P\cap I\textup{-pb},I\textup{-fw}}\big({\Ar_I(\myT)}\big)\times_{\Span_I(\myT)}\Span_{\ct,\textup{all}}\big(\Un^\ct(\Xx)\big).
	\]
	Using that $\Yy^{\op}\simeq \Span(\Yy,\iota \Yy)$ for all $\Yy$, identifying $\Ar_P(\myT) \times_{\myT} \Ar_I(\myT) \simeq \PAr_{I,P}(\myT)$, and using that $\Span$ commutes with limits then gives the claimed description of this unstraightening. The same remark lets us identify the unstraightening of $\CIP{\Xx}{E}{I}{P}$ with the span category of $\Ar_E(\myT) \times_\myT \Un^{\ct}(\Xx)$. It thus remains to show that $\Un^{\cc}(\gamma)$ agrees with the map $\Span(\comp \times_{\myT} \id)$.

	By Remark~\ref{rmk:unstraigthening-of-free} the cocartesian edges of the target are precisely the maps of the form
	\[
	\begin{tikzcd}[column sep=small,row sep=small]
		X &[-.75em]&[-.5em]\arrow[ll,"\ct"'] Y\arrow[rr,equals] && Y\\[-.5ex]
		c\arrow[dd,epmo] &&\arrow[ll] c'\arrow[ddll,phantom,"\llcorner"{very near start}]\arrow[dd,epmo]\arrow[rr,equals] && c'\arrow[dd,epmo]\\[-1em]
		& \phantom\cdot\\
		a&&\arrow[ll] a' \arrow[rr,equals] && a'\!\rlap.\,
	\end{tikzcd}
	\]
	On the other hand, combining the same remark with Remark~\ref{rmk:(co)cart-edges-of-Pi} shows that the cocartesian edges of the source are precisely those of the form
	\[
	\begin{tikzcd}[column sep=small,row sep=small]
		X &[-.75em]&[-.5em]\arrow[ll,"\ct"'] Y\arrow[rr,equals] && Y\\[-.5ex]
		c\arrow[dd,mono] &&\arrow[ll] c'\arrow[ddll,phantom,"\llcorner"{very near start}]\arrow[dd,mono]\arrow[rr,equals] && c'\arrow[dd,mono]\\[-1em]
		& \phantom\cdot\\
		\arrow[dd,epic] b&&\arrow[ll]\arrow[dd,epic]\arrow[ddll,phantom,"\llcorner"{very near start}] b'\arrow[rr,equals] && b'\arrow[dd,epic]\\[-1em]
		& \phantom\cdot\\
		a&&\arrow[ll] a'\arrow[rr,equals] && a'\!\rlap.\,
	\end{tikzcd}
	\]
	It immediately follows that $\Span(\comp\times_{\myT}\id)$ preserves cocartesian edges over $\myT\catop$. Similarly one deduces from Remark~\ref{rmk:(co)cart-edges-of-Pi} together with Corollary~\ref{cor:cart-of-CIP} that it preserves cartesian edges over $P^\op$. Finally, it follows immediately from the definitions that its restriction to $\Un^{\cc}\big(\CIP{\Xx}{I}{I}{I\cap P}\big)$ is the inclusion. Since these three properties uniquely determine the unstraightening of $\gamma$ by the universal property of $\smash{(-)^{P\text-\Pi}}$, this finishes the proof of the lemma.
\end{proof}

With the explicit description of $\gamma$ at hand, we may now state the main result of this subsection.

\begin{definition}\label{def:inverted-morphisms}
	We write $\Ww$ for the class of maps in $\big(\CIP{\Xx}{I}{I}{I\cap P}\big){}^{P\text-\Pi}$ given in degree $a \in \myT$ by the spans
	\begin{equation}\label{diag:we}
		\begin{tikzcd}
			X\arrow[r,equals] & X\arrow[r,equals] & X\\[-2ex]
			c\arrow[d, mono] & \arrow[l, equals] c\arrow[d, mono]\arrow[r, equals] & c\arrow[d, mono]\\
			b\arrow[d, epic] & \arrow[l, epic]\arrow[d,epic] b'\arrow[r, equals] & b'\arrow[d,epic]\\
			a &\arrow[l,equals] a\arrow[r, equals] & a\rlap.
		\end{tikzcd}
	\end{equation}
	Left cancellability implies that the map $c\to c\times_{b} b'$ is in both $I$ and $P$, so this indeed defines a map in $\big(\CIP{\Xx}{I}{I}{I\cap P}\big){}^{P\text-\Pi}$.
\end{definition}

\begin{theorem}\label{thm:P-completion-SpanI}
	Let $\Xx$ be any $\myT$-category. Then the $\myT$-functor
	\[
	\gamma\colon \big(\CIP{\Xx}{I}{I}{I\cap P}\big){}^{P\text-\Pi}\to
	\CIP{\Xx}{E}{I}{P}
	\]
	is a localization at the class of maps $\Ww$.
\end{theorem}

The proof will occupy the remainder of the subsection. The key ingredient will be the fact that the composition functor $\comp\colon \PAr_{I,P}(\myT) \to \Ar_E(\myT)$ is a \textit{universal localization}, where we recall that a (non-parametrized) functor $F\colon \Xx\to \Yy$ is called a universal localization if for every $\Zz\to \Yy$ the base change $F\times_{\Yy} \Zz\colon \Xx\times_{\Yy} \Zz \to \Zz$ is a localization. We will employ the following sufficient criterion for universal localizations due to Hinich:

\begin{proposition}[Hinich's Key Lemma]\label{prop:Hinich_key_lemma}
	Let $F\colon\Xx\to\Yy$ be a functor such that $\Fun([n],F)$ has weakly contractible fibers for every $n\ge0$. Then $F$ is a universal localization.
	\begin{proof}
		This is the content of \cite[Lemma~1.5]{hinich-localization}; as this preprint has been retracted because of an unrelated error, we repeat Hinich's elegant argument for the reader's convenience.

		Observe first that if $\Zz\to\Yy$ is any functor, then $\Fun([n],\Xx\times_\Yy\Zz)\to\Fun([n],\Zz)$ is a base change of $\Fun([n],F)$ and hence again has weakly contractible fibers. Replacing $F$ by $\pr\colon\Xx\times_\Yy\Zz\to\Zz$, it will therefore suffice to show that $F$ is a localization.

		For this, we recall from \cite[Theorem~3.8]{mazel-gee-Rezk-nerve} that the localization of any relative category $(\Cc,\Ww)$ can be computed in terms of the \emph{Rezk nerve} $N_\text{R}(\Cc,\Ww)$, which is the simplicial space given in degree $n$ by the geometric realization (i.e.\ localization at all morphisms) of the wide subcategory $\Fun_\Ww([n],\Cc)\subset\Fun([n],\Cc)$ spanned by the morphisms which are pointwise in $\Ww$. More precisely, \emph{loc.\ cit.}\ shows that
		the unique left adjoint $L\colon\PSh(\Delta)\to\Cat$ extending $\Delta\hookrightarrow\Cat$ sends the inclusion $N_\text{R}(\Cc,\core\Cc)\hookrightarrow N_\text{R}(\Cc,\Ww)$ to a localization. Specializing to $\Cc=\Xx,\Ww=F^{-1}(\core\Yy)$ and using further that the right adjoint $\Cc\mapsto N_\text{R}(\Cc,\core\Cc)$ of $L$ is fully faithful, we are therefore reduced to showing that $L$ sends $N_\text{R}(F)\colon N_\text{R}(\Xx,F^{-1}(\core\Yy))\to N_\text{R}(\Yy,\core\Yy)$ to an equivalence.

		To this end, we will show that already $N_\text{R}(F)$ is an equivalence in $\PSh(\Delta)$, i.e.~each of the individual functors
		\begin{equation}\label{eq:hinich-levelwise}
			F\circ{-}\colon\Fun_{F^{-1}(\core\Yy)}([n],\Xx)\to\core\Fun([n],\Yy)
		\end{equation}
		is a localization. As the target is a groupoid, $(\ref{eq:hinich-levelwise})$ is a cartesian fibration by \cite[Proposition 3.3.1.8]{lurie2009HTT}, so it will suffice to show that its fibers are weakly contractible, see e.g.\ \cite[Lemma~5.5]{HHLN2022TwoVariable}. But $(\ref{eq:hinich-levelwise})$ is simply the base change of $\Fun([n],F)$ along $\core\Fun([n],\Yy)\hookrightarrow\Fun([n],\Yy)$, so this follows directly from our assumption.
	\end{proof}
\end{proposition}

Aiming to apply this result to the functor $\comp\colon \PAr_{I,P}(\myT) \to \Ar_E(\myT)$, we start by showing that the fibers of this functor are weakly contractible; we may interpret this as saying that any map in $E$ has a weakly contractible category of factorizations into a morphism in $I$ followed by a morphism in $P$. This seems to be a well-known result, and is also a special case of \cite[Theorem~4.21]{Liu_Zheng_Gluing}; for the reader's convenience, we provide a simple proof of this fact:

\begin{proposition}
	\label{prop:Decomp_Weakly_Contractible}
	For a morphism $e\colon c \rightarrowepmo a$ in $E$, the fiber $\comp^{-1}(e)$ is cofiltered, hence weakly contractible.
\end{proposition}

\begin{proof}
	By left cancellability of $E$, we see that we may identify $\comp^{-1}(e)$ with the fiber over $e$ of the analogous map $\comp\colon \PAr_{I,P}(E) \to \Ar(E)$. We will compare this to the fiber of the similarly defined composition map $\COMP \colon \PAr_{E,P}(E) \to \Ar(E)$ from the category of factorizations $c \rightarrowepmo b \rightarrowepic a$. Note that the fully faithful inclusion $\PAr_{I,P}(E) \subset \PAr_{E,P}(E)$ induces a fully faithful inclusion $\comp^{-1}(e) \subset \COMP^{-1}(e)$ on fibers over $e$.

	We will start by showing that $\COMP^{-1}(e)$ admits finite limits. For this, note that we may identify $\COMP^{-1}(e)$ with the full subcategory of the double slice $E_{c/\mskip-5mu/a} \coloneqq (E_{/a})_{e/}$ spanned by the objects of the form
	\[
	\begin{tikzcd}[column sep=small]
		c\arrow[rr, epmo]\arrow[dr, epmo, "e"'] && b\rlap,\arrow[dl, epic]\\
		& a
	\end{tikzcd}
	\]
	i.e.~those where the map to $a$ is contained in $P\subset E$. This clearly has a terminal object given by $(\begin{tikzcd}[cramped, column sep=small] c\arrow[r,epmo,"\smash{e}"] &[2pt] a\arrow[r,equals] & a\end{tikzcd})$. Moreover, the double slice $E_{c/\mskip-5mu/a}$ has pullbacks (computed in $E$), and it therefore only remains to show that $\COMP^{-1}(e)$ is stable under pullbacks. Unravelling definitions, the only non-trivial statement is that for a cospan
	\[
	\begin{tikzcd}
		b_1\arrow[d,epic]\arrow[r,epmo] & b_3\arrow[d,epic] & \arrow[l,epmo']  b_2\arrow[d,epic]\\
		a &\arrow[l,equals] a\arrow[r,equals] & a
	\end{tikzcd}
	\]
	the induced map $b_1\times_{b_3}b_2 \to a$ is again in $P$. By left cancellability the maps to $b_3$ are in fact in $P$, so this follows at once by observing that the map in question factors as
	$b_1\times_{b_3}b_2\rightarrowepic b_1\rightarrowepic a$,
	where the first map is in $P$ since it is a base change of $b_2\rightarrowepic b_3$. This shows that $\COMP^{-1}(e)$ has finite limits.

	To show cofilteredness of $\comp^{-1}(e)$, consider now a finite category $K$, and let $X\colon K \to \comp^{-1}(e) \subset \COMP^{-1}(e)$ be any diagram. By taking the limit in $\COMP^{-1}(e)$, we may extend $X$ to a functor $\overline{X}\colon K^{\triangleleft} \to \COMP^{-1}(e)$. We now claim that any object of $\COMP^{-1}(e)$ admits a map from an object of $\comp^{-1}(e)$: applying this to the image of the initial object of $K^{\triangleleft}$ will then produce the desired extension of $X$ to a diagram $K^\triangleleft\to\comp^{-1}(e)$.

	To prove the claim, consider any object of $\COMP^{-1}(e)$ consisting of the solid arrow part of the following diagram:
	\[
	\begin{tikzcd}
		c\arrow[r,equals,dashed]\arrow[mono,dashed,d] & c\arrow[d,epmo, "f"]\\
		b'\arrow[r, epic,dashed, "q"]\arrow[d, epic, dashed, "pq"'] & b\arrow[d,epic, "p"]\\
		a\arrow[r,equals,dashed] & a\rlap.
	\end{tikzcd}
	\]
	Factoring $f$ as a map in $I$ followed by a map in $P$ produces the dashed part, yielding the desired element of $\comp^{-1}(e)$.
\end{proof}

\begin{proposition}\label{prop:comp_univ_loc}
	The functor $\comp\colon \PAr_{I,P}(\myT) \to \Ar_E(\myT)$ is a universal localization.
\end{proposition}
\begin{proof}
	By Proposition~\ref{prop:Hinich_key_lemma} we need to show that the fibers of $\Fun([n],\comp)$ are weakly contractible for every $n\ge 0$. The case $n=0$ is the content of the previous proposition; we claim that the case for general $n$ is also an instance of the same proposition, this time applied to $\myT' \coloneqq \Fun([n],\myT)$ equipped with the wide subcategories
	\begin{align*}
		I' &\coloneqq \Fun([n],\myT) \times_{\myT^{\times (n + 1)}} I^{\times (n + 1)}\\
		P' &\coloneqq \Fun([n],\myT) \times_{\myT^{\times (n + 1)}} P^{\times (n + 1)}\\ 
		E' &\coloneqq \Fun([n],\myT) \times_{\myT^{\times (n + 1)}} E^{\times (n + 1)}.
	\end{align*}
	To see this, first note that the assumptions from Convention \ref{conv:most-general} are still satisfied: the only non-trivial fact is the decomposition of an arbitrary morphism
	\begin{equation}\label{diag:generic-E'}
		\begin{tikzcd}
			c_n\arrow[d,"\gamma_n"']\arrow[r,epmo,"\epsilon_n"] & a_n\arrow[d,"\alpha_n"]\\
			c_{n-1}\arrow[d,phantom,"\vdots"{near start}]\arrow[r,epmo,"\epsilon_{n-1}"] & a_{n-1}\arrow[d,phantom,"\vdots"{near start}]\\
			c_1\arrow[d,"\gamma_1"']\arrow[r,epmo,"\epsilon_1"] & a_1\arrow[d,"\alpha_1"]\\
			c_0\arrow[r,epmo,"\epsilon_0"] & a_0
		\end{tikzcd}
	\end{equation}
	in $E'$ as a composition of a morphism in $I'$ followed by one in $P'$. We will proceed inductively, starting by factoring $\epsilon_0$ as $c_0\rightarrowmono b_0\rightarrowepic a_0$. Assume now we have already constructed a factorization as depicted on the left in the following diagram:
	\[
		\begin{tikzcd}[cramped]
			c_{n}\arrow[d,phantom,"\vdots"{near start}]\arrow[rr,epmo,] && a_{n}\arrow[d,phantom,"\vdots"{near start}]\\
			c_k\arrow[d,"\gamma_k"']\arrow[rr,epmo,"\epsilon_k"] && a_k\arrow[d,"\alpha_k"]\\
			c_{k-1}\arrow[d,phantom,"\vdots"{near start}]\arrow[r,mono] & b_{k-1}\arrow[d,phantom,"\vdots"{near start}]\arrow[r,epic] & a_{k-1}\arrow[d,phantom,"\vdots"{near start}]\\
			c_1\arrow[d,"\gamma_1"']\arrow[r,mono] & b_1\arrow[d,"\beta_1"{description}]\arrow[r,epic] & a_1\arrow[d,"\alpha_1"]\\
			c_0\arrow[r,mono] & b_0\arrow[r,epic] & a_0
		\end{tikzcd}
		\qquad\qquad
		\begin{tikzcd}[cramped,row sep=small]
			c_k\arrow[dd,"\gamma_k"']\arrow[r,mono] & b_k\arrow[dd,"\beta_k"{description}, dashed]\arrow[dr,epic]\arrow[rr,epic,dashed] &[-1.5em]&[-1em] a_k\arrow[dd,"\alpha_k"]\\
			&& a_k\times_{a_{k-1}}b_{k-1}\arrow[dl]\arrow[ur,epic,yshift=-3pt]\arrow[dr,pullback,xshift=-1pt,yshift=1pt]
			\\
			c_{k-1}\arrow[r,mono] & b_{k-1}\arrow[rr,epic] && a_{k-1}
		\end{tikzcd}
	\]
	We then factor the induced map $c_k\to a_{k}\times_{a_{k-1}}b_{k-1}$ (which is in $E$ by left cancellability) as $c_k \rightarrowmono b_k \rightarrowepic a_{k}\times_{a_{k-1}}b_{k-1}$, and we define the maps $\beta_k\colon b_k\to b_{k-1}$ and $b_k\rightarrowepic a_k$ as the dashed composites depicted in the diagram on the right. This completes the inductive step and hence the construction of the desired factorization of $(\ref{diag:generic-E'})$ into a map in $I'$ followed by one in $P'$.

	Proposition~\ref{prop:Decomp_Weakly_Contractible} now shows that the fibers of $\PAr_{I',P'}(\myT') \to \Ar_{E'}(\myT')$ are weakly contractible. Since this map agrees with $\Fun([n],\comp)\colon \Fun([n],\PAr_{I,P}(\myT)) \to \Fun([n],\Ar_E(\myT))$, this finishes the proof.
\end{proof}

\begin{proof}[Proof of Theorem~\ref{thm:P-completion-SpanI}]
	Lemma~\ref{lem:description_of_gamma} shows that \[\gamma\colon(\CIP{\Xx}{I}{I}{I\cap P})^{P\text-\Pi}\to\CIP{\Xx}{E}{I}{P}\] is given in degree $a\in\myT$ by the functor of span categories
	\[
	\left\{\,
	\begin{tikzcd}[cramped]
		\cdot &\arrow[l,"\ct"']\cdot\arrow[r]&\cdot\\[-2ex]
		\cdot\arrow[d,mono] & \cdot\arrow[l,epic]\arrow[d,mono]\arrow[r,mono] & \arrow[d,mono]\\
		\cdot\arrow[d,epic] &\arrow[l,epic]\cdot\arrow[d,epic]\arrow[r,"\sim"] & \cdot\arrow[d,epic]\\
		a&\arrow[l,equals]a\arrow[r,equals]&a
	\end{tikzcd}
	\,\right\}\longrightarrow
	\left\{\,
	\begin{tikzcd}[cramped]
		\cdot &\arrow[l,"\ct"']\cdot\arrow[r]&\cdot\\[-2ex]
		\cdot\arrow[dd,epmo] & \cdot\arrow[l,epic]\arrow[dd,epmo]\arrow[r,mono] & \arrow[dd,epmo]\\
		\vphantom\cdot\\
		a&\arrow[l,equals]a\arrow[r,equals]&a
	\end{tikzcd}
	\,\right\}
	\]
	induced by $d_1\colon[1]\to[2]$. We will use the `Separation of Variables' criterion from \cite[Theorem~4.1.7]{CHLL_Bispans}, which reduces the theorem to proving the following two statements:
	\begin{enumerate}
		\item[(A)] The induced map on forward maps is a right fibration.
		\item[(B)] The induced map on backward maps is a localization at the maps of the form
		\begin{equation}\label{diag:we-repeated}
			\begin{tikzcd}
				\cdot&\arrow[l,equals]\cdot\\[-2ex]
				\cdot\arrow[d,mono]&\arrow[l,equals]\cdot\arrow[d,mono]\\
				\cdot\arrow[d,epic]\arrow[r,epic]&\cdot\arrow[d,epic]\\
				a&\arrow[l,equals] a\rlap.
			\end{tikzcd}
		\end{equation}
	\end{enumerate}
	For the first statement, observe that the map on forwards is a base change of the analogous map for $\Xx=1$, so we may assume without loss of generality that $\Xx$ is terminal. Consider then any map in $E_{/a}\times_EI$ (i.e.~the forwards maps of $\CIP{1}{E}{I}{P}$) as depicted on the left, together with a lift (i.e.~a factorization) of its target, assembling into the solid part of the diagram on the right:
	\[
	\begin{tikzcd}
		c\arrow[r, mono, "i"]\arrow[dd,epmo] & c'\arrow[dd,epmo] && c\arrow[r, mono, "i"]\arrow[d,mono,dashed, "ji"'] & c'\arrow[d,mono,"j"]\\
		& && b\arrow[r,equals,dashed]\arrow[d,epic,dashed,"p"'] & b\arrow[d,epic,"p"]\\
		a\arrow[r, equals] & a && a\arrow[r, equals,dashed] & \hskip1pta\rlap.\hskip1pt
	\end{tikzcd}
	\]
	The dashed part is then obviously the essentially unique lift of the given arrow.

	For the second claim we observe that the induced map on backward maps is the base change of $\comp\colon\PAr_{I,P}(\myT)\to \Ar_E(\myT)$ along
	$\Un^\ct(\Xx)_\ct\times_{\myT}E_{/a}\times_EP\rightarrowepic E_{/a}\to\Ar_E(\myT)$. Since the former is a universal localization by Proposition~\ref{prop:comp_univ_loc}, the map in question is a  localization. It clearly precisely inverts the maps of the form
	\[
	\begin{tikzcd}
		\cdot\arrow[r,"\sim"] &\cdot\\[-2ex]
		\cdot\arrow[r,"\sim"]\arrow[d,mono] &\cdot\arrow[d,mono]\\
		\cdot\arrow[r,epic]\arrow[d,epic] &\cdot\arrow[d,epic]\\
		a \arrow[r,equals] & a\rlap,
	\end{tikzcd}
	\]
	and as these are in turn precisely the maps factoring as a composite of an equivalence and a map of the form $(\ref{diag:we-repeated})$, the claim follows.
\end{proof}

\subsection{Proof of the general result}\label{sec:General_Parametrized_Span}
In this subsection we will finally prove the universal property of $\CIP{\Xx}{E}{I}{P}$. To set this up, observe that for any $(I,P)$-semiadditive $\myT$-category $\Yy$ and any functor $F\colon\Xx\to\Yy$ we can uniquely extend $F$ to an $I$-cocontinuous functor $\CIP{\Xx}{I}{I}{I\cap P}\to\Yy$, which we can then further extend to a $P$-continuous functor $\smash{\big(\CIP{\Xx}{I}{I}{I\cap P}\big){}^{P\text-\Pi}\to\Yy}$. Our first big goal will be to show that this extension descends through the localization $\big(\CIP{\Xx}{I}{I}{I\cap P}\big){}^{P\text-\Pi}\to\CIP{\Xx}{E}{I}{P}$ established in the previous subsection:

\begin{proposition}\label{prop:biadd-implies-invert}
	Let $\Yy$ be any $(I,P)$-semiadditive $\myT$-category, and let \[F\colon \big(\CIP{\Xx}{I}{I}{I\cap P}\big){}^{P\text-\Pi}\to\Yy\] be a $P$-continuous functor whose restriction to $\CIP{\Xx}{I}{I}{I\cap P}$ is $I$-cocontinuous. Then $F$ inverts the class $\Ww$ from Definition \ref{def:inverted-morphisms}.
\end{proposition}

\begin{remark}\label{rmk:int-why-cont-is-invert}
To build some intuition let us first explain why $F$ sends the source and target of a map
\[
\begin{tikzcd}
	X\arrow[r,equals] & X\arrow[r,equals] & X\\[-2ex]
	c\arrow[d,mono,"i"'] & \arrow[l,equals] c\arrow[r,equals]\arrow[d,mono,"j"] & c\arrow[d,mono]\\
	b\arrow[d,epic,"p"'] & \arrow[l,epic,"r"'] b'\arrow[r,equals]\arrow[d,epic,"q"] & b'\arrow[d,epic]\\
	a &\arrow[l,equals]a\arrow[r,equals] & a\rlap.
\end{tikzcd}
\]
in $\Ww$ to equivalent objects. We will for simplicity restrict to the case that $\Xx$ is terminal, the general case is similar. First note that the morphism above is given by applying $p_*\colon (\CIP{\Xx}{I}{I}{I\cap P})^{P\text-\Pi}(b)\to (\CIP{\Xx}{I}{I}{I\cap P})^{P\text-\Pi}(a)$ to the morphism
\[
\begin{tikzcd}
	c\arrow[d,mono,"i"'] & \arrow[l,equals] c\arrow[r,equals]\arrow[d,mono,"j"] & c\arrow[d,mono]\\
	b\arrow[d,equals] & \arrow[l,epic,"r"'] b'\arrow[r,equals]\arrow[d,epic,"r"] & b'\arrow[d,epic]\\
	b &\arrow[l,equals]b\arrow[r,equals] & b\rlap.
\end{tikzcd}
\] Since $F$ is $P$-continuous, it therefore suffices to only argue for the case that $p=\id$ and hence $q=r$. We now form the pullback
\begin{equation}\label{diag:pb-P-lim}
	\begin{tikzcd}
			c\arrow[dd, mono, "j"']\arrow[dr,dashed,shorten=-3pt,xshift=1pt,yshift=-1pt,"\delta"{description}]\arrow[rr,equals] &[-1em]& c\arrow[dd,mono, "i"]\\[-1em]
		& \,d\arrow[dr,pullback,xshift=-3pt,yshift=3pt]\arrow[dl,mono,"k"{description}]\arrow[ur,dashed,epic, "s"{description}]\\
		b'\arrow[rr,epic,"r"'] && b
	\end{tikzcd}
\end{equation}
and observe that $\delta$ belongs to $I\cap P$ by left cancellability. We may then compute
\[\hskip-9pt\hfuzz=9.02pt
	\smash{F(r,j)\simeq r_*F(\id,j)\simeq r_*k_{\sharp}F(\id,\delta)\overset{(*)}{\simeq} i_{\sharp}s_*F(\delta)\overset{(\dagger)}{\simeq} i_{\sharp}s_*\delta_*F(\id,\id)\simeq i_{\sharp}F(\id,\id)\simeq F(\id,i)}
\]
where $(*)$ uses the double Beck--Chevalley condition while $(\dagger)$ uses that the restriction of $F$ is also $(I\cap P)$-continuous by $(I\cap P)$-semiadditivity of $\CIP{\Xx}{I}{I}{I\cap P}$ \cite[Corollary~3.17]{CLL_Spans}.
\end{remark}

While the remark shows that $F(r,j)$ and $F(\id,i)$ are abstractly equivalent, showing that also the map $F(r,j) \to F(\id,i)$ induced by the given morphism in $\Ww$ is an equivalence is significantly more involved, and this requires some preparations.

\begin{construction}
	Assume we are given a commutative diagram
	\begin{equation*}
		\begin{tikzcd}[column sep=small]
			& c\arrow[dl,mono, bend right=10pt, "j"']\arrow[dr,mono,"i", bend left=10pt]\\
			b'\arrow[rr,epic,"r"'] && b
		\end{tikzcd}
	\end{equation*}
	in $\myT$, which we expand to the diagram $(\ref{diag:pb-P-lim})$ as before. Note once more that $\delta$ belongs to $I\cap P$; \cite[Construction~3.3]{CLL_Spans} therefore gives us for any $(I\cap P)$-semiadditive $\myT$-category a specific map $\mu\colon\id\to\delta_{\sharp}\delta^*$ that exhibits $\delta_{\sharp}$ also as \emph{right} adjoint to $\delta^*$.

	For any $I$-cocomplete $(I\cap P)$-semiadditive category $\Yy$, we then define a natural transformation $\atled\colon r^*i_{\sharp}\to j_{\sharp}$ as the composite
	\[
		r^*i_{\sharp}\xrightarrow[\smash{\raise3pt\hbox{\ensuremath{\scriptstyle\sim}}}]{\;\BC_{\sharp}^{-1}\;} k_{\sharp}s^*\xrightarrow{\;\mu\;} k_{\sharp}\delta_{\sharp}\delta^*s^*\xrightarrow[\smash{\raise3pt\hbox{\ensuremath{\scriptstyle\sim}}}]{\;\tau^*\;} k_{\sharp}\delta_{\sharp} \xrightarrow[\smash{\raise3pt\hbox{\ensuremath{\scriptstyle\sim}}}]{\;\sigma_{\sharp}\;} j_{\sharp}
	\]
	where the equivalences $\sigma_{\sharp}$ and $\tau^*$ are induced by the homotopies $k\delta\simeq j$ and $s\delta\simeq\id$, respectively, from $(\ref{diag:pb-P-lim})$.
\end{construction}

\begin{example}\label{ex:atled}
	Let $X\in\Xx(b')$ be arbitrary. Then one directly computes that the map $\atled\colon r^*i_{\sharp}(\id,X)\to j_{\sharp}(\id,X)$ in $\CIP{\Xx}{I}{I}{I\cap P}(c)$ is given by the span
	\[
	\begin{tikzcd}
		s^*X &\arrow[l,"\ct"']\delta^*s^*X\arrow[r, "\tau^*", "\sim"'] & X\\[-2ex]
		d\arrow[d,mono,"k"'] &\arrow[l,"\delta"',both] \arrow[d,mono,"j"]c\arrow[r,equal] & c\arrow[d,mono,"j"]\\
		b' &\arrow[l,equals]b'\arrow[r,equals] & b'\!\rlap.\,
	\end{tikzcd}
	\]
	Namely, as explained in \cite[Remark~4.9]{CLL_Spans}, the component of the natural transformation $\mu\colon \id\to \delta_{\sharp}\delta^*$ on the object $(\id_{d},s^*X)$ is given by the span
		\[
	\begin{tikzcd}[column sep=small,row sep=small]
		s^*X &[-.75em]&[-.5em]\arrow[ll,"\ct"'] \delta^*s^*X\arrow[rr,equals] &&  \delta^*s^*X\\[-.5ex]
		d\arrow[dd,equals] &&\arrow[ll, "\delta"', both] b'\arrow[dd, "\delta", both]\arrow[rr,equals] && b'\arrow[dd,"\delta", both]\\[-1em]
		& \phantom\cdot\\
		d&&\arrow[ll,equals] d\arrow[rr,equals] && d\rlap.
	\end{tikzcd}
	\]
	Applying $k_{\sharp}$, which postcomposes vertically by $k$ in the bottom span, and identifying appropriately we get the span above.
\end{example}

\begin{lemma}\label{lemma:biadd-atled-invertible}
	Assume $\Yy$ is even $(I,P)$-semiadditive. Then $\atled$ adjoins to an equivalence $i_{\sharp}\to r_*j_{\sharp}$.
	\begin{proof}
		Plugging in the definitions and appealing to 2-out-of-3, we have to show that the composite
		\[
			i_{\sharp}\xrightarrow{\;\eta\;} r_*r^*i_{\sharp} \xrightarrow[\smash{\raise3pt\hbox{\ensuremath{\scriptstyle\sim}}}]{\;\BC_{\sharp}^{-1}\;} r_*k_{\sharp}s^*\xrightarrow{\;\mu\;} r_*k_{\sharp}\delta_{\sharp}\delta^*s^*
		\]
		is an equivalence. We can forget that $\mu$ was a specific choice of a unit $\id\to\delta_{\sharp}\delta^*$ and can instead replace it by a generic unit map $\eta\colon\id\to\delta_*\delta^*$. Moreover, we can check the claim after precomposing with the equivalence $s_*\delta_*\simeq\id$; altogether we are therefore reduced to showing that the composite of the top row in the diagram
		\[
			\begin{tikzcd}
				\arrow[drr, dashed, bend right=15pt]i_{\sharp}s_*\delta_*\arrow[r, "\eta"] & {\color{blue}r_*r^*}i_{\sharp}s_*\delta_*\arrow[r,"\BC_{\sharp}^{-1}"] &  r_*k_{\sharp}{\color{red}s^*s_*}\delta_*\arrow[d,"\epsilon"]\arrow[r,"\eta"] & r_*k_{\sharp}{\color{blue}\delta_*\delta^*}{\color{red}s^*s_*}\delta_*\arrow[d,"\epsilon"]\\
				&& \arrow[dr, bend right=15pt, "="']r_*k_{\sharp}\delta_*\arrow[r,"\eta"] & r_*k_{\sharp}{\color{blue}\delta_*}{\color{violet}\delta^*}{\color{red}\delta_*}\arrow[d,"\epsilon"]\\
				&&& r_*k_{\sharp}\delta_*
			\end{tikzcd}
		\]
		is invertible; here we colored the adjoints introduced by the units or killed by the counits in blue and red, respectively, for clarity. The triangle commutes by the triangle identity, while the square commutes simply by naturality. The dashed composite is induced by the double Beck--Chevalley map $i_{\sharp}s_*\to r_*k_{\sharp}$ (by definition of the latter), so it is invertible by $(I,P)$-semiadditivity of $\Yy$. On the other hand, the composite of the rightmost column is invertible as it is a possible choice of counit for the adjoint equivalence $\delta^*s^*\dashv s_*\delta_*$. The claim follows by $2$-out-of-$3$.
	\end{proof}
\end{lemma}

\begin{lemma}
	Let $F\colon\Yy\to\Zz$ be any $I$-cocontinuous functor of $(I,I\cap P)$-semiadditive categories. Then we have a commutative diagram
	\begin{equation}\label{diag:atled-natural}
		\begin{tikzcd}
			r^*i_{\sharp}F\arrow[d,"\atled"']\arrow[r,"\BC_{\sharp}", "\sim"'] & r^*Fi_{\sharp}\arrow[r,"\sim"] & Fr^*i_{\sharp}\arrow[d,"F(\atled)"]\\
			j_{\sharp}F\arrow[rr,"\BC_{\sharp}"', "\sim"] && Fj_{\sharp}\rlap.
		\end{tikzcd}
	\end{equation}
	\begin{proof}
		We contemplate the diagram
		\[
			\begin{tikzcd}
				r^*i_{\sharp}F\arrow[drrrr,phantom,"(*)"]\arrow[d,"\BC_{\sharp}^{-1}"']\arrow[r,"\BC_{\sharp}"] & r^*Fi_{\sharp}\arrow[rrr,"\sim"] &&& Fr^*i_{\sharp}\arrow[d,"F(\BC_{\sharp}^{-1})"]\\
				k_{\sharp}s^*F\arrow[rr,"\sim"]\arrow[d,"\mu"'] && k_{\sharp}Fs^*\arrow[rr, "\BC_{\sharp}"]\arrow[dl,"\mu"{description}]\arrow[rd,"F(\mu)"{description}]\arrow[d,phantom,"(\dagger)"] && Fk_{\sharp}s^*\arrow[d,"F(\mu)"]\\
				k_{\sharp}\delta_{\sharp}\delta^*s^*F\arrow[r,"\sim"']\arrow[d,"\tau^*"']\arrow[drr,phantom,"(\ddagger)"] & k_{\sharp}\delta_{\sharp}\delta^*Fs^*\arrow[r,"\sim"'] &
				k_{\sharp}\delta_{\sharp}F\delta^*s^*\arrow[r,"\BC_{\sharp}"']\arrow[d,"F(\tau^*)"{description}] &
				k_{\sharp}F\delta_{\sharp}\delta^*s^*\arrow[r,"\BC_{\sharp}"'] &
				Fk_{\sharp}\delta_{\sharp}\delta^*s^*\arrow[d,"F(\tau^*)"]\\
				k_{\sharp}\delta_{\sharp}F\arrow[rr,equal]\arrow[d,"\sigma_{\sharp}"'] && k_{\sharp}\delta_{\sharp}F\arrow[r,"\BC_{\sharp}"']\arrow[d,"\sigma_{\sharp}"{description}]\arrow[rrd,phantom,"(\text{§})"] & k_{\sharp}F\delta_{\sharp}\arrow[r,"\BC_{\sharp}"'] & Fk_{\sharp}\delta_{\sharp}\arrow[d,"F(\sigma_{\sharp})"]\\
				j_{\sharp}F\arrow[rr,equal] && j_{\sharp}F \arrow[rr,"\BC_{\sharp}"'] && Fj_{\sharp}
			\end{tikzcd}
		\]
		where the left and rightmost column spell out the definitions of $\atled$ and $F(\atled)$, respectively.

		The subdiagram $(*)$ commutes by \cite[Lemma~2.2.4]{CSY2022TeleAmbi}, the diagram $(\dagger)$ commutes by \cite[Proposition~3.7(3)]{CLL_Spans}, commutativity of $(\ddagger)$ is an instance of functoriality of $F$, and $(\text{§})$ commutes by the same argument, combined with \cite[Lemma~C.5]{CLL_Adams}. As all the remaining squares commute by naturality, this then witnesses commutativity of $(\ref{diag:atled-natural})$.
	\end{proof}
\end{lemma}

\begin{proposition}\label{prop:cartesian-edges-out-of-thin-air}
	Let $F\colon\CIP{\Xx}{I}{I}{I\cap P}\to\Yy$ be an $I$-cocontinuous functor with $(I,P)$-semiadditive target, and consider its cocartesian unstraightening $\Un^\cc(F)\colon\Un^\cc\big(\CIP{\Xx}{I}{I}{I\cap P}\big)\to\Un^\cc(\Yy)$. Then $\Un^\cc(F)$ sends any span of the form
	\begin{equation}\label{eq:will-become-cart}
		\begin{tikzcd}
			X &\arrow[l,equals] X\arrow[r,equals] & X\\[-2ex]
			c\arrow[d,mono,"i"'] &\arrow[l,equals] c\arrow[d,mono,"j"]\arrow[r,equals] & c\arrow[d,mono,"j"]\\
			b&\arrow[l,epic,"r"]b'\arrow[r,equals] & b'
		\end{tikzcd}
	\end{equation}
	to a cartesian edge.
	\begin{proof}
		One directly computes that $(\ref{eq:will-become-cart})$ is the composite of a cocartesian lift of $r$ followed by the map $\atled$ (as computed in Example~\ref{ex:atled}). By the previous lemma, its image in $\Un^\cc(\Yy)$ then factors (up to equivalence) as the composite of a cocartesian lift of $r$ followed by the map $\atled\colon r^*i_{\sharp}Y\to j_{\sharp}Y$ in the fiber over $b'$ for some $Y\in\Yy(c)$. Factoring the image as a fiberwise map followed by a cartesian edge instead, the fiberwise component is then the adjoined map $i_{\sharp}Y\to r_*j_{\sharp}Y$, which is an equivalence by Lemma~\ref{lemma:biadd-atled-invertible}; thus the composite is cartesian, as claimed.
	\end{proof}
\end{proposition}

\begin{proof}[Proof of Proposition~\ref{prop:biadd-implies-invert}]
	Let us fix a generic map
	\begin{equation*}
		\begin{tikzcd}
			X &\arrow[l,equals] X\arrow[r,equals] & X\\[-2ex]
			c\arrow[d,mono,"i"'] &\arrow[l,equals] c\arrow[d,mono,"j"]\arrow[r,equals] & c\arrow[d,mono,"j"]\\
			b\arrow[d,epic,"p"']&\arrow[l,epic,"r"']\arrow[d,epic,"q"]b'\arrow[r,equals] & b'\arrow[d,epic,"q"]\\
			a &\arrow[l,equals] a\arrow[r,equals] & a\rlap.
		\end{tikzcd}
	\end{equation*}
	in $\Ww$. We first observe that this is the image under $p_*$ of the map
	\begin{equation}\label{diag:to-be-inverted}
		\begin{tikzcd}
			X &\arrow[l,equals] X\arrow[r,equals] & X\\[-2ex]
			c\arrow[d,mono,"i"'] &\arrow[l,equals] c\arrow[d,mono,"j"]\arrow[r,equals] & c\arrow[d,mono,"j"]\\
			b\arrow[d,equals]&\arrow[l,epic,"r"']\arrow[d,epic,"r"]b'\arrow[r,equals] & b'\arrow[d,epic,"r"]\\
			b &\arrow[l,equals] b\arrow[r,equals] & b\rlap.
		\end{tikzcd}
	\end{equation}
	Since the given $\myT$-functor $F$ is $P$-continuous, it then suffices to show that it inverts maps of the form $(\ref{diag:to-be-inverted})$, i.e.~we may assume that $p=\id$ and $r=q$.

	Consider the map $\Un^\cc\big(\CIP{\Xx}{I}{I}{I\cap P}^{P\text-\Pi}\big)\to\Un^\cc(\Yy)$ obtained by unstraightening $F$. The assumption that $F$ is $P$-continuous translates to saying that $F$ preserves cartesian edges over $P^\op$; in particular $F$ sends the map
	\begin{equation}\label{diag:actually-cartesian}
		\begin{tikzcd}
			X &\arrow[l,equals] X\arrow[r,equals] & X\\[-2ex]
			c\arrow[d,mono,"j"'] &\arrow[l,equals] c\arrow[d,mono,"j"]\arrow[r,equals] & c\arrow[d,mono,"j"]\\
			b'\arrow[d,epic,"r"']&\arrow[l,equals]\arrow[d,equals]b'\arrow[r,equals] & b'\arrow[d,equals]\\
			b &\arrow[l,epic,"r"] b'\arrow[r,equals] & b'\rlap.
		\end{tikzcd}
	\end{equation}
	to a cartesian edge. On the other hand, the previous proposition implies that $F$ also sends the map
	\begin{equation}\label{diag:wannabe-cartesian}
		\begin{tikzcd}
			X &\arrow[l,equals] X\arrow[r,equals] & X\\[-2ex]
			c\arrow[d,mono,"i"'] &\arrow[l,equals] c\arrow[d,mono,"j"]\arrow[r,equals] &c\arrow[d,mono,"j"]\\
			b\arrow[d,equals]&\arrow[l,epic,"r"']\arrow[d,equals]b'\arrow[r,equals] & b'\arrow[d,equals]\\
			b &\arrow[l,epic,"r"] b'\arrow[r,equals] & b'\rlap.
		\end{tikzcd}
	\end{equation}
	in $\Un^\cc\big(\CIP{\Xx}{I}{I}{I\cap P}^{P\text-\Pi}\big)$ to a cartesian edge. A direct computation now shows that we have a commutative diagram
	\[
		\begin{tikzcd}[column sep=small]
			(\id_a,i;X)\arrow[dr, bend right=10pt,"(\ref{diag:wannabe-cartesian})"']\arrow[rr,"(\ref{diag:to-be-inverted})"] && (r,j;X)\arrow[dl, bend left=10pt, "(\ref{diag:actually-cartesian})"]\\
			& (\id_c,j;X)
		\end{tikzcd}
	\]
	in $\Un^\cc\big(\CIP{\Xx}{I}{I}{I\cap P}^{P\text-\Pi}\big)$. By the above, the maps $(\ref{diag:actually-cartesian})$ and $(\ref{diag:wannabe-cartesian})$ are sent to cartesian edges. As $(\ref{diag:to-be-inverted})$ is a lift of the identity map $\id_b \colon b \to b$, it therefore has to be sent to an equivalence by uniqueness of cartesian edges.
\end{proof}

Thanks to the above, we can now extend any $F\colon\Xx\to\Yy$ with $(I,P)$-semiadditive target to $\bar F\colon\CIP{\Xx}{E}{I}{P}\to\Yy$. To show that this extension is indeed $(I,P)$-bicontinuous, we will need the following criterion:

\begin{lemma}\label{lemma:criterion-bicont}
	Let $\Xx$ be any $\myT$-category, let $\Yy$ be $(I,P)$-semiadditive, and let $F\colon\CIP{\Xx}{E}{I}{P}\to\Yy$ be any $\myT$-functor. Then $F$ is $(I,P)$-bicontinuous if and only if both of the following conditions are satisfied:
	\begin{enumerate}
		\item $F$ restricts to an $I$-cocontinuous functor $\CIP{\Xx}{I}{I}{I\cap P}\to\Yy$.
		\item $F$ is $P$-continuous.
	\end{enumerate}
	\begin{proof}
		The conditions are obviously necessary. To see that they are also sufficient, we only have to show that any $F$ satisfying these two conditions is in fact $I$-cocontinuous.

		For this let $j\colon a\rightarrowmono a'$ be a map in $I$. We have to show that for any $X\in \Xx(c)$ and any map $e\colon c\rightarrowepmo a$ in $E$, the Beck--Chevalley map $j_{\sharp}F(e,X)\to F(j_{\sharp}(e,X)) \simeq F((je,X))$ is an equivalence. We factor $je$ (which is again a map in $E$) as a map in $I$ followed by a map in $P$, yielding the solid part of the following commutative diagram:		
		\[			
		\begin{tikzcd}				c\arrow[dd, epmo, "e"']\arrow[dr,dashed,shorten=-3pt,xshift=1pt,yshift=-1pt,"\delta"{description}]\arrow[rr,mono] &[-1em]& b\arrow[dd,epic, "p"]\\[-1em]
				& d\arrow[dr,pullback,xshift=-3pt,yshift=3pt]\arrow[dl,dashed,epic,"q"{description}]\arrow[ur,dashed,mono, "k"{description}]\\
				a\arrow[rr,mono,"j"'] && a'\llap.
			\end{tikzcd}
		\]
		Forming the pullback of $j$ by $p$ and considering the induced map yields the dashed part of the diagram. The map $k$ is in $I$ as it is a pullback of $j$, so also $\delta$ belongs to $I$ by left cancellability. On the other hand, $q$ belongs to $P$ as it is a base change of $p$. Note that we may write $(e,X)\simeq q_*(\delta,X)$, where $(\delta,X)$ is now an object of $\CIP{\Xx}{I}{I}{I\cap P}$. Using this identification we can show the Beck--Chevalley map is an equivalence on this object. To motivate this, let us again prove something weaker first, namely that both sides are abstractly equivalent. For this we simply compute:
		\[
			\smash{j_{\sharp}F\big(q_*(\delta,X)\big)\simeq j_{\sharp}q_*F(\delta,X)\overset{(*)}{\simeq} p_*k_{\sharp}F(\delta,X)\simeq F\big(p_*k_{\sharp}(\delta,X)\big)\overset{(*)}{\simeq} F\big(j_{\sharp}q_*(\delta,X)\big),}
		\]
		where the identifications $(*)$ come from the double Beck--Chevalley equivalences in $\Yy$ and $\CIP{\Xx}{E}{I}{P}$, respectively.

		For the actual proof, we instead contemplate the diagram
		\[
			\begin{tikzcd}[cramped]
				j_{\sharp}q_*F\arrow[r, "F(\eta)"] & j_{\sharp}q_*F(k^*k_{\sharp})\arrow[r,"\sim"] &[1.5em] j_{\sharp}q_*k^*F(k_{\sharp})\arrow[r, "\BC_*^{-1}"] & j_{\sharp}j^*p_*F(k_{\sharp})\arrow[r, "\epsilon"] & p_*F(k_{\sharp})\\
				\arrow[u,"\BC_*^F"] j_{\sharp}F(q_*)\arrow[d,"\BC^F_{\sharp}"']\arrow[r, "F(\eta)"] & \arrow[urr,phantom,"(*)"]\arrow[u,"\BC_*^F"]j_{\sharp}F(q_*k^*k_{\sharp})\arrow[d,"\BC^F_{\sharp}"']\arrow[r,"F(\BC_*^{-1})"] & j_{\sharp}F(j^*p_*k_{\sharp})\arrow[r,"\sim"]\arrow[d,"\BC^F_{\sharp}"] &  \arrow[u,"\BC_*^F"]j_{\sharp}j^*F(p_*k_{\sharp})\arrow[r,"\epsilon"] & F(p_*k_{\sharp})\arrow[u,"\BC_*^F"']\\
				F(j_{\sharp}q_*)\arrow[r, "F(\eta)"'] & F(j_{\sharp}q_*k^*k_{\sharp})\arrow[r, "F(\BC_*^{-1})"'] & F(j_{\sharp}j^*p_*k_{\sharp})\arrow[urr,"F(\epsilon)"', bend right=10pt]\arrow[ur,phantom,"(\dagger)"{xshift=20pt}]
			\end{tikzcd}
		\]
		where we denote the Beck--Chevalley transformation commuting $F$ past a left adjoint by $\BC_{\sharp}^F$ for clarity, and $\BC_\ast^F$ for the dual case of commuting $F$ past a right adjoint. Here the subdiagram $(*)$ commutes by \cite[Lemma~2.2.4(2)]{CSY2022TeleAmbi}, the triangle $(\dagger)$ commutes by \cite[Lemma~2.2.3(4)]{CSY2022TeleAmbi}, and all the remaining little squares commute by naturality. Thus, the whole diagram commutes.

		We want to show that the lower left vertical map $j_{\sharp}F(q_*)\to F(j_{\sharp}q_*)$ is invertible when evaluated at the object $(\delta,X)$. The bottom composite $F(j_{\sharp}q_*)\to F(p_*k_{\sharp})$ precisely spells out the definition of $F(\BC_{{\sharp},*})$, so it is invertible, as is the Beck--Chevalley map $F(p_*k_{\sharp})\to p_*F(k_{\sharp})$ by assumption. By $2$-out-of-$3$ it will suffice to show the invertibility of the top left composite $j_{\sharp}F(q_*)\to j_{\sharp}q_*F\to p_*F(k_{\sharp})$ when evaluated at $(\delta,X)$. Note that the first map is again invertible by $P$-continuity of $F$, so it suffices that the second map (i.e.~the composite of the top row) is invertible.

		The proof will proceed by comparing this composite to the double Beck--Chevalley map in $\Yy$, for which we consider the following diagram:
		\[
			\begin{tikzcd}
				& & j_{\sharp}q_*k^*k_{\sharp}F\arrow[d,"\BC_{\sharp}^F"']\arrow[r,"\BC_*^{-1}"] &[.5em] j_{\sharp}j^*p_*k_{\sharp}F\arrow[r,"\epsilon"]\arrow[d,"\BC^F_{\sharp}"] & p_*k_{\sharp}F\arrow[d,"\BC_{\sharp}^F"]\\
				j_{\sharp}q_*F\arrow[r, "F(\eta)"']\arrow[rru,"\eta", bend left=15pt] & j_{\sharp}q_*F(k^*k_{\sharp})\arrow[r,"\sim"'] & j_{\sharp}q_*k^*F(k_{\sharp})\arrow[r, "\BC_*^{-1}"'] & j_{\sharp}j^*p_*F(k_{\sharp})\arrow[r, "\epsilon"']& p_*F(k_{\sharp})\rlap.
			\end{tikzcd}
		\]
		Here the triangle on the left again commutes by general facts about Beck--Chevalley maps, while the two squares commute by naturality. The composite $j_{\sharp}q_*F\to p_*k_{\sharp}F$ of the top row is the double Beck--Chevalley map in $\Yy$, hence invertible. On the other hand, the rightmost vertical map is invertible when evaluated at $(\delta,X)\in\CIP{\Xx}{I}{I}{I\cap P}\subset\CIP{\Xx}{E}{I}{P}$ by assumption. Thus, also the bottom composite $j_{\sharp}q_*F\to p_*F(k_{\sharp})$ is invertible when evaluated at $(\delta,X)$, as claimed.
	\end{proof}
\end{lemma}

Putting all the pieces together, we now get:

\begin{proof}[Proof of Theorem~\ref{thm:biadd}]
	By Theorem~\ref{thm:P-completion-SpanI}, precomposition with the canonical functor $\gamma$ induces an equivalence
	\[
		\Fun_\myT(\CIP{\Xx}{E}{I}{P},\Yy)\iso \Fun_\myT^{\Ww^{-1}}\big(\CIP{\Xx}{I}{I}{I\cap P}{}^{P\text-\Pi},\Yy\big)
	\]
	where the right-hand side denotes the full subcategory spanned by the functors inverting the maps in the collection $\Ww$ from Definition~\ref{def:inverted-morphisms}. As
	\[\gamma\colon \big(\CIP{\Xx}{I}{I}{I\cap P}\big){}^{P\text-\Pi}\to
	\CIP{\Xx}{E}{I}{P}\] is $P$-continuous and essentially surjective, this further restricts to an equivalence between the full subcategories of $P$-continuous functors on both sides. By commutativity of
	\[
		\begin{tikzcd}[column sep=small]
			&[-1em] \CIP{\Xx}{I}{I}{I\cap P}\arrow[dl,"{j}"', bend right=10pt]\arrow[dr,hook,bend left=10pt]\\
			(\CIP{\Xx}{I}{I}{I\cap P})^{P\text-\Pi}\arrow[rr,"\gamma"'] && \CIP{\Xx}{E}{I}{P}\rlap,
		\end{tikzcd}
	\]
	we then deduce that $\gamma^*$ restricts to an equivalence between
	\begin{enumerate}
		\item[(A)] the full subcategory of $\Fun_\myT(\CIP{\Xx}{E}{I}{P},\Yy)$ spanned by the $P$-continuous functors with $I$-cocont\-inuous restriction to $\CIP{\Xx}{I}{I}{I\cap P}$, \textit{and}
		\item[(B)] the full subcategory of $\Fun_\myT\big(\CIP{\Xx}{I}{I}{I\cap P}^{P\text-\Pi},\Yy\big)$ spanned by the $P$-continuous functors with $I$-cocontinuous restriction to $\CIP{\Xx}{I}{I}{I\cap P}$ and that invert the maps in $\Ww$.
	\end{enumerate}

	By Lemma~\ref{lemma:criterion-bicont}, the full subcategory (A) consists equivalently of the $(I,P)$-bicontinuous functors, while Proposition~\ref{prop:biadd-implies-invert} shows that the invertibility condition in (B) is vacuous, given the other two conditions.

	By the universal property of $(-)^{P\text-\Pi}$, restriction then induces an equivalence between the full subcategory (B) and $\Fun_\myT^{I\text-\amalg}(\CIP{\Xx}{I}{I}{I\cap P},\Yy)$. By Proposition~\ref{prop:free-semiadd}, this is further equivalent to $\Fun_\myT(\Xx,\Yy)$ via restriction along the inclusion. Altogether, we see that restriction along
	\[
		\begin{tikzcd}[column sep=small,cramped]
			\Xx\arrow[r,hook]&\CIP{\Xx}{I}{I}{I\cap P}\arrow[r,hook]&
			\big(\CIP{\Xx}{I}{I}{I\cap P}\big)^{P\text-\Pi}\arrow[r,"\gamma"]&[.25em]\CIP{\Xx}{E}{I}{P}
		\end{tikzcd}
	\]
	induces an equivalence $\Fun_\myT^{I\text-\amalg,P\text-\Pi}(\CIP{\Xx}{E}{I}{P},\Yy)\iso\Fun_\myT(\Xx,\Yy)$. As this composite is simply the inclusion $\Xx\hookrightarrow\CIP{\Xx}{E}{I}{P}$, the theorem follows.
\end{proof}

\section{Universality of span 2-categories}\label{sec:main-results}

In this section we will prove Theorem~\ref{thm:IntroMainTheorem} from the introduction on the universal property of 2-categories of spans.
 
\subsection{Span 2-categories}\label{subsec:span2}
We start by constructing our 2-categories of spans, following \cite{HaugsengSpans}.

Unlike everything else in this paper, this construction will rely on a specific model of 2-categories, and we begin by recalling this model:

\begin{definition}
	A \emph{double category} is a complete Segal object in complete Segal spaces, i.e.~a bisimplicial space
	\[
		\twoC\colon\Delta^\op\times\Delta^\op\to\Spc
	\]
	such that the simplicial spaces $\twoC(n,-),\twoC(-,n)\colon\Delta^\op\to\Spc$ are complete Segal spaces for every $[n]\in\Delta$. We will write $\twoC_{n}$ for the complete Segal space $\twoC(-,n)$, viewed as a 1-category.

	We will moreover call $\twoC(0,0)$ the space of \emph{objects}, $\twoC(1,0)$ the space of \emph{horizontal arrows}, and $\twoC(0,1)$ the space of \emph{vertical arrows}. We will refer to points of $\twoC(1,1)$ as \emph{squares}.
\end{definition}

\begin{remark}
	The above definition follows \cite{Nuiten2021Straightening}. It is a bit more restrictive than the usual definition of double categories, which would only require completeness in one of the two simplicial directions. However, as the above still captures all of our examples, while avoiding having to form completions of various Segal spaces down the line, it will be more convenient for our purposes.
\end{remark}

\begin{definition}
	A \emph{2-category} is a double category $\twoC$ such that the functor $\twoC(0,-)\colon\Delta^\op\to\Spc$ is constant. Equivalently, this means that for every $n\ge 0$ the degeneracy $\twoC_0\to\twoC_n$ is essentially surjective (hence an equivalence on groupoid cores).
\end{definition}

\begin{remark}
	In \cite{lurieGoodwillie}, Lurie shows that the category $\Cat_2$ of 2-categories in the sense of the previous definition admits various alternative presentations, for example via a model category of \emph{scaled simplicial sets} or via a model category of $(1,1)$-categories strictly enriched in simplicial sets (viewed through the eyes of the Joyal model structure). Moreover, Haugseng shows in \cite{haugseng-comparison} that $\Cat_2$ is also equivalent to the category of categories enriched (in the higher categorical sense \cite{GepnerHaugseng2015Enriched}) in categories.
\end{remark}

\begin{definition}
	Let $\twoC$ be a double category. The \emph{horizontal fragment} of $\twoC$ is the subobject $\twoD\subset\twoC$ such that $\twoD_n\subset\twoC_n$ is the essential image of the degeneracy map $\twoC_0\to\twoC_n$.
\end{definition}

It is not hard to check that the horizontal fragment is a 2-category, and that it is the \emph{maximal 2-subcategory} of $\twoC$, meaning that this construction gives rise to a right adjoint to the inclusion of 2-categories into double categories, see e.g.~\cite[proof of Lemma~A.2.1]{BlansBlom}.

We now turn to the construction of the 2-category of iterated spans.

\begin{definition}
We define $\LambdaRune_n$ to be the full subcategory of $\Tw[n]$ spanned by those objects $(i\leq j)$ such that $j\leq i+1$. In the picture of $\Tw[n]$ from Example~\ref{ex:picture_of_Tw}, $\LambdaRune_n$ corresponds to the bottom zig-zag, i.e.~for $n=2$ this is given by all objects but $(0\le 2)$.
\end{definition}

\begin{definition}[{\cite[Definition~5.8]{HaugsengSpans}}]
If $C$ is any category with pullbacks, we define $\SpanDbl(C)$ as the subobject of the bisimplicial space ${\hom(\Tw(-)\times\Tw(-),C)}$ given in degree $(m,n)$ by those functors $\Tw[m]\times\Tw[n]\to C$ that are right Kan extended from $\LambdaRune_m\times\LambdaRune_n$.
\end{definition}

\begin{remark}\label{rmk:make_cart_functs_explicit}
	By \cite[Lemma~5.6 and Remark~5.7]{HaugsengSpans}, the above right Kan extension always exists and $\SpanDbl(C)(m,n)$ consists equivalently of those $F\colon \Tw[m]\times\Tw[n]\to C$ such that for each $0\le i\le j\le m$ the functor $F(i\le j,-)\colon\Tw[n]\to C$ preserves pullbacks of maps in $\Tw[n]^\text{f}$ along maps in $\Tw[n]^\text{b}$, and similarly if we fix the second variable.

	In other words, if we write $\Fun^\ct(\Tw[m],C)$ for the full subcategory of functors preserving pullbacks of the above form, then $\SpanDbl(C)(m,n)$ consists precisely of those functors that curry to objects in $\Fun^\ct(\Tw[m],\Fun^\ct(\Tw[n],C))$ or equivalently to objects in $\Fun^\ct(\Tw[n],\Fun^\ct(\Tw[m],C))$.
\end{remark}

\begin{proposition}[{cf.~\cite[Proposition~5.14]{HaugsengSpans}}]\label{prop:SpanDbl-is-double-cat}
	If $C$ is any category with pullbacks, then $\SpanDbl(C)$ is a double category.
	\begin{proof}
		By symmetry it will be enough to show that $\SpanDbl(C)(-,m)$ is a complete Segal space for every $m\ge0$. By the previous remark, this is the simplicial space
		\[
			[n]\mapsto\Hom^\ct\big({\Tw[n]},\Fun^\ct(\Tw[m],C)\big).
		\]
		It now suffices to observe that for any category $D$ with pullbacks the functor $[n]\mapsto\Hom^\ct(\Tw[n],D)$ is a complete Segal space; namely, in light of the previous remark it is the complete Segal space defining the span category $\Span(D)$ as recalled in Construction~\ref{cons:Span}.
	\end{proof}
\end{proposition}

\begin{definition}
	Let $C$ be a category with pullbacks. We define $\SpanTwo(C)$ as the horizontal fragment of the double category $\SpanDbl(C)$.
\end{definition}

The above agrees with Haugseng's definition in \cite[Definition~5.16]{HaugsengSpans}; note that there the horizontal fragment is denoted by $U_\text{seg}$.

\begin{remark}
	Unravelling the definitions, the groupoid $\SpanTwo(C)(0,0)$ of objects of $\SpanTwo(C)$ agrees with the underlying groupoid of $C$. The hom space between two objects $x,y$ is given by the space of diagrams $x\gets z\to y$; finally, the space of 2-cells to another such map $x\gets z'\to y$ is given by the space of diagrams
	\[
		\begin{tikzcd}
			x\arrow[d,equals] & \arrow[l] z\arrow[r] & y\arrow[d,equals]\\
			x\arrow[d,equals] & \arrow[l]\arrow[u] z''\arrow[d]\arrow[r] & y\arrow[d,equals]\\
			x & \arrow[l]z'\arrow[r] & y\rlap.
		\end{tikzcd}
	\]
	Collapsing the identities in the left and right column, this is indeed given by a span of spans, as one would expect.
\end{remark}

For our main application we will need the following refined variant:

\begin{construction}
	Let $I,P\subset E\subset C$ be wide subcategories closed under base change.

	We fix a pullback-preserving embedding $C\hookrightarrow C'$ into a category with all pullbacks; for example, we could take the Yoneda embedding. We then define $\SpanDblP{C}{E}{I}{P}$ as the sub-double category of $\SpanDbl(C')$ with $\SpanDblP{C}{E}{I}{P}(1,1)$ given by the subspace of diagrams
	\begin{equation}\label{diag:squares-in-SPAN}
		\begin{tikzcd}
			\cdot & \arrow[l] \cdot\arrow[r,epmo] & \cdot\\
			\arrow[u]\cdot\arrow[d,epmo] & \arrow[dl,phantom,"\scriptstyle\text{$I$-pb}"]\arrow[l]\arrow[u]\cdot\arrow[d,epmo]\arrow[r,epmo]\arrow[ur,phantom,"\scriptstyle\text{$P$-pb}"] & \cdot\arrow[u]\arrow[d,epmo]\\
			\cdot &\arrow[l]\cdot\arrow[r,epmo] & \cdot
		\end{tikzcd}
	\end{equation}
	where all objects are in $C$, various maps belong to the subcategory $E$ as indicated, the notation `$P$-pb' means that the induced map into the pullback belongs to $P$, and similarly for `$I$-pb.' We omit the straightforward diagram chase that this is indeed a sub-double category, i.e.~that the above diagrams are closed under horizontal and vertical pasting (but cf.~Remark~\ref{rmk:unstraigthening-of-free}).

	We can equivalently describe this as a subobject of $(m,n)\mapsto\hom(\Tw[m]\times\Tw[n],C)$. In particular, the above is independent of the choice of embedding $C\hookrightarrow C'$, and this gives rise to a functor from $\iota_1\BAdjTrip$ into the 1-category of double categories (taking $E$ to be the subcategory generated by $I$ and $P$).

	We then define $\SpanTwoP{C}{E}{I}{P}$ as the horizontal fragment of this double category, yielding a functor $\iota_1\BAdjTrip\to\Cat_2$. In other words,  $\SpanTwoP{C}{E}{I}{P}$ is the subcategory of $\SpanTwo(C')$ with objects the objects of $C$, $1$-morphisms the spans of the form $x\gets z\rightarrowepmo y$, and 2-cells those of the form
	
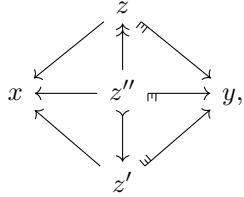
\[
		\begin{tikzcd}
			  & z\arrow[dl]\arrow[dr,epmo]\\
			x & \arrow[u,epic]\arrow[l]z''\arrow[r,epmo]\arrow[d,mono] & y\rlap,\\
			  & \arrow[ul]z'\arrow[ur,epmo']
		\end{tikzcd}
	\]
	i.e.\ all maps to $y$ are in $E$, the upward-pointing morphism is in $P$ and the downward-pointing morphism is in $I$.
\end{construction}

Having constructed the 2-category $\SpanTwoP{C}{E}{I}{P}$, we will now study it in more detail. We begin with its underlying category.

\begin{lemma}\label{lemma:span2-underlying-1}
The underlying 1-category of $\SpanTwoP{C}{E}{I}{P}$ is $\Span(C,E)$.
\end{lemma}

\begin{proof}
The underlying 1-category of $\SpanTwoP{C}{E}{I}{P}$ is obtained by restricting the bisimplicial space to $\Delta^\op\times \{[0]\}$, i.e.\ it is given by the sub-simplicial space of $[n]\mapsto\Hom^{\ct}(\Tw[n],C)$ spanned by the cartesian diagrams sending $\Tw[n]^\text{f}$ to $E$, i.e.\ the morphisms of adequate triples $\Tw[n] \to (C,C,E)$. This is precisely the complete Segal space defining $\Span(C,E)$ recalled in Construction~\ref{cons:Span}.
\end{proof}

We write $h$ for the composite
\[
\myC^{\op}\hookrightarrow \Span(C,E)\to \SpanTwoP{C}{E}{I}{P}.
\]
Our goal will be to exhibit $h$ as the universal $(I,P)$-biadjointable functor out of $C^{\op}$, whenever the assumptions of Convention~\ref{conv:most-general} are satisfied. To begin we show the following:

\newcommand{\cramp}[2]{\setbox0=\hbox{\ensuremath{#1}}\setbox1=\hbox{\ensuremath{#2}}\null\hskip.5\wd1\hskip-.5\wd0\copy0\hskip-.5\wd0\hskip.5\wd1\null}

\begin{proposition}
	\label{prop:Span2_Biadjointable}
	The inclusion $h\colon\myC\catop \hookrightarrow\SpanTwoP{\myC}{E}{I}{P}$ is $(I,P)$-biadjointable.
\end{proposition}
\begin{proof}
	A simple calculation shows that given a map $x\rightarrowmono y$ in $I$, the two 2-morphisms
	
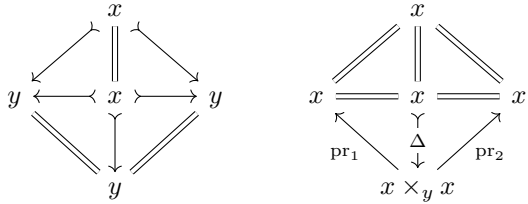
\[\begin{tikzcd}
		\vphantom{y}& x &&& x \\[.3ex]
		y & x & y & x & x & x \\[.3ex]
		& y &&& \cramp{x\times_yx}{x}
		\arrow[mono, from=1-2, to=2-1]
		\arrow[equals, from=1-2, to=2-2]
		\arrow[mono, from=1-2, to=2-3]
		\arrow[equals, from=1-5, to=2-4]
		\arrow[equals, from=1-5, to=2-6]
		\arrow[equals, from=2-1, to=3-2]
		\arrow[mono, from=2-2, to=2-1]
		\arrow[mono, from=2-2, to=2-3]
		\arrow[mono, from=2-2, to=3-2]
		\arrow[equals, from=2-3, to=3-2]
		\arrow[equals, from=2-5, to=1-5]
		\arrow[equals, from=2-5, to=2-4]
		\arrow[equals, from=2-5, to=2-6]
		\arrow[mono,"\Delta"{description}, from=2-5, to=3-5]
		\arrow["{\pr_1}", from=3-5, to=2-4]
		\arrow["{\pr_2}"', from=3-5, to=2-6]
	\end{tikzcd}\]
	provide a counit and unit, respectively, exhibiting the span $x= x \rightarrowmono y$ as a left adjoint of $y\leftarrowmono x= x$, see also \cite[Lemma~12.3]{HaugsengSpans} or \cite[Proposition~3.3.1]{Stefanich2020Correspondences}. The dual 2-morphisms show that maps in $P$ have right adjoints. To show that $h$ is left $I$-adjointable we compute that given a pullback square
	\[
	\begin{tikzcd}
		x' \dar[swap]{j}  \rar{g} \drar[pullback] & x \dar{i} \\
		y' \rar[swap]{f} & y
	\end{tikzcd}
	\]
	in $\myC$, the Beck--Chevalley transformation is given by the following composite (to be read from the top to the bottom)
	\[
		\begin{tikzcd}[column sep=scriptsize]
			& x'\arrow[dl, "j"']\arrow[dr,epmo,"g"]\\[.4em]
			y' &\arrow[l] x'\arrow[u,equals]\arrow[r,epmo]\arrow[d,mono,"{(g,\id)}"{description}] & x\\[.4em]
			&\arrow[ul, "j\circ\pr_2"]\arrow[dl,"j\circ\pr_2"']\arrow[dr, "\pr_1",epmo] x\times_yx'\arrow[ur,epmo',"\pr_1"']\\[.4em]
			y' &\arrow[l] x\times_y x'\arrow[r,epmo]\arrow[u,equals]\arrow[d,mono] & x\\[.4em]
			& \arrow[ul,"\pr_2"] x\times_y y'\arrow[ur,"\pr_1"',epmo']
		\end{tikzcd}
	\]
	of 2-morphisms in $\SpanTwoP{\myC}{E}{I}{P}$. We identify $x\times_y y'$ on the bottom with $x'$. Under this identification the bottom-most downwards-pointing map $x\times_y x'\to x'$ is simply $\pr_2$, and so the composite 2-morphism is given by the 1-morphism $\id_{x'} = \pr_2\circ (g,\id_{x'})$ in $C$, viewed as a downwards map. In particular, the Beck--Chevalley transformation is an equivalence, also see \cite[Proposition~3.4.6]{Stefanich2020Correspondences}. A dual calculation shows that $h$ is right $P$-adjointable. The calculation that the double Beck--Chevalley transformation for a square
	\[
	\begin{tikzcd}
		x' \dar[swap]{q}  \rar{j} \drar[pullback] & x \dar{p} \\
		y' \rar[swap]{i} & y
	\end{tikzcd}
	\]
	is an equivalence is similar.
\end{proof}

We next identify the $\HOM$-categories of $\SpanTwoP{C}{E}{I}{P}$, with the goal of applying Theorem \ref{thm:Recognizing_The_Universal_Biadjointable_Functor}. This will require a slight digression, which gives us a handle on the unstraightening of the functor
\[
\HOM_{\twoC}(-,-)\colon (\iota_1\twoC)^{\op}\times \iota_1\twoC\to \Cat
\]
of 1-categories associated to a 2-category $\twoC$.

\begin{definition}[{\cite[Definition~10.4.1.1]{GR2017studyDAG}, \cite[Definition~3.4.16]{Loubaton_effectivity}}]
Given a 2-category $\twoC$, we define the double category $\SQ^\lax(\twoC)$ via
\[
\SQ^{\lax}(\twoC)\colon \Delta^{\op}\times \Delta^{\op}\to \Spc, \quad (n,m)\mapsto \hom_{\Cat_2}([n]\boxtimes [m],\twoC),
\] where $-\boxtimes -$ denotes the Gray tensor product. We define the \emph{oplax arrow category} $\Ar^{\oplax}(\twoC)$ to be the 1-category $\SQ^{\lax}(\twoC)(1,-)$, i.e.~$\Ar^\oplax(\twoC)$ represents the functor $\hom([1]\boxtimes -,\twoC)\colon\Cat^\op\to\Spc$. Note that the two face maps give functors
\[
\Ar^{\oplax}(\twoC)\to \iota_1\twoC\times \iota_1\twoC.
\]
\end{definition}

\begin{theorem}[{\cite[Theorem 7.21]{HHLN2022TwoVariable}}]\label{thm:unstraighten-enriched-hom}
The functor $\Ar^{\oplax}(\twoC)\to \iota_1\twoC\times \iota_1\twoC$ is an orthofibration and it straightens to the restricted $\HOM$ functor
\[
\HOM_\twoC(-,-)\colon \iota_1 \twoC^{\op} \times \iota_1\twoC\to \Cat.
\]
In particular, pulling back to $\{A\}\times\iota_1\twoC$ for any $A\in\twoC$ yields the cocartesian unstraightening of the 1-functor $\HOM_\twoC(A,-)\colon\iota_1\twoC\to\Cat$.\qed
\end{theorem}

For the span 2-categories constructed above, we actually have a completely explicit description of the oplax arrow categories:

\begin{proposition}\label{prop:aroplax}
Let $I,P\subset E\subset C$ be closed under base change. Then there exists a natural equivalence
\[
	\Ar^\oplax{\big(\SpanTwoP{C}{E}{I}{P}\big)}\simeq \SpanDblP{C}{E}{I}{P} (1,-)
\]
of categories over $\Span(C,E)\times\Span(C,E)$.
\end{proposition}

To prove this, we will need some basics from double category theory.

\begin{definition}
	Let $\twoC$ be a double category, and let $f\colon x\to y$ be a {vertical morphism}, i.e.~an element of $\twoC_{0,1}$. A horizontal morphism $\bar f\colon x\to y$ (with the same source and target) is called a \emph{companion} of $f$ if there exist elements of the form
	\[
		\begin{tikzcd}
			x\arrow[r,equals]\arrow[d,equals] & x\arrow[d, "f"]
			&&
			x\arrow[d, "f"']\arrow[r, "\bar f"] & y\arrow[d,equals]
			\\
			x\arrow[r, "\bar f"'] & y
			&&
			y\arrow[r,equals] & y
		\end{tikzcd}
	\]
	in $\twoC_{1,1}$ such that the composites
	\[
		\begin{tikzcd}
			x\arrow[r,equals]\arrow[d,equals] & x\arrow[d, " f"{description}]\arrow[r, "\bar f"]
			& y\arrow[d,equals]
			\\
			x\arrow[r, "\bar f"'] & y\arrow[r,equals] & y
		\end{tikzcd}
		\qquad\qquad
		\begin{tikzcd}
			x\arrow[r,equals]\arrow[d,equals] & x\arrow[d, " f"]\\
			\arrow[d, " f"'] x\arrow[r, "\bar f"{description}] & y\arrow[d,equals]\\
			y\arrow[r,equals] & y
		\end{tikzcd}
	\]
	are the obvious degeneracies of $\bar f$ and $f$, respectively.
\end{definition}

\begin{example}[{cf.~\cite[Example~2.5.6]{Ruit2025Thesis}}]
	Let $I,P\subset E\subset C$ be as above and consider any span in $C$ of the form $x\gets z\rightarrowepmo y$, viewed as a vertical morphism in $\SpanDblP{C}{E}{I}{P}$. Then the same span, viewed as a horizontal morphism, is a companion as witnessed by the squares
	\[
		\begin{tikzcd}
			\vphantom{y}x &\arrow[l, equals] x\arrow[r, equals] & x\\
			\arrow[u,equals]\arrow[d,equals] x & \arrow[l]z\arrow[u]\arrow[d,equals]\arrow[r,equals] & \arrow[u]z\arrow[d,epmo]\\
			x & \arrow[l]z\arrow[r,epmo] & y
		\end{tikzcd}
		\qquad\qquad
		\begin{tikzcd}
			x & \arrow[l]z\arrow[r,epmo] & y\\
			z\arrow[u]\arrow[d,epmo] & z\arrow[l,equals]\arrow[u,equals]\arrow[d,epmo]\arrow[r,epmo] & y\arrow[u,equals]\arrow[d,equals]\\
			y & \arrow[l,equals] y\arrow[r,equals] & y
		\end{tikzcd}
	\]
	(these are indeed elements of $\big(\SpanDblP{C}{E}{I}{P}\big){}_{1,1}$ as the bottom left and top right squares are even honest pullbacks in both cases). Note that these witnesses are simply the obvious degenerate squares in $\Span(C,E)$.
\end{example}

\begin{proof}[Proof of Proposition~\ref{prop:aroplax}]
By the previous example, every vertical morphism in $\SpanDblP{C}{E}{I}{P}$ has a companion, and every horizontal morphism arises this way. By \cite[Proposition 3.4.23]{Loubaton_effectivity}, the $\SQ^\lax$-construction induces an equivalence onto the category of double categories with this property. As taking horizontal fragments is left inverse to $\SQ^\lax$ \cite[Proposition~2.2.52]{Ruit2025Thesis}, we in particular see that taking horizontal fragments is fully faithful when restricted to this subcategory. We conclude that the identity lifts uniquely to an equivalence
\[
	\SpanDblP{C}{E}{I}{P}\simeq\SQ^\lax\big(\SpanTwoP{C}{E}{I}{P}\big)
\]
and the claim follows by evaluating both sides at $(1,-)$.
\end{proof}

\subsection{Proof of Theorem~\ref{thm:IntroMainTheorem}}

Combining everything we have done so far, we obtain Theorem~\ref{thm:IntroMainTheorem} from the introduction:

\begin{theorem}\label{thm:main-theorem}
	Let $I,P\subset E\subset C$ be as in Convention~\ref{conv:most-general}. Then the inclusion
	\begin{equation*}
		h\colon C^\op\hookrightarrow\Span(C,E)\hookrightarrow\SpanTwoP{C}{E}{I}{P}
	\end{equation*}
	is the initial $(I,P)$-biadjointable functor, i.e.~restriction along $h$ induces an equivalence
	\[
		\FUN(\SpanTwoP{C}{E}{I}{P},\twoD)\iso\FUN_{(I,P)\dbadj}(C^\op,\twoD)
	\]
	of 2-categories for every 2-category $\twoD$.
\end{theorem}
\begin{proof}
	We have seen in Proposition~\ref{prop:Span2_Biadjointable} that $h$ is $(I,P)$-biadjointable. By the recognition result proven in Theorem~\ref{thm:Recognizing_The_Universal_Biadjointable_Functor}, it will therefore be enough to show that $h$ is essentially surjective and that for every $x\in C$ the composite
	\[
		\begin{tikzcd}[cramped]
			C^\op\arrow[r,hook,"h"]&[-.5em]\SpanTwoP{C}{E}{I}{P}\arrow[r,"{\HOM(x,-)}"] &[2em] \Cat
		\end{tikzcd}
	\]
	is the free $(I,P)$-biadjointable functor generated by the class $\id_x$.

	The first statement is clear. For the second one we use the  description of the orthocartesian unstraightening of $\HOM$ obtained in Proposition \ref{prop:aroplax}: it is given by the category of spans in (a subcategory of) $\Fun(\Tw[1],C)$ whose morphisms are of the form
		\[
			\begin{tikzcd}
				\cdot & \arrow[l] \cdot\arrow[r,epmo] & \cdot\\
				\arrow[u]\cdot\arrow[d,epmo] & \arrow[dl,phantom,"\scriptstyle\text{$P$-pb}"]\arrow[l]\arrow[u]\cdot\arrow[d,epmo]\arrow[r,epmo]\arrow[ur,phantom,"\scriptstyle\text{$I$-pb}"] & \cdot\arrow[u]\arrow[d,epmo]\\
				\cdot &\arrow[l]\cdot\arrow[r,epmo] & \cdot\rlap;
			\end{tikzcd}
		\]
		note that this diagram is flipped compared to the description of the squares in $\SpanDblP{C}{E}{I}{P}$ given above since we stick to the usual convention here that objects in an arrow category are displayed as vertical maps.

		In the above description, the structure maps to $\Span(C,E)\times\Span(C,E)$ are given by projecting to the top and bottom row. In particular, restricting to $\{x\}\times\Span(C,E)$ we arrive at the category of spans of the form
		\begin{equation}\label{diag:un-enr-hom-as-spans}
			\begin{tikzcd}
				x & \arrow[l,equals] x\arrow[r,equals] & x\\
				\arrow[u]\cdot\arrow[d,epmo] & \arrow[dl,phantom,"\scriptstyle\text{$P$-pb}"]\arrow[l]\arrow[u]\cdot\arrow[d,epmo]\arrow[r,mono] & \cdot\arrow[u]\arrow[d,epmo]\\
				\cdot &\arrow[l]\cdot\arrow[r,epmo] & \cdot\rlap;
			\end{tikzcd}
		\end{equation}
		in $C_{/x}\times_C\Ar_E(C)$. Pulling this back to $C^\op$ then precisely recovers the description of the unstraightening of the free $(I,P)$-biadjointable functor generated by $x=x=x$ given in Example~\ref{ex:unst_of_Span_of_slices}.

		It remains to check that the identity is indeed sent to $x=x=x$ under the above chain of equivalences. While unravelling all the definitions would be a pain (in particular for the equivalence from Theorem~\ref{thm:unstraighten-enriched-hom}), we can use the following trick: in the cocartesian unstraightening of $\iota\HOM(x,-)$ (which is just the slice $\Span(C,E)_{x/}$), the identity of $x$ corresponds to the initial object. As any equivalence preserves initial objects, it will therefore suffice to identify the initial object in the subfibration of $(\ref{diag:un-enr-hom-as-spans})$ spanned by the cocartesian edges. By \cite[Theorem~3.1]{HHLN2022TwoVariable}, this subcategory is precisely given by the maps of the form
		\[
			\begin{tikzcd}
				x & \arrow[l,equals] x\arrow[r,equals] & x\\
				\arrow[u]\cdot\arrow[d,epmo] & \arrow[dl,phantom,"\llcorner"{very near start}]\arrow[l]\arrow[u]\cdot\arrow[d,epmo]\arrow[r,equals] & \cdot\arrow[u]\arrow[d,epmo]\\
				\cdot &\arrow[l]\cdot\arrow[r,epmo] & \cdot
			\end{tikzcd}
		\]
		and it is then an easy exercise to check that $x=x=x$ is indeed initial in this category: the pullback condition becomes equivalent to the bottom middle vertical map being invertible, and we see that for a generic object $x\gets z\rightarrowepmo y$ the unique map from $x=x=x$ is given by
		\[
			\begin{tikzcd}[anchor=south, baseline=.575em]
				\vphantom{y}x &[-.1em] \arrow[l,equals] x\arrow[r,equals] &[-.1em] x\\[-.15em]
				\arrow[u,equals]x\arrow[d,equals] & \arrow[l] z\vphantom{y} \arrow[u]\arrow[d,equals]\arrow[r,equals] & z\arrow[u]\arrow[d,epmo]\\[-.15em]
				x &\arrow[l]z\arrow[r,epmo] & y\rlap.
			\end{tikzcd}\qedhere
		\]
\end{proof}

Let us separately record two special cases of this theorem. The first one concerns the universal property of the 2-category $\smash{\SpanThreeHalves}(C,I)\coloneqq\SpanTwoP{C}{I}{I}{\core C}$ where all our 2-cells must be forward maps. In this case we recover the universal property proven by Macpherson \cite{MacPherson2022Bivariant} and Stefanich \cite{Stefanich2020Correspondences}:

\begin{corollary}
	If $I \subset \myC$ is a left cancellable wide subcategory closed under pullbacks, then the inclusion $h\colon \myC\catop \hookrightarrow \SpanThreeHalves(\myC,I)$ is the initial left $I$-adjointable functor.\qed
\end{corollary}

As the other extreme, we can consider the case where $I=P=E=C$, giving an affirmative answer to \cite[Conjecture~2.25]{Ben-Moshe_Transchromatic}, see also \cite[Remark~4.2.5]{hopkinsLurie2013ambidexterity}:

\begin{corollary}
	If $\myC$ is a category with pullbacks such that every morphism is truncated, then the inclusion $h\colon \myC\catop \hookrightarrow \SpanTwo(\myC)$ is the initial biadjointable functor.\qed
\end{corollary}

\begin{remark}
	There is also another natural approach one could try to use to construct the universal biadjointable functor; we briefly sketch this construction, and then explain why it is not clear---albeit quite likely---that this has the correct universal property. For simplicity, we will restrict to the fully biadjointable case and assume that $C$ has finite limits, although this construction easily adapts to the general setting by similar embedding arguments as above.

	We proceed via the general formalism of \emph{categories of kernels} associated to 3-functor formalisms from \cite{HeyerMann}, which will produce a $\Cat$-enriched category for us. We begin by equipping $\Span(\myC)$ with the symmetric monoidal structure induced from the cartesian structure on $\myC$; by \cite[Proposition~2.4.1]{HeyerMann} this is a closed symmetric monoidal structure with internal homs given by $\ul\hom(x,y)=x\times y$, so we may view $\Span(C)$ as a category enriched over itself by \cite[Corollary~7.4.10]{GepnerHaugseng2015Enriched} or \cite[Example~C.1.12]{HeyerMann}. We then change the enrichment along the lax (symmetric) monoidal functor $\Span(C)\to\Cat, x\mapsto\Span(C_{/x})$ from \cite[Lemma~4.2.2]{HeyerMann}, yielding a $\Cat$-enriched category $\SpanTwo^\text{enr}(C)$. It is not hard to show that the underlying $1$-category of this recovers $\Span(C)$ (see in particular Lemma~C.3.5 of \emph{op.\ cit.}), and that the resulting inclusion $C^{\op}\hookrightarrow\SpanTwo^\text{enr}(C)$ is biadjointable. Therefore the universal property gives us a 2-functor $f\colon\SpanTwo(C)\to\SpanTwo^\text{enr}(C)$ that is the identity on objects. Moreover, by construction the individual morphism categories on the right-hand side are given by
	\begin{equation}\label{eq:hom_in_enr}
	\HOM_{\SpanTwo^{\text{enr}}}(x,y)\simeq\Span(C_{/x\times y})\simeq\Span(C_{/x}\times_CC_{/y})
	\end{equation}
	as one would want, and so the morphism categories on both sides of the functor $f$ are \emph{abstractly} equivalent. This makes it seem very likely that this 2-functor is in fact an equivalence, but this is not yet a rigorous proof.

	In order to show this, note that it suffices by our recognition principle to make the equivalence \eqref{eq:hom_in_enr} functorial in $x \in C\catop$. While the functoriality in $(\core C)^\op$ is automatic from the construction of $\SpanTwo^\text{enr}(C)$, we expect that the full functoriality in $x$ would require a discussion of the interaction between transfer of enrichments and the enriched Yoneda embedding, which does not yet seem to have appeared in the literature.
\end{remark}

\section{Lax monoidal universality}
\label{sec:Lax_Monoidal_Universality}

Assume that $\myC$ comes equipped with a symmetric monoidal structure such that $I$ and $P$ are closed under tensor products and such that the functor $\otimes\colon \myC \times \myC \to \myC$ preserves pullbacks along morphisms in $I \times I$ as well as pullbacks along morphisms in $P \times P$. We will show that in this case the universality of $\SpanTwoP{\myC}{E}{I}{P}$ extends to the symmetric monoidal and lax symmetric monoidal settings. The main application of this is Theorem \ref{thm:6FF}: the construction of 6-functor formalisms out of the data of suitable functors $C\catop \to \CMon(\Cat)$.

Throughout this section, we work with the convention that symmetric monoidal categories are commutative monoids in $\Cat$, which we define as product-preserving functors $\myC^{\otimes}\colon\Span(\Fin) \to \Cat$, denoted by $\textbf{n} \mapsto \myC^{\otimes}(\textbf{n})$ with $\textbf{n} \coloneqq \{1, \dots, n\}$; this is equivalent to the usual definition in terms of functors out of $\Fin_* \simeq \Span_{\mathrm{inj},\all}(\Fin)$ via restriction and right Kan extension, see \cite[Proposition~C.1]{BachmannHoyois2021Norms}. As \cite[Corollary 5.1.15]{BHS_Algebraic_Patterns} shows, also the natural notions of lax symmetric monoidal functors agree in both contexts. We similarly define symmetric monoidal 2-categories as commutative monoids in $\Cat_2$.

The assumptions on $I$ and $P$ guarantee that $\myC^{\otimes}$ refines to a product-preserving functor $(\myC^{\otimes},I^{\otimes},P^{\otimes})\colon\Span(\Fin)\to\BAdjTrip$, i.e.\ a commutative monoid in $\BAdjTrip$. It is clear from the construction that $\SpanTwo\colon\BAdjTrip\to\CatTwo$ preserves finite products, and so postcomposition defines a product-preserving functor
\begin{equation}\label{eq:SpanTwo-sym-mon-structure}
	\begin{aligned}
		\SpanTwoP{\myC}{E}{I}{P}^\otimes\colon\Span(\Fin)&\to\CatTwo
		\\\textbf{n}&\mapsto \SpanTwoP{\myC^\otimes(\textbf{n})}{E^{\otimes}(\textbf{n})}{I^{\otimes}(\textbf{n})}{P^{\otimes}(\textbf{n})},
	\end{aligned}
\end{equation}
turning $\SpanTwoP{\myC}{E}{I}{P}$ into a symmetric monoidal 2-category. By naturality, the inclusion $h\colon\myC^\op\to\SpanTwoP{\myC}{E}{I}{P}$ acquires the structure of a symmetric monoidal $2$-functor.

We will show that this inclusion is universal among \textit{lax symmetric monoidal biadjointable functors}, to be defined below. Our strategy is to reduce to the non-monoidal universality of span 2-categories. Since lax symmetric monoidal 2-functors between two symmetric monoidal 2-categories are defined in terms of their 2-cocartesian unstraightenings, the crucial ingredient for this reduction step is that the 2-cocartesian unstraightening of the functor $\SpanTwoP{\myC}{E}{I}{P}^\otimes\colon\Span(\Fin) \to\CatTwo$ is itself a span 2-category, namely of the cartesian unstraightening of $\myC\colon \Span(\Fin) \to \Cat$. We establish this in Section \ref{subsec:Unstraightening_Span_2_Cats}. We deduce in Section \ref{subsec:Lax_Monoidal_Universality} the lax symmetric monoidal universal property of $\SpanTwoP{\myC}{E}{I}{P}$. The application to 6-functor formalisms appears in Section \ref{sec:6FF}.

\subsection{Unstraightening span 2-categories}
\label{subsec:Unstraightening_Span_2_Cats}

As an auxiliary result, we show in this section that the cocartesian unstraightening of a diagram of span 2-categories is itself a span 2-category. We start with a brief recollection of 2-categorical (un)straightening and 2-cocartesian fibrations between 2-categories.

\begin{definition}
	Let $p\colon\twoC\to\twoD$ be a 2-functor. A 1-morphism $f\colon x\to y$ in $\twoC$ is called \emph{2-cocartesian} if the square
	\[
	\begin{tikzcd}
		\HOM_{\twoC}(y,z)\arrow[r, "-\circ f"]\arrow[d,"p"'] &[1em] \HOM_{\twoC}(x,z)\arrow[d,"p"]\\
		\HOM_{\twoD}(p(y),p(z))\arrow[r, "-\circ p(f)"'] & \HOM_{\twoD}(p(x),p(z))
	\end{tikzcd}
	\]
	of categories is a pullback for every $z\in\twoC$.

	We say that $p$ is a \emph{2-cocartesian fibration} if all of the following conditions are satisfied:
	\begin{enumerate}
		\item It admits all 2-cocartesian lifts of 1-morphisms, i.e.~for every $x\in\twoC$ and $\bar f\colon p(x)\to \bar y$ in $\twoD$ there exists a 2-cocartesian 1-morphism $f\colon x\to y$ such that $\bar f$ factors as $p(f)$ followed by an equivalence.
		\item For every $x,y\in\twoC$, the map $p\colon\HOM_{\twoC}(x,y)\to\HOM_{\twoD}(p(x),p(y))$ is a \emph{cartesian} fibration. We call the cartesian edges of this fibration \emph{cartesian 2-morphisms} (of $\twoC$).
		\item The cartesian 2-morphisms in $\twoC$ are stable under whiskering with arbitrary 1-morphisms.
	\end{enumerate}
	If $p$ is merely assumed to satisfy (2) and (3), we say $p$ is a \emph{homwise cartesian fibration}.
\end{definition}

\begin{construction}
	We write $(\CatTwo)_{/\twoD}^\text{cc}$ for the 1-subcategory spanned by the 2-cocartesian fibrations and those maps over $\twoD$ that preserve 2-cocartesian 1-morphisms as well as cartesian 2-morphisms. The $\CatTwo$-tensoring of $(\CatTwo)_{/\twoD}$ obviously restricts to a $\CatTwo$-tensoring of $(\CatTwo)_{/\twoD}^\text{cc}$, turning the latter into a locally full 3-subcategory of $(\CatTwo)_{/\twoD}$.
\end{construction}

\begin{proposition}[{\cite[Section 6]{Nuiten2021Straightening}}]
	There is a natural (un)straightening equivalence $(\CatTwo)_{/\twoD}^\textup{cc}\simeq\FUN(\twoD,\CatTwo)$ of 3-categories.
\end{proposition}

We now come to unstraightenings of diagrams of span 2-categories.

\begin{construction}
	\label{cons:Unstraightening_Biadjointable_Triple}
	Let $J$ be a category and let $\Cc\colon J\to\BAdjTrip$ be a functor such that for each $j \in J$ the triple $\Cc(j) = (C(j), I(j),P(j))$ satisfies the assumptions from Convention \ref{conv:most-general}. Composition with $\SpanTwo(-)$ results in a functor 
	\[
		\SpanTwo(\Cc(-)) := \SpanTwoP{C(-)}{E(-)}{I(-)}{P(-)}\colon J\to\CatTwo,
	\]
	together with a natural transformation $C(-)\catop \Rightarrow \SpanTwo(\Cc(-))$. Passing to unstraightenings results in a morphism in $(\CatTwo)_{/J}^\text{cc}$ of the form
	\[
		{\widetilde{C}}\catop \simeq \Un^{\cc}(C(-)\catop) \hookrightarrow \Un^{\cc}\!\big(\SpanTwo(\Cc(-))\big),
	\]
	where $\widetilde{C} := \Un^{\ct}(C)$ denotes the cartesian unstraightening of $C$. Let $p\colon \widetilde{C} \to J\catop$ denote the resulting cartesian fibration. We may similarly consider wide subcategories
	\[
		\widetilde{I},\widetilde{P},\widetilde{E} \; \subset \; \widetilde{C}\times_{J \catop} \iota J\catop \; \subset \; \widetilde{C},
	\]
	defined to be the cartesian unstraightenings of the functors $E(-),P(-),I(-)\colon \iota (J^{\op}) \to \Cat$. In other words, a morphism lies in one of the categories above if, up to equivalence, it lies in a single fiber $\widetilde{C}_j\simeq C(j)$ and belongs there to either $I(j)$, $P(j)$, or $E(j)$, respectively.
\end{construction}

\begin{lemma}
	\label{lem:Unstraightening_Biadjointable_Triple}
	In the situation of Construction \ref{cons:Unstraightening_Biadjointable_Triple}, the quadruple $(\widetilde{C},\widetilde{E},\widetilde{I},\widetilde{P})$ satisfies the assumptions from Convention \ref{conv:most-general}.
\end{lemma}
\begin{proof}
	It is clear that $\widetilde{I}$ and $\widetilde{P}$ are left cancellable and that morphisms in $\widetilde{E}$ are closed under composition. To show that $\widetilde{I}$ is closed under base change, consider a morphism $i\colon x \to y$ in $\widetilde{I}$ and an arbitrary morphism $\phi\colon y' \to y$ in $\widetilde{C}$. Letting $f\colon py' \to py$ be the image of $\phi$ in $J\catop$, we may factor $\phi$ as a composite $y' \to f^*y \to y$ of a fiberwise morphism followed by a cartesian morphism. By assumption on $\Cc$, the pullback $x' := f^*x \times_{f^*y} y'$ exists in the fiber $\widetilde{C}_{py} \simeq C(py)$, and the projection map $i'\colon x' \to y'$ lies in $I(py) \subset \widetilde{I}$. We claim that a pullback of $i$ and $\phi$ in $\widetilde{C}$ is given by the outer rectangle
	\begin{equation}
	\label{eq:Pullback_Unstraightening}
	\begin{tikzcd}
		x' \dar[swap]{i'} \rar & f^*x \dar{f^*i} \rar & x \dar{i} \\
		y' \rar & f^*y \rar & y.
	\end{tikzcd}
	\end{equation}
	Indeed, the right-hand square is easily seen to be a pullback square in light of the cartesianness of its two horizontal morphisms \cite[proof of Proposition 2.6]{HHLN2022TwoVariable}, and the left-hand square is a pullback square in $\widetilde{C}$ because the corresponding square in $\widetilde{C}_{py} = C(py)$ is a pullback which is preserved by the functor $C(py) \to C(z)$ for every morphism $z \to py$ in $J$. This shows that $\widetilde{I}$ is closed under base change, and that the inclusion $C(j) \hookrightarrow \widetilde{C}$ preserves pullbacks for all $j \in J$. The argument for $\widetilde{P}$ proceeds dually. The fact that $\widetilde{I} \cap \widetilde{P}$ consists of truncated morphisms now follows at once from the analogous assumption in each category $C(j)$.
\end{proof}

In the situation of Construction \ref{cons:Unstraightening_Biadjointable_Triple}, the functor $p\colon \widetilde{C} \to J$ sends each of the classes $\widetilde{I}$, $\widetilde{P}$ and $\widetilde{E}$ to equivalences, so induces a 2-functor
\[
	q := \SpanTwo(p) \colon \SpanTwoP{\widetilde{C}}{\widetilde{E}}{\widetilde{I}}{\widetilde{P}} \to \SpanTwoP{J\catop}{\core{J\catop}}{\core{J\catop}}{\core{J\catop}} \simeq J.
\]

\begin{lemma}
	\label{lem:Span2-cocart-fib}
	The 2-functor $q$ is a 2-cocartesian fibration. Moreover, all its 2-cocartesian 1-morphisms lie in the image of the inclusion ${\widetilde{C}}\catop \hookrightarrow \SpanTwoP{\widetilde{C}}{\widetilde{E}}{\widetilde{I}}{\widetilde{P}}$.
\end{lemma}
\begin{proof}
	Since $J$ is a 1-category, its hom categories are groupoids. Since any functor to a groupoid is a cartesian fibration with cartesian morphisms given by the equivalences, we conclude that $q$ is automatically a homwise cartesian fibration. For the 2-cocartesian lifts, consider a cartesian morphism $f\colon Y \to X$ in $\widetilde{C}$ with respect to $p\colon \widetilde{C} \to J\catop$. To finish the proof, it remains to show that the corresponding left-pointing span $X \xleftarrow{} Y \xrightarrow{=} Y$ in $\SpanTwoP{\widetilde{C}}{\widetilde{E}}{\widetilde{I}}{\widetilde{P}}$ is 2-cocartesian with respect to $q$, i.e.\ that for every $Z \in \widetilde{C}$ the square
	\[\begin{tikzcd}
		{\HOM_{\SpanTwoP{\widetilde{C}}{\widetilde{E}}{\widetilde{I}}{\widetilde{P}}}(Y,Z)} & {\HOM_{\SpanTwoP{\widetilde{C}}{\widetilde{E}}{\widetilde{I}}{\widetilde{P}}}(X,Z)} \\
		{\HOM_{J}(p(Y),p(Z))} & {\HOM_{J}(p(X),p(Z))}
		\arrow["{- \circ f}", from=1-1, to=1-2]
		\arrow[from=1-1, to=2-1, "q"']
		\arrow[from=1-2, to=2-2, "q"]
		\arrow["{-\circ p(f)}"', from=2-1, to=2-2]
	\end{tikzcd}\]
	is a pullback square of categories. Using the description of the hom categories of $\SpanTwo$ obtained in the proof of Theorem \ref{thm:main-theorem}, we may identify this with the outer square in the diagram
	\[\hskip-7.8pt\hfuzz=7.8pt\begin{tikzcd}[cramped]
		{\Span_{\widetilde{I},\widetilde{P}}(\widetilde{C}_{/Y})\times_{\Span_{\widetilde{I},\widetilde{P}}}(\widetilde{C})\Span_{\widetilde{I},\widetilde{P}}(\widetilde{E}_{/Z})} &[-1em] {\Span_{\widetilde{I},\widetilde{P}}(\widetilde{C}_{/X})\times_{\Span_{\widetilde{I},\widetilde{P}}(\widetilde{C})}\Span_{\widetilde{I},\widetilde{P}}(\widetilde{E}_{/Z})} \\
		{\core((J\catop)_{/p(Y)}) \times_{\core(J\catop)} \{Z\} } & {\core((J\catop)_{/p(X)}) \times_{\core(J)} \{Z\}},
		\arrow[from=1-1, to=1-2]
		\arrow[from=1-1, to=2-1]
		\arrow[from=1-2, to=2-2]
		\arrow[from=2-1, to=2-2]
	\end{tikzcd}\]
	where in the bottom row one should think of $\core(J\catop)$ as $\Span_{\core J\catop, \core J\catop}(J\catop)$ and similarly for $\core((J\catop)_{/p(Y)})$. Since span categories preserve pullbacks, it thus remains to show that $\widetilde{C}_{/Y} \simeq \widetilde{C}_{/X} \times_{(J\catop)_{/p(X)}} (J\catop)_{/p(Y)}$, which is immediate from the fact that $f$ is cartesian with respect to $p\colon \widetilde{C} \to J\catop$.
\end{proof}

We are interested in characterizing biadjointable functors out of $\widetilde{C}$. This will rely on the following auxiliary result on adjointable squares in 2-cocartesian fibrations.

\begin{lemma}\label{lem:adjointability_in_cocart_fib}
Let $p\colon \twoC \to J$ be a 2-cocartesian fibration and suppose 
\begin{equation}\label{eq:diagram_adj_cocart}
	\begin{tikzcd}
	x & y & z \\
	{x'} & {y'} & {z'}
	\arrow[tail, from=1-1, to=1-2]
	\arrow[from=1-1, to=2-1]
	\arrow[from=1-2, to=1-3]
	\arrow[from=1-2, to=2-2]
	\arrow[from=1-3, to=2-3]
	\arrow[tail, from=2-1, to=2-2]
	\arrow[from=2-2, to=2-3]
\end{tikzcd}\end{equation}
is a commutative diagram in $\twoC$ living over the diagram 
\[\begin{tikzcd}
	i & j & j \\
	i & j & j
	\arrow["f", from=1-1, to=1-2]
	\arrow[equals, from=1-1, to=2-1]
	\arrow[equals, from=1-2, to=1-3]
	\arrow[equals, from=1-2, to=2-2]
	\arrow[equals, from=1-3, to=2-3]
	\arrow["f", from=2-1, to=2-2]
	\arrow[equals, from=2-2, to=2-3]
\end{tikzcd}\]
in $J$ such that the tailed edges are 2-cocartesian. Then the left square is left adjointable if and only if $x\to x'$ has a left adjoint in $\twoC_i$. In this case, the right square is left adjointable in $\twoC_j$ if and only if the outer square is left adjointable in $\twoC$. Similarly, the dual statements concerning right adjointability hold.
\end{lemma}

\begin{proof}
We will write $\twoC(-)\colon J\to \CAT_2$ for the functor associated to $p\colon \twoC\to J$ by straightening. We begin with the proof of the first statement. By definition the morphism $x\to x'$ extends to a functor $\Adj \to \twoC(i)$. We can then consider the commutative diagram in $\CAT_2$ of the form
\[
\begin{tikzcd}
	\Adj \dar[swap] \rar[equal] & \Adj \dar \\
	\twoC(i) \rar & \twoC(j),
\end{tikzcd}
\]
By interpreting this as a natural transformation from $\const_{\Adj}$ to the restriction of $\twoC(-)$ to $\{f\}\colon [1]\to J$ and unstraightening we obtain  a map
\[
\begin{tikzcd}
	{[1]} \times \Adj \dar \rar[dashed] & \twoC \dar \\
	{[1]} \rar{\{f\}} & J.
\end{tikzcd}
\]
of 2-cocartesian fibrations. By unwinding the construction, the restriction of the top map to $[1] \times [1]$ encodes the left-hand square in the above decomposition, which by Proposition \ref{prop:corep-badj} is therefore left adjointable.

Now we prove the second statement. Suppose that the right square is left adjointable in $\twoC_j$, hence also in $\twoC$. Then the outer square is left adjointable, since such squares compose. On the other hand, suppose that the outer square is left adjointable. By Proposition \ref{prop:corep-badj} the outer square extends to a functor $F\colon [1] \times \Adj \to \twoC$, which we curry and consider as a functor $F\colon [1]\to \FUN(\Adj,\twoC)$. We note that the functor $\FUN(\Adj,\twoC)\to \FUN(\Adj,J)$ is again a 2-cocartesian fibration (\cite[Corollary 3.5.6]{AHM_FreeFibrations} for $\mathbb{K} = J$ with the maximal marking)  and that a natural transformation is cocartesian if and only if all of its components are. Therefore we may decompose $F$ as a composite of a cocartesian edge followed by a fiberwise edge, and obtain a functor $\bar{G}\colon[2]\times \Adj \to \twoC$ whose restriction to $\{0 \leq 2\} \times \Adj$ agrees with $F$. By uniqueness of cocartesian edges this must also extend the original diagram  \eqref{eq:diagram_adj_cocart}. We conclude that the right square extends to a functor $[1]\times \Adj \to \twoC$, and so is left adjointable by Proposition~\ref{prop:corep-badj} again.
\end{proof}

We can now establish a criterion for biadjointability of functors out of ${\widetilde{C}}\catop$ into a 2-cocartesian fibration over $J$. For this we require the following notation:

\begin{notation}\label{not:fibr_lax_tran}
Let $\twoC(-),\twoD(-)\colon J\to \CAT_2$ be two functors and let $\phi\colon \Un^\cc(\twoC)\to \Un^{\cc}(\twoD)$ be a functor over $J$. Given any morphism $f\colon i\to j$ in $J$ and any object $X\in \Un^\cc(\twoC)$ in the fiber over $i\in J$, we can factor the image of the cocartesian edge $X\rightarrowmono \twoC(f)(X)$ under $\phi$ into a cocartesian edge in $\Un^{\cc}(\twoD)$ followed by a fiberwise edge  
\[
\phi(X) \rightarrowmono \twoD(f)\phi(X) \xrightarrow{\,\phi_f\,} \phi\twoC(f)(X)
\]
which we denote by $\phi_f$. In fact, an easy argument shows that the edges $\phi_f$ assemble into a natural transformation of the form $\phi_f\colon \twoD(f)\phi_i \Rightarrow \phi_j\twoC(f)$, see \cite[Construction~3.2.4]{HHLNother} for an analogous construction.
\end{notation}

\begin{lemma}\label{lem:fiberwise_biadjointability}
	Let $\twoD\colon J \to \CAT_2$ be a functor. A functor $\phi\colon {\widetilde{C}}\catop \to \Un^\cc(\twoD)$ over $J$ is $(\widetilde{I},\widetilde{P})$-biadjointable if and only if the following two conditions are satisfied:
	\begin{enumerate}
		\item\label{it:fiberwise_biadj} The restriction $\phi_j\colon C(j)\catop \to \twoD(j)$ is $(I(j),P(j))$-biadjointable for every $j \in J$.
		\item\label{it:transition_biadj} For each morphism $f\colon j \to j'$ in $J$, the canonical natural transformation $\phi_f\colon \twoD(f) \circ \phi_j \Rightarrow \phi_{j'} \circ C(f)$ is $(I(j),P(j))$-biadjointable.
	\end{enumerate}
	Similarly, a natural transformation $\epsilon\colon \phi \Rightarrow \phi'$ between two $(\widetilde{I},\widetilde{P})$-biadjointable functors is $(\widetilde{I},\widetilde{P})$-biadjointable if and only if each fiberwise restriction $\epsilon_j\colon \phi_j \Rightarrow \phi_j'$ is $(I(j),P(j))$-biadjointable.
\end{lemma}

\begin{proof}
	Since morphisms in $\widetilde{I}$ and $\widetilde{P}$ lie in single fibers of $\widetilde{C}$, the condition that $\phi$ sends morphisms in $\widetilde{I}$ to right adjoints reduces to a fiberwise condition, which is captured by \eqref{it:fiberwise_biadj}. It remains to analyze the Beck--Chevalley conditions.

	We focus on left $\widetilde{I}$-adjointability first. By the proof of Lemma \ref{lem:Unstraightening_Biadjointable_Triple}, any pullback square in $\widetilde{C}$ of a morphism $i\colon x \to y$ in $\widetilde{I}$ can be decomposed as in \eqref{eq:Pullback_Unstraightening}:
	\[
	\begin{tikzcd}
		x' \dar[swap]{i'} \rar \drar[pullback] & C(f)(x) \dar[swap]{f^*i} \rar \drar[pullback] & x \dar{i} \\
		y' \rar & C(f)(y) \rar & y,
	\end{tikzcd}
	\]
	where $f\colon j' \to j$ is the image in $J$ of the maps $x' \to x$ and $y' \to y$, the left-hand square is a pullback in the fiber $C(py')$, and the horizontal morphisms in the right-hand square are $p$-cartesian. We may thus analyze the two squares individually.

	For the left-hand square, this is a pullback square in the fiber $C(py')$, so $\phi$ sends it to a left adjointable square in $\Un^\cc(\twoD)$ if and only if $\phi_{py'}$ sends it to a left adjointable square in $\Dd(py')$. This is exactly the Beck--Chevalley condition for the fiberwise left $I$-adjointability of $\phi_{py'}$, captured by \eqref{it:fiberwise_biadj}.

	For the right-hand square, observe that its image under $\phi$ decomposes as
	\begin{equation}\label{eq:decomp_of_cocart_biadj}
		\begin{tikzcd}
		{\phi_j(y)} & {\twoD(f)\phi_j(y)} & {\phi_{j'}(C(f)(y))} \\
		{\phi_j(x)} & {\twoD(f)\phi_j(x)} & {\phi_{j'}(C(f)(x)),}
		\arrow["{\phi_j(i)}"', from=1-1, to=2-1]
		\arrow[tail, from=1-1, to=1-2]
		\arrow[tail, from=2-1, to=2-2]
		\arrow["{\twoD(f)\phi_j(i)}"', from=1-2, to=2-2]
		\arrow["{\phi_f(y)}", from=1-2, to=1-3]
		\arrow["{\phi_f(x)}"', from=2-2, to=2-3]
		\arrow["{\phi_{j'}(C(f)(i))}", from=1-3, to=2-3]
	\end{tikzcd}\end{equation}
	where the tailed edges are 2-cocartesian. Observe that the condition that the right-hand square is vertically left adjointable for all $f\colon j\to j'$ is precisely condition \eqref{it:transition_biadj}, while the left-hand square is left adjointable by Lemma~\ref{lem:adjointability_in_cocart_fib}. Therefore the outer square is left adjointable, establishing left $\widetilde{I}$-adjointability.

	The argument for right $\widetilde{P}$-adjointability is dual. Finally, since every pullback of a morphism in $\widetilde{I}$ by a morphism in $\widetilde{P}$ lies in a single fiber $C(j)$, the mixed Beck--Chevalley condition reduces to the fiberwise case.

	This shows that conditions \eqref{it:fiberwise_biadj} and \eqref{it:transition_biadj} are sufficient. Conversely, suppose $\phi$ is $(\widetilde{I},\widetilde{P})$-biadjointable. Then condition \eqref{it:fiberwise_biadj} is immediate. For condition  \eqref{it:transition_biadj}, we note that $(\widetilde{I},\widetilde{P})$-biadjointability implies that the outer square in \eqref{eq:decomp_of_cocart_biadj} is left adjointable for all $i\colon x\to y$ in $\tilde{I}$ and $f\colon j\to j'$ in $J$. By Lemma~\ref{lem:adjointability_in_cocart_fib} the left square of \eqref{eq:decomp_of_cocart_biadj} is therefore also always left adjointable, proving that $\phi_f$ is $I(j)$-left adjointable for all $f\colon j\to j'$. Arguing dually for $\widetilde{P}$ completes the proof that condition  \eqref{it:transition_biadj} holds. 

	Finally for the statement about natural transformations it suffices to observe that since $\widetilde{I}$ and $\widetilde{P}$ consist solely of fiberwise morphisms, a natural transformation is $(\widetilde{I},\widetilde{P})$-biadjointable if and only if its restriction to each fiber is.
\end{proof}

We now come to the main result of this subsection. As before, we let $J$ be a category, and let $\Cc\colon J\to\BAdjTrip$ be a functor such that for each $j \in J$ the triple $\Cc(j) = (C(j), I(j),P(j))$ satisfies the assumptions from Convention \ref{conv:most-general}. We let $\widetilde{C} = \Un^{\ct}(C)$ denote the cartesian unstraightening of $C\colon J \to \Cat$, and let $\widetilde{I},\widetilde{P},\widetilde{E}\subset \widetilde{C}$ be the wide subcategories from Construction \ref{cons:Unstraightening_Biadjointable_Triple}.

\begin{theorem}
	\label{thm:Unstraightening-Span-2}
	The inclusion ${\widetilde{C}}\catop \hookrightarrow \Un^\cc\big(\SpanTwo(\Cc(-))\big)$ is $(\widetilde{I},\widetilde{P})$-biadjointable, and uniquely extends to an equivalence of 2-categories over $J$ of the form
	\[
		\SpanTwoP{\widetilde{C}}{\widetilde{E}}{\widetilde{I}}{\widetilde{P}} \iso \Un^\cc\big(\SpanTwo(\Cc(-))\big).
	\]
\end{theorem}
\begin{proof}
	For the first claim, we check that the inclusion $\phi\colon {\widetilde{C}}\catop \hookrightarrow \Un^\cc\big(\SpanTwo(\Cc(-))\big)$ satisfies the conditions of Lemma \ref{lem:fiberwise_biadjointability}. Condition \eqref{it:fiberwise_biadj} holds because the restriction to the fiber over $j$ is the inclusion $C(j)\catop \hookrightarrow \SpanTwo(\Cc(j))$, which is $(I(j),P(j))$-biadjointable by Proposition~\ref{prop:Span2_Biadjointable}. For condition \eqref{it:transition_biadj}, consider a morphism $f\colon j \to j'$ in $J$. The 2-cocartesian pushforward $f_!$ on $\Un^\cc\big(\SpanTwo(\Cc(-))\big)$ is given by the 2-functor $f^*\colon \SpanTwo(\Cc(j)) \to \SpanTwo(\Cc(j'))$. Since the inclusion $C(j)\catop \hookrightarrow \SpanTwo(\Cc(j))$ is natural in $j$, the canonical natural transformation $\phi_f$ is an equivalence, and hence automatically $(I(j),P(j))$-biadjointable.

	By the universal property, the inclusion uniquely extends to a 2-functor
	\[
		\Phi\colon \SpanTwoP{\widetilde{C}}{\widetilde{E}}{\widetilde{I}}{\widetilde{P}} \to \Un^\cc\big(\SpanTwo(\Cc(-))\big).
	\]
	Since the functor $\phi$ lived over $J$, the universal property implies that $\Phi$ also lives over $J$. To finish the theorem, we must show $\Phi$ is an equivalence. By Lemma \ref{lem:Span2-cocart-fib}, the left-hand side is a 2-cocartesian fibration over $J$, and so is the right-hand side by construction. Moreover, since all 2-cocartesian morphisms on the left lie in the image of the inclusion from ${\widetilde{C}}\catop$ and the inclusion ${\widetilde{C}}\catop \hookrightarrow \Un^\cc\big(\SpanTwo(\Cc(-))\big)$ preserves 2-cocartesian morphisms by construction, the same then follows for its extension to $\SpanTwoP{\widetilde{C}}{\widetilde{E}}{\widetilde{I}}{\widetilde{P}}$. Since all cartesian 2-morphisms on both sides are equivalences, these are preserved by $\Phi$, showing that $\Phi$ is a morphism of 2-cocartesian fibrations over $J$.

	By straightening-unstraightening, it remains to show that $\Phi$ induces an equivalence on fibers over each object $j \in J$. Since $\SpanTwoP{-}{-}{}{}$ preserves pullbacks, the fiber of the left-hand side is the span 2-category $\SpanTwoP{C(j)}{E(j)}{I(j)}{P(j)}$, which is also the fiber over $j$ of the right-hand side. Since the comparison functor $\SpanTwoP{C(j)}{E(j)}{I(j)}{P(j)} \to \SpanTwoP{C(j)}{E(j)}{I(j)}{P(j)}$ is an extension of the inclusion $C(j)\catop \hookrightarrow \SpanTwoP{C(j)}{E(j)}{I(j)}{P(j)}$, it follows by the universal property that it is the identity, hence in particular an equivalence.
\end{proof}

\begin{corollary}\label{cor:lax_tran_out_of_span}
Let $\Cc(-) = (C(-),I(-),P(-))\colon J\to \BAdjTrip$ and $\twoD(-)\colon J\to \CAT_2$ be two functors. Restriction along the inclusion
\[
h\colon \Un^{\cc}(C(-)\catop) \hookrightarrow \Un^{\cc}(\SpanTwo(\Cc(-)))
\]
induces a locally full inclusion
\[\hskip-31.2pt\hfuzz=31.2pt
h^*\colon \HOM_{(\CAT_2)_{/J}}\big(\Un^{\cc}(\SpanTwo(\Cc(-))), \Un^{\cc}(\twoD(-))\big) \hookrightarrow \HOM_{(\CAT_2)_{/J}}\big(\Un^{\cc}(C(-)\catop), \Un^{\cc}(\twoD(-))\big).
\]
Its image is spanned
\begin{enumerate}
	\item\label{it:biad_tran} on objects by those functors $\phi\colon \Un^{\cc}(C(-)\catop) \to \Un^{\cc}(\twoD(-))$ over $J$ such that:
	\begin{enumerate}
		\item\label{it:biad_tran_one} For each $j \in J$, the restriction $\phi_j\colon C(j)\catop \to \twoD(j)$ to the fibers over $j$ is $(I(j),P(j))$-biadjointable,
		\item\label{it:biad_tran_two} For each morphism $f\colon j\to j'$ in $J$, the canonical natural transformation $\phi_f\colon \twoD(f) \phi_j \Rightarrow \phi_{j'} C(f)$ is $(I(j),P(j))$-biadjointable.
		\end{enumerate}
	\item\label{it:biadj_mod} on 1-morphisms by those natural transformations $\epsilon \colon \phi \Rightarrow \phi'$ such that the restriction to each fiber $\epsilon_j\colon \phi_j\Rightarrow \phi_{j}'$ is $(I(j),P(j))$-biadjointable for all $j\in J$.
\end{enumerate}
\end{corollary}
\begin{proof}
We continue using the notation of Construction \ref{cons:Unstraightening_Biadjointable_Triple}. By Theorem \ref{thm:Unstraightening-Span-2}, the inclusion ${\widetilde{C}}\catop \hookrightarrow \Un^{\cc}(\SpanTwo (\Cc(-)))$ induces an equivalence
\[
\SpanTwoP{\widetilde{C}}{\widetilde{E}}{\widetilde{I}}{\widetilde{P}} \iso \Un^{\cc}(\SpanTwo (\Cc(-))).
\]
From the universal property of the left-hand side, we deduce that restricting along
\[
h\colon {\widetilde{C}}\catop\simeq \Un^{\cc}(C(-)^{\op}) \hookrightarrow \Un^{\cc}(\SpanTwo (\Cc(-)))
\]
defines a locally full inclusion
\[
\HOM_{(\CAT_2)_{/J}}(\Un^{\cc}(\SpanTwo(\Cc(-))), \Un^{\cc}(\twoD(-))) \hookrightarrow \HOM_{(\CAT_2)_{/J}}({\widetilde{C}}\catop, \Un^{\cc}(\twoD(-)))
\]
whose image consists of the $(\widetilde{I},\widetilde{P})$-biadjointable functors ${\widetilde{C}}\catop  \to \Un^{\cc}(\twoD(-))$ over $J$ and $(\widetilde{I},\widetilde{P})$-biadjointable natural transformations between them. The claim thus follows from Lemma \ref{lem:fiberwise_biadjointability}.
\end{proof}

\subsection{Proof of Theorem~\ref{thm:introlax}}
\label{subsec:Lax_Monoidal_Universality}

We now apply the result from the previous subsection to deduce the lax symmetric monoidal universal property of $\SpanTwoP{\myC}{E}{I}{P}^\otimes$.

\begin{definition}
Let $\twoC^\otimes$ and $\twoD^\otimes$ be symmetric monoidal 2-categories. A \emph{lax symmetric monoidal 2-functor} $\twoC^\otimes\to\twoD^\otimes$ is a functor $\Unco(\twoC^\otimes)\to \Unco(\twoD^\otimes)$ over $\Span(\Fin)$ which preserves 2-cocartesian morphisms over backwards spans. We write
	\[
	\FUN^{\text{lax-}\otimes}(\twoC^\otimes,\twoD^\otimes) \quad \subset \quad \HOM_{(\CAT_2)_{/\Span(\Fin)}}\big(\Unco(\twoC^\otimes), \Unco(\twoD^\otimes)\big)
	\]
	for the full sub-2-category spanned by the lax symmetric monoidal 2-functors. We refer to the 1-morphisms of this 2-category as \emph{natural transformations of lax symmetric monoidal 2-functors}.
\end{definition}

\begin{definition}\label{def:lax-badj}
	Let $(C,I,P)^{\otimes}\in\CMON(\BAdjTrip)$ and let $\twoD^\otimes$ be any symmetric monoidal 2-category. We say that a lax symmetric monoidal 2-functor $F\colon (C^\otimes)\catop \to \twoD^\otimes$ is \emph{lax symmetric monoidal $(I,P)$-biadjointable} if both of the following conditions are satisfied:
	\begin{enumerate}
		\item\label{it:fun_lax_sym_mon_(I,P)_biadj}The underlying functor $C\catop\to\twoD$ is $(I,P)$-biadjointable.
		\item\label{it:tran_lax_sym_mon_(I,P)_biadj} The transformation $F(-)\otimes F(-)\to F(-\otimes-)$ is $(I \times I, P \times P)$-biadjointable. Explicitly:
		\begin{enumerate}
			\item For every $Y\in C$ and every $i\colon X\to X'$ in $I$ the naturality square
			\begin{equation}\label{diag:lax-structure_IL}
				\begin{tikzcd}
					F(X)\otimes F(Y)\arrow[d]\arrow[r, "F(i)\otimes F(Y)"] &[2em] F(X')\otimes F(Y)\arrow[d]\\
					F(X\otimes Y)\arrow[r, "F(i\otimes Y)"'] & F(X'\otimes Y)
				\end{tikzcd}
			\end{equation}
			for the lax structure is horizontally left adjointable.
			\item Dually, for every $X$ and every $p\colon Y\to Y'$ in $P$ the square
			\[
			\begin{tikzcd}
				F(X)\otimes F(Y)\arrow[d, "F(X)\otimes p"']\arrow[r] &[2em] F(X\otimes Y)\arrow[d, "F(X\otimes p)"]\\
				F(X)\otimes F(Y')\arrow[r] & F(X\otimes Y')
			\end{tikzcd}
			\]
			is vertically right adjointable.
		\end{enumerate}
	\end{enumerate}
	Note that the existence of left adjoints for the horizontal maps in $(\ref{diag:lax-structure_IL})$ is already guaranteed by the first assumption: for the bottom map this follows from $i\otimes Y$ being contained in $I$ again by assumption, while for the top map this uses that the tensor product of $\twoD$ is a 2-functor. Similarly, the vertical maps in the square from Condition~(2b) are already guaranteed to admit right adjoints.
\end{definition}

\begin{theorem}\label{thm:lax-univ-prop}
	Let $C^{\otimes}$ be a symmetric monoidal category, let $I,P\subset E\subset C$ be wide subcategories such that $(I,P)$ is a suitable decomposition of $(C,E)$ in the sense of Convention~\ref{conv:most-general}, assume that $I$ and $P$ are closed under tensor products, and assume that $\otimes \colon C \times C \to C$ preserves pullbacks along morphisms in $I \times I$ as well as pullbacks along morphisms in $P \times P$. Then the restriction
	\[
		h^*\colon \FUN^\textup{lax-$\otimes$}\big(\SpanTwoP{C}{E}{I}{P}^\otimes,\twoD^\otimes\big)\to\FUN^\textup{lax-$\otimes$}\big((C^\otimes)^\op,\twoD^\otimes\big)
	\]
	is a locally full inclusion for any symmetric monoidal 2-category $\twoD^\otimes$. Moreover, the objects in its image are precisely the lax symmetric monoidal $(I,P)$-biadjointable functors in the sense of Definition~\ref{def:lax-badj}, and its 1-morphisms are precisely those natural transformations of lax symmetric monoidal 2-functors whose underlying natural transformations of functors $C^\op\rightrightarrows\twoD$ are $(I,P)$-biadjointable.
\end{theorem}
\begin{proof}
	By Corollary~\ref{cor:lax_tran_out_of_span}, there is a locally full inclusion
	\[\hskip-45.5pt\hfuzz=45.5pt
	\HOM_{(\CAT_2)_{/\Span(\Fin)}}\big(\Un^{\cc}(\SpanTwoP{C}{E}{I}{P}^\otimes), \Un^{\cc}(\twoD^{\otimes})\big)\hookrightarrow \HOM_{(\CAT_2)_{/\Span(\Fin)}}\big(\Un^{\cc}((C^{\otimes})^{\op}), \Un^{\cc}(\twoD^{\otimes})\big).
	\]
	The characterization of the 2-cocartesian morphisms in $\SpanTwoP{\widetilde{C}}{\widetilde{E}}{\widetilde{I}}{\widetilde{P}}$ recorded in Lemma~\ref{lem:Span2-cocart-fib} shows that this restricts immediately to lax symmetric monoidal functors. Therefore it suffices to prove that a lax symmetric monoidal 2-functor $F\colon (C^{\otimes})^{\op}\to \twoD^\otimes$ is lax symmetric monoidal $(I,P)$-biadjointable if and only if it satisfies \eqref{it:biad_tran} in Corollary~\ref{cor:lax_tran_out_of_span}, and similarly for 1-morphisms. Since $F(\textbf{n})$ is an $\textbf{n}$-fold product of $F(\textbf{1})$, it is clearly $(I^n,P^n)$-biadjointable if and only if $F(\textbf{1})$ is $(I,P)$-biadjointable. So Condition~\eqref{it:biad_tran_one} is equivalent to Condition~\eqref{it:fun_lax_sym_mon_(I,P)_biadj} in the definition of lax symmetric monoidal $(I,P)$-biadjointability. Now observe that Condition~\eqref{it:tran_lax_sym_mon_(I,P)_biadj} corresponds to Condition~\eqref{it:biad_tran_two} for the special case of the unique active morphism $\textbf{2}\to \textbf{1}$. However this clearly implies the same statement for an arbitrary forward span, since we can decompose it as a composite of equivalences and maps $\textbf{2}\amalg \textbf{n} \to \textbf{1}\amalg\textbf{n}$. Moreover, by assumption $F$ preserves cocartesian edges over backwards spans, meaning that the associated natural transformation is a natural equivalence and so automatically $(I,P)$-biadjointable. Finally any span composes as a composite of a backwards span followed by a forwards span, and so we conclude that Condition~\eqref{it:biad_tran_two} holds for every morphism in $\Span(\Fin)$. This shows that Condition~\eqref{it:tran_lax_sym_mon_(I,P)_biadj} is equivalent to Condition~\eqref{it:biad_tran_two}. Similarly, the fact that both $(C^\otimes)\catop$ and $\twoD^\otimes$ are product preserving shows that the condition on 1-morphisms agrees.
\end{proof}

\begin{remark}
The arguments of this section work just as well for lax $\mathcal{O}$-monoidal functors for any operad $\mathcal{O} \to \Span(\Fin)$, by making the following replacements in Definition \ref{def:lax-badj}:
\begin{enumerate}
	\item Condition~\eqref{it:fun_lax_sym_mon_(I,P)_biadj} is replaced by the analogous condition indexed over all colors of the operad and 
	\item Condition~\eqref{it:tran_lax_sym_mon_(I,P)_biadj} is replaced with the analogous condition indexed over a generating set of active maps of $\mathcal{O}$.
\end{enumerate} 
\end{remark}

\begin{proposition}\label{prop:sym-mon-univ-prop}
	Let $(C^{\otimes},I^{\otimes},P^{\otimes})$ be as above, with $I,P\subset E\subset C$ as in Theorem~\ref{thm:lax-univ-prop}. Then
	\[
h\colon(C^\otimes)^\op\to\SpanTwoP{C}{E}{I}{P}^\otimes
	\]
	is the initial strong symmetric monoidal functor whose underlying functor is $(I,P)$-biadjointable.
	\begin{proof}
		Whether a lax symmetric monoidal 2-functor is strong symmetric monoidal can be checked objectwise, so since $h$ is essentially surjective, the locally full inclusion of Theorem \ref{thm:lax-univ-prop} restricts to a locally full inclusion
		\[
		\FUN^{\otimes}\big(\SpanTwoP{C}{E}{I}{P}^\otimes,\twoD^\otimes\big) \hookrightarrow \FUN^{\otimes}\big((C\catop)^\otimes,\twoD^\otimes\big),
		\]
		with image given by the symmetric monoidal $(I,P)$-biadjointable functors and transformations.
	\end{proof}
\end{proposition}

\subsection{Constructing 6-functor formalisms}
\label{sec:6FF}

We will now specialize the previous result to the situation relevant for 6-functor formalisms. Here we assume that $C$ is equipped with the cartesian symmetric monoidal structure, so that $C^\op$ carries the \emph{cocartesian} symmetric monoidal structure. By \cite[Theorem~2.4.3.18]{lurie2016HA}, a lax symmetric monoidal functor $\mathcal D_0\colon (C^\times)^\op\to\twoD^\otimes$ can then be equivalently described as a functor $\mathcal D_0'\colon C^\op\to\CAlg(\iota_1\twoD^\otimes)$. Postcomposing $\mathcal D_0'$ with the forgetful functor to $\twoD$ just recovers the underlying functor of $\mathcal D_0$, so the first half of the biadjointability condition of Definition~\ref{def:lax-badj} is equivalent to demanding that the composite $C^\op\to\CAlg(\iota_1\twoD^\otimes)\to\twoD$ be $(I,P)$-biadjointable. Similarly one sees by unravelling definitions---also cf.~\cite[proof of Theorem~A.5.8(iv)]{Mann2022SixFunctor}---that the second half of Definition~\ref{def:lax-badj} translates to the \emph{projection formul\ae} for $\mathcal D_0'$, i.e.~demanding that the Beck--Chevalley maps
\[
i_{\sharp}(i^*(-)\otimes-)\to (-)\otimes i_{\sharp}(-)
\qquad\text{and}\qquad
(-)\otimes p_*(-) \to p_*(p^*(-)\otimes-)
\]
be invertible for every $i\in I$ and $p\in P$.

Thus, Theorem~\ref{thm:lax-univ-prop} specializes to the following 2-categorical refinement of \cite[Proposition~A.5.10]{Mann2022SixFunctor}, where we unpack most of our definitions for easier reference:

\begin{theorem}\label{thm:6FF}
	Let $C$ be a category with finite products, and let $I,P\subset E\subset C$ be wide subcategories closed under base change satisfying the following assumptions:
	\begin{enumerate}
		\item $I$ and $P$ are left cancellable.
		\item Every map in $E$ factors as a map in $I$ followed by a map in $P$.
		\item Every map in $I\cap P$ is truncated.
	\end{enumerate}
	Let $\mathcal D_0\colon C^\op\to\CAlg(\Cat)$ be a functor satisfying the following assumptions:
	\begin{enumerate}
		\item For every $i\colon a\to b$ in $I$ the functor $i^*\colon\mathcal D_0(b)\to\mathcal D_0(a)$ admits a left adjoint $i_\sharp$ satisfying the Beck--Chevalley condition with respect to pullbacks in $C$ and satisfying the projection formula $i_\sharp(-\otimes i^*(-))\simeq i_\sharp(-)\otimes(-)$.
		\item For every $p\colon a\to b$ in $P$ the functor $p^*\colon\mathcal D_0(b)\to\mathcal D_0(a)$ admits a right adjoint $p_*$ satisfying the dual Beck--Chevalley condition and the dual projection formula.
		\item For every pullback
		\[
		\begin{tikzcd}
			a\arrow[r, "i"]\arrow[d,"p"']\arrow[dr,pullback] & b\arrow[d, "q"]\\
			c\arrow[r, "j"'] & d
		\end{tikzcd}
		\]
		in $C$ with vertical maps in $P$ and horizontal maps in $I$, the double Beck--Chevalley map $j_\sharp p_*\to q_*i_\sharp$ is an equivalence.
	\end{enumerate}
	Then the lax symmetric monoidal functor $(C^\times)^\op\to\Cat^\times$ corresponding to $\mathcal D_0$ admits a \emph{unique} extension to a lax symmetric monoidal 2-functor
	\[
	\mathcal D\colon\SpanTwoP{C}{E}{I}{P}^\otimes\to\CAT^\times\!.
	\]
	Moreover, this extension has the following property: if $e\colon x\to y$ is any map in $E$ and $e=pi$ is any factorization into a map in $I$ followed by a map in $P$, then the exceptional pushforward $e_!\coloneqq\mathcal D(x\xleftarrow{\;\smash{\lower1pt\hbox{$\scriptstyle=$}}\;}x\xrightarrow{\;\smash{\lower1pt\hbox{$\scriptstyle e$}}\;}y)$ satisfies $e_!\simeq p_*i_\sharp$.\qed
\end{theorem}

\section{Lax transformations of 6-functor formalisms}\label{sec:lax_transformations_6FF}

In this final section, we prove analogues of the preceding universal properties for \emph{biadjointable (op)lax transformations}, to be defined below. As before, we first treat the non-monoidal situation, then pass to the lax symmetric monoidal setting, and finally specialize to lax and oplax natural transformations of 6-functor formalisms.

The concept of lax and oplax transformations requires the Gray tensor product of 2-categories, which we denote by
\[
-\boxtimes - \colon \Cat_2\times \Cat_2\to \Cat_2, \quad (\twoC,\twoD)\mapsto \twoC\boxtimes \twoD.
\]
For our definition of the Gray tensor product, we follow Gagna--Harpaz--Lanari \cite{Gagna_Harpaz_Lanari_Gray}. 

\begin{definition}\label{def:FunLax}
We define $\FUN(\twoE,-)^\lax$ and $\FUN(\twoE,-)^{\oplax}$ to be the right adjoints of the functors 
\[
-\boxtimes \twoE \colon \Cat_2 \to \Cat_2\qquad \text{and}\qquad \twoE\boxtimes -\colon \Cat_2\to \Cat_2
\]
respectively.
\end{definition}

\begin{remark}
	We note that ${\ast\boxtimes \twoE }\simeq {\twoE\boxtimes \ast} \simeq \twoE$, and so the objects of $\FUN(\twoE,\twoC)^{\textup{(op)lax}}$ are given by functors $\twoE \to \twoC$. We call the morphisms $\theta\colon F\to G$ of $\FUN(\twoE,\twoC)^{\textup{lax}}$ \emph{lax natural transformations}. By definition they are given by functors
	\[
	\theta\colon [1]\boxtimes \twoE \to \twoC.
	\]
	Given an object $X\in \twoE$, we can restrict $\theta$ along $\{X\}\colon \ast \to \twoE$ to obtain a 1-morphism $\eta_X\colon [1] \to \twoC$, which we call the \emph{component of $\theta$ at $X$}. The 2-category $[1]\boxtimes [1]$ agrees with the 2-category
\[\begin{tikzcd}
	\cdot & \cdot \\
	\cdot & \cdot,
	\arrow[from=1-1, to=1-2]
	\arrow[from=1-1, to=2-1]
	\arrow[Rightarrow, from=2-1, to=1-2]
	\arrow[from=1-2, to=2-2]
	\arrow[from=2-1, to=2-2]
\end{tikzcd}\]
 where we continue our convention of denoting the first component vertically. Given a morphism $f\colon X\to Y$ in $\twoE$, we will denote the image of the restriction of $\theta$ along $[1]\boxtimes\{f\} \colon [1]\boxtimes [1]\to [1]\boxtimes \twoE$ as follows: 
\[\begin{tikzcd}
	{F(X)} & {F(Y)} \\
	{G(X)} & {G(Y).}
	\arrow["{{F(f)}}", from=1-1, to=1-2]
	\arrow["{{\theta_X}}"', from=1-1, to=2-1]
	\arrow["{{\theta_Y}}", from=1-2, to=2-2]
	\arrow["{{\theta_f}}", Rightarrow, from=2-1, to=1-2]
	\arrow["{{G(f)}}"', from=2-1, to=2-2]
\end{tikzcd}\]
 Since the Gray tensor product preserves colimits in each variable, if $\twoE$ is a 1-category one can intuitively think of a lax natural transformation as a coherent collection of squares above. If $\twoE$ is a 2-category then a lax natural transformation additionally contains coherent data witnessing the relation $\theta_f F(\varphi) \simeq G(\varphi)\theta_g$ for every 2-morphism $\varphi\colon f\Rightarrow g$ in $\twoE$.
 
 We can make the same observations and definitions for the 1-morphisms $\theta\colon {[1]\boxtimes \twoE}\to \twoC$ of $\FUN^{\oplax}(\twoE,\twoC)$, which we call \emph{oplax natural transformations}. In this case the component $\theta_f$ is a 2-morphism $\theta_Y F(f)\Rightarrow G(f)\theta_X$. 
\end{remark}

\begin{remark}\label{def:Fun_lax}
Let us write $\NAT(F,G)^{\lax}$ for the Hom-category in $\FUN(J,\CAT_2)^{\lax}$ between functors $F,G\colon J\to \CAT_2$. By \cite[Theorem~A]{Gray-vs-cc}, $\NAT(F,G)^{\lax}$ is equivalent to the underlying 1-category of $\HOM_{(\CAT_2)_{/J}}\big(\Un^{\cc}(F), \Un^{\cc}(G)\big)$. Moreover, unravelling the proof of the cited result shows that the 2-morphisms defined in Notation~\ref{not:fibr_lax_tran} agree with the identically named 2-morphisms above. 

In particular let us note that, under this identification, Corollary~\ref{cor:lax_tran_out_of_span} implies that restriction along the natural transformation $C(-)\catop\to \SpanTwo(\Cc(-))$ induces a faithful inclusion
\[
\mathrm{Nat}(\SpanTwo(\Cc(-)),\twoD(-))^{\lax} \hookrightarrow \mathrm{Nat}(C(-)^{\op},\twoD(-))^{\lax}
\]
of 1-categories whose image on objects consists of those lax natural transformations $\phi\colon C(-)\catop\Rightarrow \twoD(-)$ satisfying conditions~\eqref{it:biad_tran_one} and~\eqref{it:biad_tran_two}, and whose image on morphisms consists of those modifications satisfying condition~\eqref{it:biadj_mod}.
\end{remark}

\begin{definition}
We will call an (op)lax natural transformation $\theta$ \emph{strong} if the transformation $\theta_f$ is invertible for all 1-morphisms $f\colon X\to Y$ in $\twoE$. By \cite[Corollary 2.8.12]{Gray-vs-cc} a strong (op)lax natural transformation is equivalent to an ordinary natural transformation.
\end{definition}

\begin{definition}
Consider an oplax square $\sigma\colon [1]\boxtimes [1]\to \twoC$ in a 2-category $\twoC$:
\[\begin{tikzcd}
	{X'} & {Y'} \\
	X & Y
	\arrow["{{f^*}}", from=1-1, to=1-2]
	\arrow["{{g^*}}"', from=1-1, to=2-1]
	\arrow["{{h^*}}", from=1-2, to=2-2]
	\arrow["\theta", Rightarrow, from=2-1, to=1-2]
	\arrow["{{k^*}}"', from=2-1, to=2-2]
\end{tikzcd}\]
We say $\sigma$ is (vertically) \emph{left adjointable} if the functors $g^*$ and $h^*$ admit left adjoints $g_\sharp$ and $h_\sharp$ and the Beck--Chevalley 2-morphism 
\[
h_\sharp k^* \xrightarrow{\eta_g} h_\sharp k^*  g^* g_\sharp \xrightarrow{\theta} h_\sharp h^*f^* g_\sharp \xrightarrow{\epsilon_h} f^* g_\sharp
\]
is an equivalence. Dually, we can define the notion of (horizontal) \emph{right adjointability}.
\end{definition}

\begin{remark} 
	When $\theta$ is an equivalence, the definitions above clearly agree with the notions of vertical left adjointability and horizontal right adjointability, as defined in Definition~\ref{def:adj_squares}, after passing through the identification of strong oplax squares and commutative squares. Note that, in contrast to Definition~\ref{def:adj_squares}, it is not possible to define vertical right adjointability or horizontal left adjointability of oplax squares.
\end{remark}

\begin{theorem}\label{thm:lax_nat_trans_of_4FF}
	Let $I,P\subset E\subset C$ be as in Convention~\ref{conv:most-general}. Then for every 2-category $\twoD$ the inclusion $h\colon C^\op\hookrightarrow\SpanTwoP{C}{E}{I}{P}$ induces a locally full inclusion of 2-categories
	\[
	\FUN(\SpanTwoP{C}{E}{I}{P},\twoD)^{\lax}\hookrightarrow \FUN(C^\op,\twoD)^{\lax}
	\]
	with essential image given on 
	\begin{enumerate}
		\item objects by the $(I,P)$-biadjointable functors,
		\item and on morphisms by those lax natural transformations $\theta\colon D\to D'$ of $(I,P)$-biadjointable functors such that
		\begin{enumerate}
			\item[(a)] $\theta$ restricted to $I$ is strong,
			\item[(b)] and for all $p\colon x\to y$ in $P$ the oplax square
			\[\begin{tikzcd}
				{D(y)} & {D(x)} \\
				{D'(y)} & {D'(x)}
				\arrow["{{p^*}}", from=1-1, to=1-2]
				\arrow["{{\theta_y}}"', from=1-1, to=2-1]
				\arrow["{{\theta_x}}", from=1-2, to=2-2]
				\arrow[Rightarrow, from=2-1, to=1-2]
				\arrow["{{p^*}}"', from=2-1, to=2-2]
			\end{tikzcd}\]
			is right adjointable.
		\end{enumerate}
	\end{enumerate}
	Dually, we obtain a locally full inclusion
	\[
	\FUN(\SpanTwoP{C}{E}{I}{P},\twoD)^{\oplax}\hookrightarrow \FUN(C^\op,\twoD)^{\oplax}
	\]
	with essential image given on 
\begin{enumerate}
	\item objects by the $(I,P)$-biadjointable functors,
	\item and on morphisms by those oplax natural transformations $\theta\colon D\to D'$ of $(I,P)$-biadjointable functors such that
		\begin{enumerate}
		\item $\theta$ restricted to $P$ is strong,
		\item and for all $i\colon x\to y$ in $I$, the oplax square
		\[\begin{tikzcd}
			{D(y)} & {D'(y)} \\
			{D(x)} & {D'(x)}
			\arrow["{{\theta_y}}", from=1-1, to=1-2]
			\arrow["{{i^*}}"', from=1-1, to=2-1]
			\arrow["{{i^*}}", from=1-2, to=2-2]
			\arrow[Rightarrow, from=2-1, to=1-2]
			\arrow["{{\theta_x}}"', from=2-1, to=2-2]
		\end{tikzcd}\]
		is left adjointable.
		\end{enumerate}
\end{enumerate}
\end{theorem}

\begin{remark}
	We say that a (op)lax natural natural transformation is \emph{$(I,P)$-biadjointable} if the corresponding condition (2) is satisfied.
\end{remark}

Before we prove our result we recall the following characterization of adjoints in (op)lax functor categories.

\begin{proposition}[{\cite[Corollary 5.2.10]{Gray-vs-cc}}]\label{prop:adj_in_oplax}
	Let $\twoE$ and $\twoC$ be 2-categories and let $\theta\colon F\Rightarrow G$ be a 1-morphism in $\FUN(\twoE,\twoC)^{\lax}$. Then: 
	\begin{enumerate}
		\item $\theta$ is a left adjoint if and only if for all 1-morphisms $f\colon X\to Y$ in $\twoE$, the oplax square
		\[\begin{tikzcd}
			{F(X)} & {F(Y)} \\
			{G(X)} & {G(Y).}
			\arrow["{{F(f)}}", from=1-1, to=1-2]
			\arrow["{{\theta_X}}"', from=1-1, to=2-1]
			\arrow["{{\theta_Y}}", from=1-2, to=2-2]
			\arrow["{{\theta_f}}", Rightarrow, from=2-1, to=1-2]
			\arrow["{{G(f)}}"', from=2-1, to=2-2]
		\end{tikzcd}\]
		is left adjointable.
		\item $\theta$ is a right adjoint if and only if it is strong and $\theta_X$ admits a right adjoint for all $X\in \twoE$.
		\item A commutative square 
		\[\begin{tikzcd}
			F & G \\
			{F'} & {G'}
			\arrow["\theta", from=1-1, to=1-2]
			\arrow["\phi"', from=1-1, to=2-1]
			\arrow["\psi", from=1-2, to=2-2]
			\arrow["{\theta'}"', from=2-1, to=2-2]
		\end{tikzcd}\]
		in $\FUN(\twoE,\twoC)^{\lax}$ is horizontally left/right adjointable if and only if $\theta$ and $\theta'$ are left/right adjoints and for each $X\in \twoE$, the square 
		\[\begin{tikzcd}
			{F(X)} & {G(X)} \\
			{F'(X)} & {G'(X)}
			\arrow["{\theta_X}", from=1-1, to=1-2]
			\arrow["{\phi_X}"', from=1-1, to=2-1]
			\arrow["{\psi_X}", from=1-2, to=2-2]
			\arrow["{\theta'_X}"', from=2-1, to=2-2]
		\end{tikzcd}\]
		in $\twoC$ is horizontally left/right adjointable.
	\end{enumerate}
	Moreover, the dual statements hold for $\FUN(\twoE,\twoC)^{\oplax}$.
\end{proposition}

\begin{proof}
	Points (1) and (2) are precisely \cite[Corollary 5.2.10]{Gray-vs-cc}, while the final point follow from the observation that the functors $\ev_X\colon \FUN(\twoE,\twoC)^{\textup{(op)lax}} \to \twoC$ for $X\in \twoE$ jointly detect invertibility of $2$-morphisms, see \cite[Proposition 2.4.3(1)]{Gray-vs-cc}. 
\end{proof}

\begin{proof}[Proof of Theorem~\ref{thm:lax_nat_trans_of_4FF}]
	We will prove the second statement; the first one is dual. Let $\twoE$ be another 2-category. Note that it suffices to prove that the top map in the following commutative diagram
	\[\begin{tikzcd}
		{\Hom_{\CatTwo}(\twoE,\FUN(\SpanTwoP{C}{E}{I}{P},\twoD)^{\oplax}}) & {\Hom_{\CatTwo}(\twoE,\FUN(C^\op,\twoD)^{\oplax})} \\
		{\Hom_{\CatTwo}(\SpanTwoP{C}{E}{I}{P}\boxtimes \twoE,\twoD)} & {\Hom_{\CatTwo}(C^{\op}\boxtimes \twoE,\twoD)} \\
		{\Hom_{\CatTwo}(\SpanTwoP{C}{E}{I}{P},\FUN(\twoE,\twoD)^{\lax})} & {\Hom_{\CatTwo}(C^{\op},\FUN(\twoE,\twoD)^{\lax})}
		\arrow["{h^*}", from=1-1, to=1-2]
		\arrow["\sim"', from=1-1, to=2-1]
		\arrow["\sim", from=1-2, to=2-2]
		\arrow["{(\twoE\boxtimes h)^*}", from=2-1, to=2-2]
		\arrow["\sim"', from=2-1, to=3-1]
		\arrow["\sim", from=2-2, to=3-2]
		\arrow["{h^*}", from=3-1, to=3-2]
	\end{tikzcd}\]
	is an inclusion of path-components, with essential image given by those functors $F\colon \twoE \to \FUN(C\catop,\twoD)^{\oplax}$ such that the restriction of $F$ to all objects and 1-morphisms of $\twoE$ is $(I,P)$-biadjointable.
	
	By Theorem~\ref{thm:main-theorem}, we know the bottom map is an inclusion of path-components, and that its image consists of the $(I,P)$-biadjointable functors $C^{\op}\to \FUN(\twoE,\twoD)^{\lax}$. As a result of Proposition~\ref{prop:adj_in_oplax}, we may identify this with the subspace of functors  $F\colon  C^{\op}\boxtimes \twoE \to \twoD$ such that the following conditions are satisfied:
	\begin{enumerate}
		\item For every object $e\in \twoE$, the functor 
		\[
			F(-,e)\colon C^{\op}\to \twoD
		\]
		is $(I,P)$-biadjointable.
		\item For every pair of a morphism $i\colon x\to y$ in $I$ and $h\colon e\to e'$ in $\twoE$, the oplax square 
		\[\begin{tikzcd}
			{F(y,e)} & {F(y,e')} \\
			{F(x,e)} & {F(x,e')}
			\arrow["{{F(y,h)}}", from=1-1, to=1-2]
			\arrow["{{F(i,e)}}"', from=1-1, to=2-1]
			\arrow["{{F(i,e')}}", from=1-2, to=2-2]
			\arrow[Rightarrow, from=2-1, to=1-2]
			\arrow["{{F(x,h)}}"', from=2-1, to=2-2]
		\end{tikzcd}\]
		is left adjointable.
		\item For every pair of a morphism $p\colon x\to y$ in $P$ and $h\colon e\to e'$ in $\twoE$, the oplax square 
		\[\begin{tikzcd}
			{F(y,e)} & {F(y,e')} \\
			{F(x,e)} & {F(x,e')}
			\arrow["{{F(y,h)}}", from=1-1, to=1-2]
			\arrow["{{F(p,e)}}"', from=1-1, to=2-1]
			\arrow["{{F(p,e')}}", from=1-2, to=2-2]
			\arrow[Rightarrow, from=2-1, to=1-2]
			\arrow["{{F(x,h)}}"', from=2-1, to=2-2]
		\end{tikzcd}\]
		is strong.
	\end{enumerate}	
	Translating this further to 2-functors $\widetilde{F}\colon\twoE \to \Fun(C\catop,\twoD)^{\oplax}$, condition~(1) says exactly that each object $\tilde{F}(e)$ is $(I,P)$-biadjointable, while conditions~(2) and~(3) identify the image of each morphism of $\twoE$ under $\widetilde{F}$ with an $(I,P)$-biadjointable lax transformation.
\end{proof}

As the final result, we provide a lax symmetric monoidal analogue of Theorem~\ref{thm:lax_nat_trans_of_4FF}. Since the functors $\FUN(\twoC,-)^\lax$ and $\FUN(\twoC,-)^{\oplax}$ are right adjoints, they preserve products. It follows that for a symmetric monoidal 2-category $\twoD^\otimes$, the 2-categories $\FUN(\twoC,\twoD^\otimes)^\lax$ and $\FUN(\twoC,\twoD^\otimes)^{\oplax}$ inherit canonical symmetric monoidal structures. We refer to these as the \emph{pointwise} symmetric monoidal structures.

We make the following definition.

\begin{definition}
Let $\twoC^\otimes,\twoD^\otimes\in \CMon(\CatTwo)$. We define 2-categories \[\FUN^\textup{lax-$\otimes$}(\twoC^\otimes,\twoD^\otimes)^{\lax}\quad \text{and}\quad \FUN^\textup{lax-$\otimes$}(\twoC^\otimes,\twoD^\otimes)^{\oplax}\] by the universal property that for all $\twoE\in \CatTwo$,
\begin{align*}
	\Hom_{\Cat_2}(\twoE, \FUN^\textup{lax-$\otimes$}(\twoC^\otimes,\twoD^\otimes)^{\lax}) &\simeq \iota\FUN^\textup{lax-$\otimes$}(\twoC^\otimes, \FUN(\twoE,\twoD^\otimes)^{\oplax}),\\
	\Hom_{\Cat_2}(\twoE, \FUN^\textup{lax-$\otimes$}(\twoC^\otimes,\twoD^\otimes)^{\oplax}) &\simeq \iota\FUN^\textup{lax-$\otimes$}(\twoC^\otimes, \FUN(\twoE,\twoD^\otimes)^{\lax}).
\end{align*}
\end{definition}

To see that this is well-defined, i.e.\ that the right hand sides are indeed representable as functors of $\mathbb E$, it will suffice by \cite[Proposition 5.5.2.2]{lurie2009HTT} to show:

\begin{lemma}
The functors
\[
 \FUN^\textup{lax-$\otimes$}(\twoC^\otimes, \FUN(-,\twoD^\otimes)^{\oplax}) \quad\text{and}\quad \FUN^\textup{lax-$\otimes$}(\twoC^\otimes, \FUN(-,\twoD^\otimes)^{\lax})
\]	
are limit-preserving functors $\Cat_2^{\op}\to \Cat$.
\end{lemma}

\begin{proof}
	Consider the locally full subcategory $(\CAT_2)_{/\Span(\Fin)}^{\textup{bw-cc}}$ of $(\CAT_2)_{/\Span(\Fin)}$ whose objects are the 2-cocartesian fibrations and whose morphisms are those 2-functors over $\Span(\Fin)$ which preserve 2-cocartesian edges over backwards maps. By definition, $\FUN^\textup{lax-$\otimes$}(\twoC^{\otimes},-)$ denotes the hom category in this 2-category, and in particular preserves limits. It remains to show that the composite functors
	\[
		\FUN(-,\twoD^\otimes)^{\oplax}, \FUN(-,\twoD^\otimes)^{\lax}\colon \Cat_2^\op \to \CMon(\Cat_2) \hookrightarrow (\Cat_2)_{/\Span(\Fin)}^{\textup{bw-cc}}
	\] 
	preserve limits. Here the inclusion is given by the composite
	\[
		\CMon(\Cat_2) \subset \Fun(\Span(\Fin),\Cat_2) \simeq (\Cat_2)_{/\Span(\Fin)}^{\textup{cc}} \hookrightarrow (\Cat_2)_{/\Span(\Fin)}^{\textup{bw-cc}}.
	\]
	The first inclusion clearly preserves limits, and the final inclusion preserves limits as it admits a left adjoint by \cite[Theorem 5.6.5]{AHM_FreeFibrations}. It thus remains to show that $\FUN(-,\twoD^\otimes)^{\oplax}$ and $\FUN(-,\twoD^\otimes)^{\lax}$ preserve limits when regarded as functors from $\Cat_2\catop$ to $\Fun(\Span(\Fin),\Cat_2)$. We prove the claim in the lax case; the other case is dual. Since limits in the latter category are computed pointwise, we need to show that for each $\textbf{n}$ the functor $\FUN(-, \twoD^{\times n})^{\lax}\colon \Cat_2\catop \to \Cat_2$ preserves limits. But this functor admits a left adjoint given by $\FUN(-,\twoD^{\times n})^{\op\lax}\colon \Cat_2 \to \Cat_2\catop$, finishing the proof.
\end{proof}

\begin{remark}
	The objects of $\FUN^\textup{lax-$\otimes$}(\twoC^\otimes,\twoD^\otimes)^{\textup{(op)lax}}$ are precisely lax symmetric monoidal 2-functors $\twoC^\otimes\to \twoD^\otimes$. We call the 1-morphisms \emph{(op)lax transformations of lax symmetric monoidal 2-functors}. Note that forgetting the lax symmetric monoidal structures gives a functor 
	\[
	\FUN^\textup{lax-$\otimes$}(\twoC^\otimes,\twoD^\otimes)^{\textup{(op)lax}} \to \FUN(\twoC,\twoD)^{\textup{(op)lax}}.
	\]
\end{remark}

\begin{theorem}\label{thm:lax_nat_trans_of_6FF}
	Let $(C,I,P)^\otimes \in \CMON(\BAdjTrip)$ and let $\twoD^\otimes$ be a symmetric monoidal 2-category. Then restriction induces a locally full inclusion
	\[
	\FUN^\textup{lax-$\otimes$}(\SpanTwoP{C}{E}{I}{P}^\otimes,\twoD^\otimes)^{\textup{(op)lax}}\hookrightarrow \FUN^\textup{lax-$\otimes$}((C^\otimes)^\op,\twoD^\otimes)^{\textup{(op)lax}}
	\]
	of 2-categories with essential image given by the lax symmetric monoidal $(I,P)$-biadjointable functors and (op)lax natural transformations thereof which are underlying $(I,P)$-biadjointable.
\end{theorem}

\begin{proof}
	We will prove the lax case of the result; the oplax case is dual. It suffices to prove that for every 2-category $\twoE$, the top map in the following commutative diagram
	\[\hskip-18.31pt\hfuzz=18.31pt\begin{tikzcd}[cramped]
		{\Hom_{\CatTwo}(\twoE,\FUN^\textup{lax-$\otimes$}(\SpanTwoP{C}{E}{I}{P}^\otimes,\twoD^\otimes)^{\lax})} & {\Hom_{\CatTwo}(\twoE,\FUN^\textup{lax-$\otimes$}((C^\otimes)^\op,\twoD^\otimes)^{\lax})} \\
		{\iota\FUN^\textup{lax-$\otimes$}(\SpanTwoP{C}{E}{I}{P}^\otimes,\FUN(\twoE,\twoD^\otimes)^{\oplax})} & {\iota\FUN^\textup{lax-$\otimes$}((C^\otimes)^{\op},\FUN(\twoE,\twoD^\otimes)^{\oplax})}
		\arrow["{h^*}", from=1-1, to=1-2]
		\arrow["\sim"', from=1-1, to=2-1]
		\arrow["\sim", from=1-2, to=2-2]
		\arrow["{h^*}", from=2-1, to=2-2]
	\end{tikzcd}\]
	is an inclusion of path-components with essential image given by those functors $F\colon \twoE\to \FUN^\textup{lax-$\otimes$}((C^\otimes)^\op,\twoD^\otimes)^{\lax})$ such that for all objects of $\twoE$ the restriction of $F$ is lax symmetric monoidal $(I,P)$-biadjointable and for all morphisms of $\twoE$ the lax natural transformation is underlying $(I,P)$-biadjointable. By Theorem~\ref{thm:lax-univ-prop}, the bottom map is an inclusion of path-components, and its image consists of the lax symmetric monoidal $(I,P)$-biadjointable functors $F\colon (C^\otimes)^{\op}\to \FUN(\twoE,\twoD^\otimes)^{\oplax}$. By definition this means that:
	\begin{enumerate}
		\item The underlying functor of $F$ is $(I,P)$-biadjointable, and
		\item The transformation $F(-)\otimes F(-)\to F(- \otimes -)$ is $(I\times I,P\times P)$-biadjointable.
	\end{enumerate}
	Denote by $\tilde{F}\colon \twoE\to \FUN^\textup{lax-$\otimes$}((C^\otimes)^{\op},\twoD^\otimes)^\lax$ the functor associated to $F$ under the right-hand equivalence in the diagram. By Proposition~\ref{prop:adj_in_oplax}, condition $(1)$ is equivalent to:
	\begin{enumerate}
		\item[(1a)] For each $e \in \twoE$, the lax symmetric monoidal functor $\tilde{F}(e)\colon (C^{\otimes})\catop \to \twoD^{\otimes}$ is $(I,P)$-biadjointable, and
		\item[(1b)] For each 1-morphism $f\colon e\to e'$ of $\twoE$, the underlying lax natural transformation $\tilde{F}(f)\colon \tilde{F}(e) \to \tilde{F}(e')$ is $(I,P)$-biadjointable.
	\end{enumerate}
	Since the collection of functors $\ev_x\colon \FUN(\twoE,\twoD^\otimes)^{\textup{oplax}} \to \twoD^\otimes$ for $x\in \twoE$ are strong symmetric monoidal and jointly conservative, we see that $(2)$ and $(1\textup{a})$ together are equivalent to the condition that $\tilde{F}(e)$ is lax symmetric monoidal $(I,P)$-biadjointable for all $e\in \twoE$, while $(\text{1b})$ is precisely the required condition on morphisms.
\end{proof}

\begin{remark}
	When $C$ is equipped with the cartesian symmetric monoidal structure, Theorem~\ref{thm:lax_nat_trans_of_6FF} specializes to the situation relevant for 6-functor formalisms. In particular, it recovers the corresponding lax and oplax natural transformations between the induced 6-functor formalisms discussed in the introduction.
\end{remark}

\bibliographystyle{amsalpha}
\bibliography{Bibliography.bib}
\end{document}

%% file: preamble.tex
\newif\ifdutch
\dutchfalse
\ifdutch
\let\bm=\relax
\let\LambdaRune=\Lambda

\else
\usepackage[bbgreekl]{mathbbol}
\DeclareMathSymbol{\LambdaRune}{\mathord}{bbold}{"03}
\usepackage[bb=stixtwo]{mathalpha}
\usepackage{microtype}
\fi
\usepackage{bm}

\usepackage{amsmath, amssymb, amsthm}
\usepackage{mathtools}
\usepackage{enumerate}

\usepackage[hang, symbol]{footmisc}
\usepackage{hyperref, url}
\hypersetup{bookmarksdepth=subsubsection}
\theoremstyle{plain}
\newtheorem{theorem}{Theorem}[section]
\newtheorem{corollary}[theorem]{Corollary}

\newtheorem{lemma}[theorem]{Lemma}
\newtheorem{proposition}[theorem]{Proposition}

\newtheorem{thmx}{Theorem}

\theoremstyle{definition}
\newtheorem{construction}[theorem]{Construction}
\newtheorem{convention}[theorem]{Convention}
\newtheorem{definition}[theorem]{Definition}
\newtheorem{example}[theorem]{Example}
\newtheorem{notation}[theorem]{Notation}

\newtheorem{remark}[theorem]{Remark}

\newcommand{\Hom}{{\textup{Hom}}}
\let\hom=\Hom
\DeclareMathOperator{\iHom}{\underline{\operatorname{Hom}}}
\newcommand{\Fun}{{\textup{Fun}}}

\DeclareMathOperator{\id}{id}
\DeclareMathOperator{\pr}{pr}

\DeclareMathOperator{\const}{const}

\let\lim\relax
\DeclareMathOperator{\lim}{lim}

\DeclareMathOperator{\ev}{ev}

\DeclareMathOperator{\res}{res}

\DeclareMathOperator{\Nm}{Nm}

\DeclareMathOperator{\CMon}{CMon}
\DeclareMathOperator{\CAlg}{CAlg}

\DeclareMathOperator{\PSh}{PSh}

\DeclareMathOperator{\Ar}{Ar}
\DeclareMathOperator{\Tw}{Tw}
\DeclareMathOperator{\PAr}{PAr}
\DeclareMathOperator{\Span}{Span}

\newcommand{\Un}{{\textup{Un}}}
\newcommand{\cc}{{\textup{cc}}}
\newcommand{\ct}{{\textup{ct}}}
\newcommand{\Unco}{\Un^{\cc}}
\newcommand{\Unct}{\Un^{\ct}}
\newcommand{\fw}{{\textup{fw}}}
\newcommand{\comp}{{\textup{comp}}}
\newcommand{\COMP}{{\textup{\textsc{comp}}}}
\newcommand{\core}[1]{\iota{#1}}

\newcommand{\lax}{{\textup{lax}}}
\newcommand{\oplax}{{\textup{oplax}}}
\newcommand{\SQ}{{\textup{SQ}}}
\newcommand{\co}{2\textup{-op}}
\newcommand{\fold}{\textup{fold}}
\newcommand{\all}{\textup{all}}
\newcommand{\op}{{\textup{op}}}
\newcommand{\catop}{^{\textup{op}}}

\newcommand{\SpanTwo}{{\textup{\textsc{Span}}}_2}
\newcommand{\SpanThreeHalves}{\textup{\textsc{Span}}_{1{\frac{1}{2}}}}
\newcommand{\SpanTwoP}[4]{\SpanTwo(#1,#2)_{#4,#3}}
\newcommand{\SpanDblP}[4]{\SpanDbl(#1,#2)_{#4,#3}}
\newcommand{\SpanDbl}{{\textup{SPAN}_\textup{dbl}}}
\newcommand{\BC}{{\textup{BC}}}
\newcommand{\CIP}[4]{\ul\Span_{#4\textup{-fold},#3}(#1^{#2\text-\amalg})}

\newcommand{\twoB}{{\mathbb B}}
\newcommand{\twoC}{{\mathbb C}} 
\newcommand{\twoD}{{\mathbb D}}
\newcommand{\twoE}{{\mathbb E}}

\newcommand{\twoI}{{\mathbb I}}
\newcommand{\Univ}[3]{{\mathbb B}_{#2,#3}(#1)}
\newcommand{\UnivGen}{{\mathbb B}}

\newcommand{\CAT}{{\textup{\textsc{Cat}}}}
\newcommand{\CatTwo}{{\CAT_2}}

\newcommand{\Adj}{{\textup{\textsc{Adj}}}}
\newcommand{\FUN}{\textup{\textsc{Fun}}} 
\newcommand{\NAT}{\textup{\textsc{Nat}}} 
\newcommand{\HOM}{\textup{\textsc{Hom}}} 
\newcommand{\CMON}{\textup{\textsc{CMon}}} 

\newcommand{\dbadj}{\textup{-badj}}
\newcommand{\BAdjTrip}{{\textup{\textsc{BAdjTrip}}}}
\newcommand{\AdTrip}{{\textup{AdTrip}}}
\newcommand{\myC}{{C}}
\newcommand{\myT}{{C}}

\newcommand{\catname}[1]{{\textup{#1}}}

\newcommand{\Spc}{\catname{Spc}}

\newcommand{\Fin}{\catname{Fin}}
\newcommand{\Cat}{\catname{Cat}}

\newcommand{\qquadtext}[1]{\qquad\textrm{#1}\qquad}

\newcommand{\ulhelper}[2]{\underline{\setbox0=\hbox{$#1#2$}\dp0=0.4pt \box0\relax}}
\newcommand{\ul}[1]{{\mathpalette\ulhelper{#1}}\hbox{\rule[-2pt]{0pt}{0pt}}} 
\newcommand{\iso}{\xrightarrow{\;\smash{\raisebox{-0.5ex}{\ensuremath{\scriptstyle\sim}}}\;}} 
\newcommand{\noloc}{
	\nobreak
	\mspace{6mu plus 1mu}
	{:}
	\nonscript\mkern-\thinmuskip
	\mathpunct{}
	\mspace{2mu}
}

\newcommand{\Cc}{\mathcal{C}}
\newcommand{\Dd}{\mathcal{D}}

\newcommand{\Ww}{\mathcal{W}}
\newcommand{\Xx}{\mathcal{X}}
\newcommand{\Yy}{\mathcal{Y}}
\newcommand{\Zz}{\mathcal{Z}}

\renewcommand{\phi}{\varphi}
\renewcommand{\epsilon}{\varepsilon}
\let\atled=\varrho 

\usepackage{tikz, tikz-cd}
\usetikzlibrary{matrix,arrows,decorations.pathmorphing, decorations.markings, cd, calc, nfold}

\tikzset{curve/.style={settings={#1},to path={(\tikztostart)
			.. controls ($(\tikztostart)!\pv{pos}!(\tikztotarget)!\pv{height}!270:(\tikztotarget)$)
			and ($(\tikztostart)!1-\pv{pos}!(\tikztotarget)!\pv{height}!270:(\tikztotarget)$)
			.. (\tikztotarget)\tikztonodes}},
	settings/.code={\tikzset{quiver/.cd,#1}
		\def\pv##1{\pgfkeysvalueof{/tikz/quiver/##1}}},
	quiver/.cd,pos/.initial=0.35,height/.initial=0}

\tikzset{mono/.style={>->},
	epic/.style={->>},
	both/.style={>->>},
	epmo/.style={->,
		postaction={decorate,
			decoration={markings,
				mark=at position 0pt with {
					\draw[-] (0,0)--(0,-3pt);
					\draw[-] (1.5pt,0)--(1.5pt,-2pt);
					\draw[-] (3pt,0)--(3pt,-3pt);
	}}}},
	epmo'/.style={->,
		postaction={decorate,
			decoration={markings,
				mark=at position 0pt with {
					\draw[-] (0,0)--(0,3pt);
					\draw[-] (1.5pt,0)--(1.5pt,2pt);
					\draw[-] (3pt,0)--(3pt,3pt);
}}}}}

\newcommand{\buildarrowfromtikz}[1]{\mathrel{\begin{tikzcd}[ampersand replacement=\&, cramped, cells={nodes={inner sep=0pt}}, column sep=small] \null\arrow[r, #1]\&\null
\end{tikzcd}}}
\newcommand{\rightarrowepic}{\buildarrowfromtikz{epic}}
\newcommand{\rightarrowmono}{\buildarrowfromtikz{mono}}
\newcommand{\rightarrowepmo}{\buildarrowfromtikz{epmo}} 
\newcommand{\buildarrowfromtikzl}[1]{\mathrel{\begin{tikzcd}[ampersand replacement=\&, cramped, cells={nodes={inner sep=0pt}}, column sep=small] \null\&\arrow[l, #1]\null
\end{tikzcd}}}

\newcommand{\leftarrowmono}{\buildarrowfromtikzl{mono}}

\tikzcdset{pullback/.style = {phantom, "\lrcorner"{very near start}}}
\tikzcdset{pushout/.style = {phantom, "{\rotatebox{180}{$\lrcorner$}}"{very near end}}}

\newcommand{\notehelper}[3]{\textcolor{#3}{$\blacksquare$}\marginpar{\ifodd\thepage\raggedright\else\raggedleft\fi\color{#3}\tiny \textbf{#2:} #1}}
